\documentclass[12pt]{article}

\usepackage{mathptmx} 
\DeclareSymbolFont{cmletters}{OML}{cmm}{m}{it}              
\DeclareSymbolFont{cmsymbols}{OMS}{cmsy}{m}{n}
\DeclareSymbolFont{cmlargesymbols}{OMX}{cmex}{m}{n}
\DeclareMathSymbol{\myjmath}{\mathord}{cmletters}{"7C}     \let\jmath\myjmath 
\DeclareMathSymbol{\myamalg}{\mathbin}{cmsymbols}{"71}     \let\amalg\myamalg
\DeclareMathSymbol{\mycoprod}{\mathop}{cmlargesymbols}{"60}\let\coprod\mycoprod
\DeclareMathSymbol{\myalpha}{\mathord}{cmletters}{"0B}     \let\alpha\myalpha 
\DeclareMathSymbol{\mybeta}{\mathord}{cmletters}{"0C}      \let\beta\mybeta
\DeclareMathSymbol{\mygamma}{\mathord}{cmletters}{"0D}     \let\gamma\mygamma
\DeclareMathSymbol{\mydelta}{\mathord}{cmletters}{"0E}     \let\delta\mydelta
\DeclareMathSymbol{\myepsilon}{\mathord}{cmletters}{"0F}   \let\epsilon\myepsilon
\DeclareMathSymbol{\myzeta}{\mathord}{cmletters}{"10}      \let\zeta\myzeta
\DeclareMathSymbol{\myeta}{\mathord}{cmletters}{"11}       \let\eta\myeta
\DeclareMathSymbol{\mytheta}{\mathord}{cmletters}{"12}     \let\theta\mytheta
\DeclareMathSymbol{\myiota}{\mathord}{cmletters}{"13}      \let\iota\myiota
\DeclareMathSymbol{\mykappa}{\mathord}{cmletters}{"14}     \let\kappa\mykappa
\DeclareMathSymbol{\mylambda}{\mathord}{cmletters}{"15}    \let\lambda\mylambda
\DeclareMathSymbol{\mymu}{\mathord}{cmletters}{"16}        \let\mu\mymu
\DeclareMathSymbol{\mynu}{\mathord}{cmletters}{"17}        \let\nu\mynu
\DeclareMathSymbol{\myxi}{\mathord}{cmletters}{"18}        \let\xi\myxi
\DeclareMathSymbol{\mypi}{\mathord}{cmletters}{"19}        \let\pi\mypi
\DeclareMathSymbol{\myrho}{\mathord}{cmletters}{"1A}       \let\rho\myrho
\DeclareMathSymbol{\mysigma}{\mathord}{cmletters}{"1B}     \let\sigma\mysigma
\DeclareMathSymbol{\mytau}{\mathord}{cmletters}{"1C}       \let\tau\mytau
\DeclareMathSymbol{\myupsilon}{\mathord}{cmletters}{"1D}   \let\upsilon\myupsilon
\DeclareMathSymbol{\myphi}{\mathord}{cmletters}{"1E}       \let\phi\myphi
\DeclareMathSymbol{\mychi}{\mathord}{cmletters}{"1F}       \let\chi\mychi
\DeclareMathSymbol{\mypsi}{\mathord}{cmletters}{"20}       \let\psi\mypsi
\DeclareMathSymbol{\myomega}{\mathord}{cmletters}{"21}     \let\omega\myomega
\DeclareMathSymbol{\myvarepsilon}{\mathord}{cmletters}{"22}\let\varepsilon\myvarepsilon
\DeclareMathSymbol{\myvartheta}{\mathord}{cmletters}{"23}  \let\vartheta\myvartheta
\DeclareMathSymbol{\myvarpi}{\mathord}{cmletters}{"24}     \let\varpi\myvarpi
\DeclareMathSymbol{\myvarrho}{\mathord}{cmletters}{"25}    \let\varrho\myvarrho
\DeclareMathSymbol{\myvarsigma}{\mathord}{cmletters}{"26}  \let\varsigma\myvarsigma
\DeclareMathSymbol{\myvarphi}{\mathord}{cmletters}{"27}    \let\varphi\myvarphi

\usepackage{amsthm,amsmath,amssymb,amscd,graphics,enumerate,latexsym,verbatim,stmaryrd,graphicx,color}
\usepackage[all]{xy}

\theoremstyle{plain}
\newtheorem{thm}{Theorem}[section]
\newtheorem{cor}[thm]{Corollary}
\newtheorem{lemma}[thm]{Lemma}
\newtheorem{prop}[thm]{Proposition}

\theoremstyle{definition}
\newtheorem{df}[thm]{Definition}

\newtheorem{rem}[thm]{Remark}
\newtheorem{ex}[thm]{Example}

\DeclareMathOperator{\Mat}{Mat}
\DeclareMathOperator{\GL}{GL}
\DeclareMathOperator{\SL}{SL}

\DeclareMathOperator{\PSL}{PSL}
\DeclareMathOperator{\OO}{O}
\DeclareMathOperator{\SO}{SO}

\DeclareMathOperator{\Sp}{Sp}

\DeclareMathOperator{\Spec}{Spec}

\DeclareMathOperator{\Sch}{Sch}

\DeclareMathOperator{\Hom}{Hom}
\DeclareMathOperator{\Aut}{Aut}

\DeclareMathOperator{\Proj}{Proj}

\DeclareMathOperator{\Rad}{Rad}
\DeclareMathOperator{\Nil}{Nil}
\DeclareMathOperator{\sign}{sign}

\DeclareMathOperator{\tensor}{\otimes}

\DeclareMathOperator{\diag}{{diag}}
\DeclareMathOperator{\mon}{{mon}}

\def\0{{\bf 0}}
\def\A{{\mathbb A}}
\def\B{{\mathbb B}}
\def\C{{\mathbb C}}
\def\F{{\mathbb F}}
\def\G{{\mathbb G}}

\def\N{{\mathbb N}}
\def\P{{\mathbb P}}
\def\Q{{\mathbb Q}}

\def\Z{{\mathbb Z}}

\def\cC{{\mathcal C}}
\def\cD{{\mathcal D}}

\def\cF{{\mathcal F}}
\def\cG{{\mathcal G}}
\def\cH{{\mathcal H}}
\def\cI{{\mathcal I}}

\def\cM{{\mathcal M}}

\def\cO{{\mathcal O}}
\def\cP{{\mathcal P}}

\def\cR{{\mathcal R}}

\def\cT{{\mathcal T}}
\def\cU{{\mathcal U}}

\def\cW{{\mathcal W}}
\def\cX{{\mathcal X}}

\def\cZ{{\mathcal Z}}

\def\fe{{\mathfrak e}}
\def\fg{{\mathfrak g}}
\def\fh{{\mathfrak h}}
\def\fl{{\mathfrak l}}

\def\fp{{\mathfrak p}}
\def\fq{{\mathfrak q}}

\def\fP{{\mathfrak P}}
\def\Fun{{\F_1}}
\def\Funsq{{\F_{1^2}}}
\def\Funn{{\F_{1^n}}}

\def\int{\textup{int}}

\def\pr{\textup{pr}}
\def\id{\textup{id}}
\def\1{\textbf{1}}

\def\barx{{\overline{x}}}
\def\bary{{\overline{y}}}

\def\cprime{$'$}
\def\blanc{-}
\def\bp{{\mathcal{B}lpr}}

\def\canc{{\textup{canc}}}
\def\inv{{\textup{inv}}}
\def\proper{{\textup{prop}}}

\def\kar{{\textup{char\,}}}
\def\rk{{\textup{rk}}}

\def\prk{{\sim}}

\def\SRings{{\mathcal{S\!R}\textit{ings}}}

\def\Sets{\mathcal{S}\!\!\textit{ets}}
\def\Top{\mathcal{T}\!\!\textit{op}}
\def\={\equiv}
\def\n={\equiv\hspace{-10,5pt}/\hspace{3,5pt}}
\def\pos{{\geq 0}}
\def\un{\underline{\bf n}}
\def\bs{{\bar{s}}}
\def\px{\star} 
\def\Fpx{\F^\px}
\def\red{{\textup{red}}}

\def\top{{\textup{top}}}
\def\Sotimes{\otimes^{\!+}}
\def\Stimes{\times^{\!+}}

\def\toptimes{\times^\top}
\def\SA{{\vphantom{\A}}^{+\!}{\A}}   
\def\SP{{\vphantom{\P}}^{+}{\P}}   
\def\SG{{\vphantom{\G}}^{+}{\G}}   
\def\sT{{\scriptscriptstyle\cT\hspace{-3pt}}}
\def\hatexp{{{}_{\textstyle\hat\ }}}

\DeclareMathOperator{\BSch}{Sch_\Fun}

\DeclareMathOperator{\SSch}{Sch_\N^+}
\DeclareMathOperator{\TSch}{Sch_\cT}

\DeclareMathOperator{\BTSch}{Sch_{\Fun,\cT}}
\DeclareMathOperator{\rkBSch}{Sch_\Fun^\rk}

\newdir{ >}{{}*!/-5pt/@^{(}}\newcommand{\arincl}[1]{\ar@{ >->}@<-0,0ex>#1} 

\newcommand{\norm}[1]{\left| #1 \right|}
\newcommand{\prerk}[1]{\hat{\overline{ #1 }}}
\newcommand{\closure}[1]{\overline{\{ #1 \}}}
\newcommand{\gen}[1]{\langle #1 \rangle}

\newcommand{\bpquot}[2]{#1\!\sslash\!#2}
\newcommand{\bpgenquot}[2]{#1\!\sslash\!\gen{#2}}

\newcommand{\tinymat}[4]{\bigl( \begin{smallmatrix} #1 & #2 \\ #3 & #4 \end{smallmatrix} \bigr)}
\newcommand{\mat}[4]{ \begin{pmatrix} #1 & #2 \\ #3 & #4 \end{pmatrix} }

\begin{document}

\title{\bf The geometry of blueprints\\[0,8cm] \large\bf Part II: Tits-Weyl models of algebraic groups}
\author{Oliver Lorscheid\footnote{Department of Mathematics, University of Wuppertal, Gau\ss str.\ 20, 42097 Wuppertal, Germany, {\tt lorscheid@math.uni-wuppertal.de}.}}
\date{}

\maketitle

\begin{abstract}
 This paper is dedicated to a problem raised by Jacquet Tits in 1956: the Weyl group of a Chevalley group should find an interpretation as a group over what is nowadays called $\mathbb{F}_1$, \emph{the field with one element}. Based on Part I of The geometry of blueprints, we introduce the class of \emph{Tits morphisms} between blue schemes. The resulting \emph{Tits category} $\textup{Sch}_\mathcal{T}$ comes together with a base extension to (semiring) schemes and the so-called \emph{Weyl extension} to sets.

 We prove for $\cG$ in a wide class of Chevalley groups---which includes the special and general linear groups, symplectic and special orthogonal groups, and all types of adjoint groups---that a linear representation of $\cG$ defines a model $G$ in $\textup{Sch}_\mathcal{T}$ whose Weyl extension is the Weyl group $W$ of $\cG$. We call such models \emph{Tits-Weyl models}. The potential of Tits-Weyl models lies in \textit{(a)} their intrinsic definition that is given by a linear representation; \textit{(b)} the (yet to be formulated) unified approach towards thick and thin geometries; and \textit{(c)} the extension of a Chevalley group to a functor on blueprints, which makes it, in particular, possible to consider Chevalley groups over semirings. This opens applications to idempotent analysis and tropical geometry.
\end{abstract}
\pagebreak[4]

\begin{footnotesize}\tableofcontents\end{footnotesize}


\section*{Introduction}
\label{intro}\addcontentsline{toc}{section}{Introduction}


One of the main themes of $\Fun$-geometry was and is to give meaning to an idea of Jacques Tits that dates back to 1956 (see Section 13 in \cite{Tits56}). Namely, Tits proposed that there should be a theory of algebraic groups over a field of ``caract\'eristique une'', which explains certain analogies between geometries over finite fields and combinatorics. 

There are good expositions of Tits' ideas from a modern viewpoint (for instance, \cite{Soule04}, \cite{CC08} or \cite{L09}). We restrict ourselves to the following example that falls into this line of thought. The number of $\F_q$-rational points $\GL_n(\F_q)$ of the general linear group is counted by a polynomial $N(q)$ in $q$ with integral coefficients. The limit $\lim_{q\to1} N(q)/(q-1)^{n}$ counts the elements of the Weyl group $W=S_n$ of $\GL_n$. The same holds for any standard parabolic subgroup $P$ of $\GL_n$ whose Weyl group $W_P$ is a parabolic subgroup of the Weyl group $W$. While the group $\GL_n(\F_q)$ acts on the coset space $\GL_n/P(\F_q)$, which are the $\F_q$-rational points of a flag variety, the Weyl group $W=S_n$ acts on the quotient $W/W_P$, which is the set of decompositions of $\{1,\dotsc,n\}$ into subsets of cardinalities that correspond to the flag type of $\GL_n/P$.

The analogy of Chevalley groups over finite fields and their Weyl groups entered $\Fun$-geometry as the slogan: $\Fun$-geometry should provide an $\Fun$-model $G$ of every Chevalley group $\cG$ whose group $G(\Fun)=\Hom(\Spec\Fun,G)$ of $\Fun$-rational points equals the Weyl group $W$ of $\cG$. Many authors contributed to this problem: see \cite{Kapranov-Smirnov}, \cite{Manin95}, \cite{Soule04}, \cite{Deitmar05}, \cite{Haran07}, \cite{TV08}, \cite{CC08}, \cite{LL09}, \cite{Borger09}, \cite{L09}, \cite{Deitmar11} (this list is roughly in the order of appearance, without claiming to be complete).

However, there is a drawback to this philosophy. Recall that the Weyl group $W$ of a Chevalley group $\cG$ is defined as the quotient $W=N(\Z)/T(\Z)$ where $T$ is a split maximal torus of $\cG$ and $N$ is its normalizer in $\cG$. Under certain natural assumptions, a group isomorphism $G(\Fun)\stackrel\sim\to W$ yields an embedding $W\hookrightarrow N(\Z)$ of groups that is a section of the quotient map $N(\Z)\to W$. However, such a section does exist in general as the example $\cG=\SL_2$ witnesses (see Problem B in the introduction of \cite{L09} for more detail). 

This problem was circumvented in different ways. While some approaches restrict themselves to treat only a subclass of Chevalley groups over $\Fun$ (in the case of $\GL_n$, for instance, one can embed the Weyl group as the group of permutation matrices), other papers describe Chevalley groups merely as schemes without mentioning a group law. The more rigorous attempts to establish Chevalley groups over $\Fun$ are the following two approaches. In the spirit of Tits' later paper \cite{Tits66}, which describes the extended Weyl group, Connes and Consani tackled the problem by considering schemes over $\Funsq$ (see \cite{CC08}), which stay in connection with the extended Weyl group in the case of Chevalley groups. In the author's earlier paper \cite{L09}, two different classes of morphisms were considered: while rational points are so-called \emph{strong morphisms}, group laws are so-called \emph{weak morphisms}.

In this paper, we choose a different approach: we break with the convention that $G(\Fun)$ should be the Weyl group of $\cG$. Instead, we consider a certain category $\TSch$ of $\Fun$-schemes that comes together with ``base extension'' functors $(\blanc)_\Z:\TSch\to\Sch_\Z$\footnote{Note a slight incoherence with the notation of the main text of this paper where the functor $(\blanc)_\Z$ is denoted by $(\blanc)_\Z^+$. We will omit the superscript ``$+$'' also at other places of the introduction to be closer to the standard notation of algebraic geometry. An explanation for the need of the additional superscript is given in Section \ref{subsection:notation-and-conventions}.} to usual schemes and $\cW:\TSch\to\Sets$ to sets. Roughly speaking, a \emph{Tits-Weyl model} of a Chevalley group $\cG$ is an object $G$ in $\TSch$ together with a morphism $\mu:G\times G\to G$ such that $G_\Z$ together with $\mu_\Z$ is isomorphic to $\cG$ as a group scheme and such that $\cW(G)$ together with $\cW(\mu)$ is isomorphic to the Weyl group of $\cG$. We call the category $\TSch$ the \emph{Tits category} and the functor $\cW$ the Weyl extension.

\subsubsection*{A first heuristic}

Before we proceed with a more detailed description of the Tits category, we explain the fundamental idea of Tits-Weyl models in the case of the Chevalley group $\SL_2$. The standard definition of the scheme $\SL_{2,\Z}$ is as the spectrum of $\Z[\SL_2]=\Z[T_1,T_2,T_3,T_4]/(T_1T_4-T_2T_3-1)$, which is a closed subscheme of $\A^4_\Z=\Spec\Z[T_1,T_2,T_3,T_4]$. The affine space $\A^4_\Z$ has an $\Fun$-model in the language of Deitmar's $\Fun$-geometry (see \cite{Deitmar05}). Namely, $\A^4_\Fun=\Spec\Fun[T_1,T_2,T_3,T_4]$ where 
\[
 \Fun[T_1,T_2,T_3,T_4] \quad = \quad \{T_1^{n_1}T_2^{n_2}T_3^{n_3}T_4^{n_4}\}_{n_1,n_2,n_3,n_4\geq0}
\] 
is the monoid\footnote{For the sake of simplification, we do not require $\Fun[T_1,T_2,T_3,T_4]$ to have a zero. This differs from the conventions that are used in the main text, but this incoherence does not have any consequences for the following considerations.} of all monomials in $T_1$, $T_2$, $T_3$ and $T_4$. Its prime ideals are the subsets 
\[
 (T_i)_{i\in I} \quad = \quad \{T_1^{n_1}T_2^{n_2}T_3^{n_3}T_4^{n_4}\}_{n_i>0\text{ for one }i\in I}
\]
of $\Fun[T_1,T_2,T_3,T_4]$ where $I$ ranges to all subsets of $\{1,2,3,4\}$. Note that this means $(T_i)_{i\in I}=\emptyset$ for $I=\emptyset$. Thus $\A^4_\Fun=\{(T_i)_{i\in I}\}_{I\subset\{1,2,3,4\}}$. 

If one applies the naive intuition that prime ideals are closed under addition and subtraction to the equation
\[
 T_1T_4 \ - \ T_2T_3 \quad = \quad 1,
\]
then the points of $SL_{2,\Fun}$ should be the prime ideals $(T_i)_{i\in I}$ that do not contain both terms $T_1T_4$ and $T_2T_3$. This yields the set $\SL_{2,\Fun}=\{(\emptyset),(T_1),(T_2),(T_3),(T_4),(T_1,T_4),(T_2, T_3)\}$, which can be illustrated as
\begin{center}
  \includegraphics{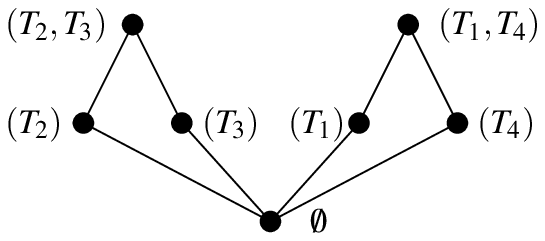}
\end{center}
where the vertical lines express the inclusion relation $(T_i)_{i\in J}\subset (T_i)_{i\in I}$. The crucial observation is that the two maximal ideals $(T_2,T_3)$ and $(T_1,T_4)$ of this set correspond to the subscheme $\bigl\{\tinymat \ast 0 0 \ast \bigr\}$ of diagonal matrices and the subscheme $\bigl\{\tinymat 0 \ast \ast 0 \bigr\}$ of anti-diagonal matrices of $\SL_{2,\Z}$, respectively, which, in turn, correspond to the elements of the Weyl group $W=N(\Z)/T(\Z)$ where $T=\bigl\{\tinymat \ast 0 0 \ast \bigr\}$ is the diagonal torus and $N$ its normalizer.

This example was the starting point for the development of the geometry of blueprints. A formalism that puts the above ideas on a solid base is explained in the preceeding Part I of this paper (see \cite{blueprints1}). Please note that we give brief definitions of blueprints and blue schemes in the introduction of Part I. In the proceeding, we will assume that the reader is familiar with this.

\subsubsection*{The Tits category}

It is the topic of this paper to generalize the above heuristics to other Chevalley groups and to introduce a class of morphisms that allows us to descend group laws to morphisms of the $\Fun$-model of Chevalley groups. Note that the approach of \cite{L09} is of a certain formal similarity: the tori of minimal rank in a torification of $\SL_{2,\Z}$ are the diagonal torus and the anti-diagonal torus. Indeed the ideas of \cite{L09} carry over to our situation.

The \emph{rank space} $X^\rk$ of a blue scheme $X$ is the set of the so-called ``points of minimal rank'' (which would be the points $(T_2,T_3)$ and $(T_1,T_4)$ in the above example) together with certain algebraic data, which makes it a discrete blue scheme. A \emph{Tits morphism} $\varphi:X\to Y$ between two blue schemes $X$ and $Y$ will be a pair $\varphi=(\varphi^\rk,\varphi^+)$ of a morphism $\varphi^\rk:X\to Y$ between the rank spaces and a morphism $\varphi^+:X^+\to Y^+$ between the associated semiring schemes\footnote{Please note that we avoid the notation ``$X_\N$'' from the preceeding Part I of this paper for reasons that are explained in Section \ref{subsection:notation-and-conventions}.} $X^+=X_\N$ and $Y^+=Y_\N$ that satisfy a certain compatibility condition. 

The Tits category $\TSch$ is defined as the category of blue schemes together with Tits morphisms. The Weyl extension $\cW:\TSch\to \Sets$ is the functor that sends a blue scheme $X$ to the underlying set $\cW(X)$ of its rank space $X^\rk$ and a Tits morphism $\varphi:X\to Y$ to the underlying map $\cW(\varphi):\cW(X)\to \cW(Y)$ of the morphism $\varphi^\rk:X^\rk\to Y^\rk$. The base extension $(\blanc)_\Z:\TSch\to\Sch_\Z$ sends a blue scheme $X$ to the scheme $X^+_\Z$ and a Tits morphism $\varphi:X\to Y$ to the morphism $\varphi^+_\Z:X^+_\Z\to Y^+_\Z$. Note that we can replace $\Z$ by a semiring $k$, which yields a base extension $(\blanc)_k:\TSch\to\Sch_k$ for every semiring $k$. We obtain the diagram
\[
 \xymatrix@R=0,5pc@C=6pc{  &  {\Sets} \\ {\TSch} \ar[ur]^{\cW}\ar[dr]^{(\blanc)^+} \\ & {\Sch_\N} \ar[r]^{(\blanc)_k}   & {\Sch_k}}.
\]

\subsubsection*{Results and applications}

The main result of this paper is that a wide class of Chevalley groups has a Tits-Weyl model. This includes the special and the general linear groups, symplectic groups, special orthogonal groups (of both types $B_n$ and $D_n$) and all Chevalley groups of adjoint type. Next to this, we obtain Tits-Weyl models for split tori, parabolic subgroups of Chevalley groups and their Levi-subgroups.

The strength of the theory of Tits-Weyl models can be seen in the following reasons. This puts it, in particular, in contrast to earlier approaches towards $\Fun$-models of algebraic groups, 


\smallskip\begin{center}\textit{Intrinsic definition through explicit formulas}\end{center}
Tits-Weyl models are determined by explicit formulas (as $T_1T_4-T_2T_3=1$ in the case of $\SL_2$), which shows that Tits-Weyl models are geometric objects that are intrinsically associated to representations in terms of generators and relations of the underlying scheme. The examples in Appendix \ref{app: examples} show that they are indeed accessible via explicit calculations. In other words, we can say that every linear representation of a group scheme $\cG$ yields an $\Fun$-model $G$. The group law of $\cG$ descends uniquely (if at all) to a Tits morphism $\mu:G\times G\to G$ that makes $G$ a Tits-Weyl model of $\cG$.

\smallskip\begin{center}\textit{Unified approach towards thick and thin geometries}\end{center}
Tits-Weyl models combine the geometry of algebraic groups (over fields) and the associated geometry of their Weyl groups in a functorial way. This has applications to a unified approach towards thick and thin geometries as alluded by Jacques Tits in \cite{Tits56}. A treatment of this will be the matter of subsequent work.

\smallskip\begin{center}\textit{Functorial extension to blueprints and semirings}\end{center}
A Chevalley group $\cG$ can be seen as a functor $h_\cG$ from rings to groups. A Tits-Weyl model $G$ of $\cG$ can be seen as an extension of $h_\cG$ to a functor $h_G$ from blueprints to monoids whose values $h_G(\Fun)$ and $h_G(\Funsq)$ stay in close connection to the Weyl group and the extended Weyl group (see Theorem \ref{thm: properties of tits-weyl groups}). In particular, $h_G$ is a functor on the subclass of semirings. This opens applications to geometry that is build on semirings; by name, to idempotent analysis as considered by Kolokoltsov and Maslov, et al.\ (see, for instance, \cite{Maslov94}), tropical geometry as considered Itenberg, Mikhalkin, et al.\ (see, for instance, \cite{Gathmann06}, \cite{Mikhalkin06} and, in particular, \cite[Chapter 2]{Mikhalkin10}), idempotent geometry that mimics $\Fun$-geometry (see \cite{CC09}, \cite{Lescot09} and \cite{Takagi10}) and analytic geometry from the perspective of Paugam (see \cite{Paugam09}), which generalizes Berkovich's and Huber's viewpoints on (non-archimedean) analytic geometry (see \cite{Berkovich93}, \cite{Berkovich98} and \cite{Huber96}).

\subsubsection*{Remarks and open problems}

The guiding idea in the formulation of the theory of Tits-Weyl models is to descend algebraic groups ``as much as possible''. This requires us to relinquish many properties that are known from the theory of group schemes, and to substitute these losses by a formalism that has all the desired properties, which are, roughly speaking, that the category and functors of interest are Cartesian and that Chevalley groups have a model such that its Weyl group is given functorially. As a consequence, we yield only monoids instead of group objects and there are no direct generalizations to relative theories---with one exception: there is a good relative theory over $\Funsq$. Tits monoids over $\Funsq$ are actually much easier to treat: the rank space has a simpler definition that does not require inverse closures, the universal semiring scheme is a scheme, Tits-Weyl models over $\Funsq$ are groups in $\TSch$ and many subtleties in the proofs about the existence of $-1$ in certain blueprints vanish. Note that the Tits-Weyl models that are established in this paper, immediately yield Tits-Weyl models over $\Funsq$ by the base extension $\blanc\otimes_\Fun\Funsq$ from $\Fun$ to $\Funsq$.

The strategy of this paper is to establish Tits-Weyl models by a case-by-case study. There are many (less prominent) Chevalley groups that are left out. Only for adjoint Chevalley groups, we construct Tits-Weyl models in a systematic way by considering their root systems. This raises the problem of the classification of Tits-Weyl models of Chevalley groups. In particular, the following questions suggest themselves.
\begin{itemize}
 \item Does every Chevalley group have a Tits-Weyl model? Is there a systematic way to establish such Tits-Weyl models?
 \item As explained before, a linear representation of a Chevalley group defines a unique Tits-Weyl model if at all. When do different linear representations of Chevalley groups lead to isomorphic Tits-Weyl models? Can one classify all Tits-Weyl models in a reasonable way?
 \item Every Tits-Weyl model of a Chevalley group in this text comes from a ``standard'' representation of the Chevalley group. Can one find a ``canonical'' Tits-Weyl model? What properties would such a canonical Tits-Weyl model have among all Tits-Weyl models of the Chevalley group?
\end{itemize}
See\ Appendix \ref{app: tits-weyl models of type a_1} for the explicit description of some Tits-Weyl models of type $A_1$.

\subsubsection*{Content overview}

The paper is organized as follows. In Section \ref{section:background}, we provide the necessary background on blue schemes to define the rank space of a blue scheme and the Tits category. This section contains a series of results that are of interest of its own while other parts are straightforward generalizations of facts that hold in usual scheme theory (as the results on sober spaces, closed immersions, reduced blueprints and fibres of morphisms). We try to keep these parts short and omit some proofs that are in complete analogy with usual scheme theory. Instead, we remark occasionally on differences between the theory for blue schemes and classical results. 

The more innovative parts of Section \ref{section:background} are the following. In Section \ref{subsection:mixed_characteristics}, we investigate the fact that a blue field can admit embeddings into semifields of different characteristics, which leads to the distinction of the \emph{arithmetic characteristic} and the \emph{potential characteristics} of a blue field and of a point $x$ of a blue scheme. Section \ref{subsection:fibres-of-morphisms} shows that the base extension morphism $\alpha_X:X_\N\to X$ is surjective; in case $X$ is cancellative, also the base extension morphism $\beta_X:X_\Z\to X$ is surjective. From the characterization of prime semifields in Section \ref{subsection:mixed_characteristics}, it follows that the points of a blue scheme are dominated by algebraic geometry over algebraically closed fields and idempotent geometry over the semifield $B_1=\bpgenquot{\{0,1\}}{1+1\=1}$. In Section \ref{subsection:products}, we investigate the underlying topological space of the fibre product of two blue schemes. In contrast to usual scheme theory, these fibre products are always a subset of the Cartesian product of the underlying sets. In Section \ref{subsection: relative additive closures}, we define \emph{relative additive closures}, a natural procedure, which will be of importance for the definition of rank spaces in the form of \emph{inverse closures}. As a last piece of preliminarily theory, we introduce \emph{unit fields} and \emph{unit schemes} in Section \ref{subsection: unit field and unit scheme}. Namely, the unit field of a blueprint $B$ is the subblueprint $B^\px=\{0\}\cup B^\times$ of $B$, which is a blue field. 

In Section \ref{section: the tits category}, we introduce the Tits category. In particular, we define \emph{pseudo-Hopf points} and the \emph{rank space} in Section \ref{subsection: rank space} and investigate the subcategory $\rkBSch$ of blue scheme that consists of rank spaces. Such blue schemes are called \emph{blue schemes of pure rank}. In Section \ref{subsection: tits morphisms}, we define \emph{Tits morphisms} and investigate its connections with usual morphisms between blue schemes. In particular, we will see that the notions of usual morphisms and Tits morphisms coincide on the common subcategories of semiring schemes and blue schemes of pure rank. 

In Section \ref{section: tits monoids}, we introduce the notions of a \emph{Tits monoid} and of \emph{Tits-Weyl models}. After recalling basic definitions and facts on groups and monoids in \emph{Cartesian categories} in Section \ref{subsection: reminder on cartesian categories}, we show in Section \ref{subsection: sch_t is cartesian} that the Tits category as well as some other categories and functors between them are Cartesian. In Section \ref{subsection: tits-weyl models}, we are finally prepared to define a Tits monoid as a monoid in $\TSch$ and a Tits-Weyl model of a group scheme $\cG$ as a Tits monoid with certain additional properties as described before. As first applications, we establish constant group schemes and tori as Tits monoids in $\rkBSch$ in Section \ref{subsection: tits-weyl models of pure rank}. Tori and certain semi-direct products of tori by constant group schemes, as they occur as normalizers of maximal tori in Chevalley groups, have Tits-Weyl models in $\rkBSch$.

In Section \ref{section: tits-weyl models of chevalley groups}, we establish Tits-Weyl models for a wide range of Chevalley groups. As a first step, we introduce the Tits-Weyl model $\SL_n$ of the special linear group in Section \ref{subsection: special linear group}. All other Tits-Weyl models of Chevalley groups will be realized by an embedding of the Chevalley group into a special linear group. In order to do so, we will frequently use an argument, which we call the \emph{cube lemma}, to descend morphisms. In Section \ref{subsection: closed subgroups of tits-weyl models}, we prove the core result Theorem \ref{thm: tits-weyl models of subgroups}, which provides a Tits-Weyl model for subgroups of a group scheme with a Tits-Weyl model under a certain hypothesis on the position of a maximal torus and its normalizer in the subgroup. We apply this to describe the Tits-Weyl model of general linear groups, symplectic groups and special orthogonal groups and some of their isogenies like adjoint Chevalley groups of type $A_n$ and orthogonal groups of type $D_n$. In Section \ref{subsection: adjoint chevalley groups}, we describe Tits-Weyl models of Chevalley groups of adjoint type that come from the adjoint representation of the Chevalley group on its Lie algebra. This requires a different strategy from the cases before and is based on formulas for the adjoint action over algebraically closed fields.

In Section \ref{section: tits-weyl models of subgroups}, we draw further conclusions from Theorem \ref{thm: tits-weyl models of subgroups}. If $\cG$ is a Chevalley group with a Tits Weyl model, then certain parabolic subgroups of $\cG$ and their Levi subgroups have Tits-Weyl models. We comment on unipotent radicals, but the problem of Tits-Weyl models of their unipotent radicals stays open.

We conclude the paper with Appendix \ref{app: examples}, which contains examples of non-standard Tits-Weyl models of tori and explicit calculations for three Tits-Weyl models of type $A_1$.

\section{Background on blue schemes} 
\label{section:background}

In this first part of the paper, we establish several general results on blue schemes that we will need to introduce the Tits category and Tits-Weyl models.


\subsection{Notations and conventions}
\label{subsection:notation-and-conventions}

To start with, we will establish certain notations and conventions used throughout the paper. We assume in general that the reader is familiar with the first part \cite{blueprints1} of this work. Occasionally, we will repeat facts if it eases the understanding, or if a presentation in a different shape is useful. For the purposes of this paper, we will, however, slightly alter notations from \cite{blueprints1} as explained in the following.

 \subsubsection*{All blueprints are proper and with a zero}

 The most important convention---which might lead to confusion if not noticed---is that we change a definition of the preceding paper \cite{blueprints1}, in which we introduced blueprints and blue schemes:

\begin{quote}
 \textit{Whenever we refer to a blueprint or a blue scheme in this paper, we understand that it is proper and with $0$.}
\end{quote}

\noindent 
When we make occasional use of the more general definition of a blueprint as in \cite{blueprints1}, then we will refer to it as a \emph{general blueprint}. In \cite{blueprints1}, we denoted the category of proper blueprints with $0$ by $\bp_0$. There is a functor $(\blanc)_0$ from the category $\bp$ of general blueprints to $\bp_0$. 

While for a monoid $A$ and a pre-addition $\cR$ on $A$, we denoted by $B=\bpquot{A}{\cR}$ the general blueprint with underlying monoid $A$, we mean in this paper by $\bpquot{A}{\cR}$ the proper blueprint $B_\proper$ with $0$, whose underlying monoid $A'$ differs in general from $A$. Namely, $A'$ is a quotient of $A\cup\{0\}$. 
 
To acknowledge this behaviour, we will call $\bpquot{A}{\cR}$ a \emph{representation of $B$} if $B=\bpquot{A}{\cR}$. If $A$ is the underlying monoid of $B$, then we call $\bpquot{A}{\cR}$ the \emph{proper representation of $B$ (with $0$)}.

We say that a morphism between blueprint is \emph{surjective} if it is a surjective map between the underlying monoids. In other words, $f:B\to C$ is surjective if for all $b\in C$, there is an $a\in B$ such that $b=f(a)$. If $B=\bpquot A\cR$ and $C=\bpquot\Spec{A'}{\cR'}$ are representations, which do not necessarily have to be proper, and $f:A\to A'$ is a surjective map, then $f:B\to C$ is a surjective morphism of blueprints.

Note that the canonical morphism $B\to B_0$ for a general blueprint $B$ induces a homeomorphism between their spectra. To see this, remember that the proper quotient is formed by identifying $a,b\in B$ if they satisfy $a\=b$. If $a\=b$, then a prime ideal of $B$ contains either both elements or none. Since every ideal contains $0$ if $B$ has a zero, it follows that the spectra of $B$ and $B_\proper$ are homeomorphic. 

Accordingly, we refer to proper blue schemes with $0$ simply by blue schemes, and call blue schemes in the sense of \cite{blueprints1} \emph{general} blue schemes. If $X$ is a general blue scheme, then $X_0\to X$ is a homeomorphism, thus we might make occasional use of general blue schemes if we are only concerned with topological questions. We denote the category of blue schemes (in the sense of this paper) by $\BSch$. 

Note that we do not require that blueprints are global. We will not mention this anymore, but remark here that all explicit examples of blueprints in this text are global. For general arguments that need the fact that morphisms between blue schemes are locally algebraic (Theorem 3.23 of \cite{blueprints1}), we take care to work with the coordinate blueprints $\Gamma X=\Gamma(X,\cO_X)$ and $\Gamma Y=\Gamma(Y,\cO_Y)$, which are by definition global blueprints.

 \subsubsection*{Blue schemes versus semiring schemes}

By a \emph{(semiring) scheme}, we mean a blue scheme whose coordinate blueprints are (semi-) rings. We denote the category of semiring schemes by $\SSch$ and the category of schemes by $\Sch^+_\Z$. Though the categories $\SSch$ and $\Sch^+_\Z$ embed as full subcategories into $\BSch$, and these embeddings have left-adjoints, one has to be careful with certain categorical constructions like fibre products or affine spaces, whose outcome depends on the chosen category. Roughly speaking, we will apply the usual notation from algebraic geometry if we carry out a construction in the larger category $\BSch$, and we will use a superscript ${}^+$ if we refer to the classical construction in the category of schemes. Usually, constructions in $\SSch$ coincide with constructions in $\Sch_\Z^+$, so that we can use the superscript ${}^+$ also for constructions in $\SSch$. 

We explain in the following, which constructions are concerned, and how the superscript ${}^+$ is used.

\begin{center} 
 \textit{Tensor products and fibre products}
\end{center}

\noindent We denote the functor that associates to a blueprint the generated semiring by $(\blanc)^+$. Thus we write $B^+$ for the associated semiring (which is $B_\N$ in the notation of \cite{blueprints1}), and $X^+$ for the semiring scheme associated to a blue scheme $X$. These come with canonical morphisms $B\to B^+$ and $\beta: X^+\to X$.

We have seen in \cite{blueprints1} that the category of blueprints contains tensor products $B\otimes_DC$. To distinguish these from the tensor product of semirings in case $B$, $C$ and $D$ are semirings, we write for the latter construction $B\Sotimes_DC$. 

Since $(\blanc)^+:\bp_0\to\SRings$ is the right-adjoint of the forgetful functor $\SRings\to\bp_0$, we have that $(B\otimes_DC)^+=B^+\Sotimes_{D^+}C^+$. Since we are considering only $\bp_0$, the functor $(\blanc)_\inv$ from \cite{blueprints1}, which adjoins additive inverses to a blueprint $B$ is isomorphic to the functor $(\blanc)\otimes_\Fun\Funsq$ (recall from \cite[Lemma 1.4]{blueprints1} that a blueprint is with inverses if and only if it is with $-1$). This implies that $B^+\Sotimes_{D^+}C^+$ if and only if one of $B$, $C$ or $D$ is with a $-1$. In particular, $B\Sotimes_DC$ is a ring if $B$, $C$ and $D$ are rings; and $(B\otimes_\Fun\Funsq)^+$ is the ring generated by a blueprint $B$. 

The corresponding properties of the tensor product hold for fibre products of blue schemes. We denote by $X\times_ZY$ the fibre product in $\BSch$, while $X\Stimes_ZY$ stays for the fibre product in $\SSch$. Then we have $(X\times_ZY)^+=X^+\Stimes_{Z^+}Y^+$. For a blue scheme $X$ we denote $X_B$ or $X\times_\Fun B$ the base extension $X\times_{\Spec\Fun}\Spec B$. If $B$ is a semiring, then $X^+_B$ stays for $(X_B)^+$. Note that in general $(X^+)_B$ is \emph{not} a semiring scheme. In particular $X^+_\N=X^+$ and $X_\Z^+=(X_\Funsq)^+$, which is the scheme associated to $X$.

\begin{center} 
 \textit{Free algebras and affine space}
\end{center}

\noindent Another construction that needs a specification of the category is the functor of free algebras. We denote the \emph{free object in a set $\{T_i\}_{i\in I}$ over a blueprint $B$} in the category $\bp_0$ by $B[T_i]$. If $B$ is a semiring, we denote the free object in $\SRings$ by $B[T_i]^+$. If $B$ is a ring, then $B[T_i]^+$ is a ring. The spectrum of the free object on $n$ generators is $n$-dimensional affine space: $\A^n_B=\Spec B[T_i]$ if $B$ is a blueprint, and $\SA^n_B=\Spec B[T_i]^+$ is $B$ is a semiring.

Note that localizations coincide for blueprints and semirings, i.e.\ if $S$ is a multiplicative subset of a blueprint $B$ and $\sigma:B\to B^+$ is the canonical map, then $(S^{-1}B)^+=\sigma(S)^{-1}B^+$. We denote the localization of the free blueprint in $T$ over $B$ by $B[T^{\pm1}]$ or $B[T^{\pm1}]^+$, depending whether we formed the free algebra in $\bp_0$ or $\SRings$ The corresponding geometric objects are the \emph{multiplicative group schemes} $\G_{m,B}$ and $\SG_{m,B}$, respectively. The \emph{higher-dimensional tori} $\G^n_{m,B}$ and $\SG^n_{m,B}$ are defined in the obvious way.

There are other schemes that can be defined either category $\BSch$ and $\SSch$. For example, the definition of projective $n$-space (as a scheme) by gluing $n$-dimensional affine spaces along their intersections generalizes to semiring schemes and blue schemes. We define $\P^n_B$ as the \emph{projective $n$-space} obtained by gluing affine planes $\A^n_B$ if $B$ is a blueprint, $\SP^n_B$ as the {projective $n$-space} obtained by gluing $\SA^n_B$ if $B$ is a semiring and $\SP^n_B$ as the {projective $n$-space} obtained by gluing $\SA^n_B$ if $B$ is a ring. A more conceptual viewpoint on this is given in a subsequent paper where we introduce the functor $\Proj$ for graded blueprints and graded semirings.


\subsection{Sober and locally finite spaces}
\label{subsection-sober-and-locally-finite-spaces}

While the underlying topological space of a scheme of finite type over an (algebraically closed) field consists typically of infinitely many points, a scheme of finite type over $\Fun$ has only finitely many points. This allows a more combinatorial view for the latter spaces, which is the objective of this section. 

To begin with, recall that a topological space is \emph{sober} if every irreducible closed subset has a unique generic point.

\begin{prop}
 The underlying topological space of a blue scheme is sober.
\end{prop}

\begin{proof}
 Since the topology of a blue scheme is defined by open affine covers, a blue scheme is sober if all of its affine open subsets are sober. Thus assume $X=\Spec B$ is an affine blue scheme. A basis of the topology of closed subsets of $X$ is formed by
 \[
  V_a \quad = \quad \{ \ \fp\subset B\text{ prime ideal}\ | \ a\in\fp \ \}
 \]
 where $a$ ranges through all elements of $B$. Given an irreducible closed subset $V$, we define $\eta=\bigcap_{\fp\in V}\fp$, which is an ideal of $B$. 

 We claim that $\eta$ is a prime ideal. Let $ab\in\eta$. Since every $\fp\in V$ contains $\eta$ and therefore $ab$, we have $V\subset V_{ab}=V_a\cup V_b$. Thus
 \[
  V \ = \ V_{ab}\cap V \ = \ (V_a \cup V_b) \cap V \ = \ (V_a\cap V) \cup (V_b\cap V).
 \]
 Since $V$ is irreducible, either $V=V_a\cap V$ or $V=V_b\cap V$, i.e.\ $V\subset V_a$ or $V\subset V_b$ This means that either $a\in\eta$ or $b\in\eta$, which shows that $\eta$ is a prime ideal.

 The closed subset $V$ is the intersection of all $V_a$ with $a\in\fp$ for all $\fp\in V$. Since $\eta$ is defined as intersection of all $\fp\in V$, it is contained in all $V_a$ that contain $V$. Thus $\eta\in V$.

 We show that $\eta$ is the unique generic point of $V$. The closure $\overline{\{\eta\}}$ of $\eta$ consists of all prime ideals that contain $\eta$, and thus $V\subset \overline{\{\eta\}}$. Thus $\eta$ is a generic point of $V$. If $\eta'$ is another generic point of $V$, then $\eta'$ is contained in every prime ideal $\fp\in V$. Thus $\eta'=\eta$, and $\eta$ is unique.
\end{proof}

\begin{df}
 A topological space is \emph{finite} if it has finitely many points. A topological space is \emph{locally finite} if it has an open covering by finite topological spaces.
\end{df}

These notions find application to blue schemes of (locally) finite type as introduced in Section \ref{subsection:closed_subschemes}.

\begin{lemma}
 Let $X$ be a locally finite and sober topological space. Let $x\in X$. Then the set $\{x\}$ is locally closed in $X$.
\end{lemma}

\begin{proof}
 Since this is a local question, we may assume that $X$ is finite. Define $V=\bigcup_{x\notin\closure{y}}\closure{y}$, which is a finite union of closed subsets, which does not contain $x$. Thus $U=X-V$ is an open neighbourhood of $x$. If $x\in\closure{y}$, i.e.\ $y\in U$, and $y\in\closure{x}$, then $x=y$ since $X$ is sober. Therefore $U\cap\closure{x}=\{x\}$, which verifies that $\{x\}$ is locally closed.
\end{proof}

In the following we consider a topological space $X$ as a poset by the rule $x\leq y$ if and only if $y\in\closure x$ for $x,y\in X$.

\begin{lemma}\label{lemma-open-and-closed-subsets-of-locally-finite-space}
 Let $X$ be a locally finite topological space, and $U$ a subset of $X$. Then $U$ is open (closed) if and only if for all $x\leq y$, $y\in U$ implies $x\in U$ ($x\in U$ implies $y\in U$).
\end{lemma}

\begin{proof}
 We prove only the statement about closed subsets. The statement about open subsets is complementary and can be easily deduced by formal negation of the following.

 Since this is a local question, we may assume that $X$ is finite. If $U$ is closed, then $x\leq y$ and $x\in U$ implies $y\in\closure{x}\subset U$. 

 Conversely, if $x\leq y$ and $x\in U$ implies $y\in U$ for all $x,y\in X$, then we have that for all $x\in U$ its closure $\closure{x}=\{y\in X|x\leq y\}$ is a subset of $U$. Since $U$ is finite, $\closure{U}=\bigcup_{x\in U}\closure{x}$ is the closure of $U$, and it is contained in $U$. Thus $U$ is closed.
\end{proof}

\begin{prop} \label{prop: continuous=order-preserving for locally finite spaces}
 Let $X$ and $Y$ be topological spaces. A continuous map $f:X\to Y$ is order-preserving. If $X$ is locally finite, then an order-preserving map $f:X\to Y$ is continuous.
\end{prop}

\begin{proof}
 Let $f:X\to Y$ be continuous and $x\leq y$ in $X$. The set $f^{-1}(\closure{f(x)})$ is closed and contains $x$. Thus $y\in f^{-1}(\closure{f(x)})$, which means that $f(x)\leq f(y)$. This shows that $f$ is order-preserving.

 Let $X$ be locally finite and $f:X\to Y$ order-preserving. Let $V$ be a closed subset of $Y$. We have to show that $f^{-1}(V)$ is a closed subset of $X$. We apply the characterization of closed subsets from Lemma \ref{lemma-open-and-closed-subsets-of-locally-finite-space}: let $x\in f^{-1}(V)$ and $x\leq y$. Since $f$ is order-preserving, $f(x)\leq f(y)$. This means that $f(y)\in\closure{f(x)}\subset V$ and thus $y\in f^{-1}(V)$.
\end{proof}

\begin{ex}
 The previous lemma and proposition show that the underlying topological space of a locally finite blue scheme is completely determined by its associated poset. We will illustrate locally finite schemes $X$ by diagrams whose points are points $x\in X$ and with lines from a lower point $x$ to a higher point point $y$ if $x<y$ and their is no intermediate $z$, i.e.\ $x<z<y$. For example, the underlying topological space of $\A^1_{\Fun}=\Spec\Fun[T]$ consists of the prime ideals $(0)$ and $(T)$, the latter one being a specialization of the former one. Similarly, $\A^2_\Fun=\Spec\Fun[S,T]$ has four points $(0)$, $(S)$, $(T)$ and $(S,T)$. The projective line $\P^1_{\Fun}=\A^1_{\Fun}\coprod_{\G_{m,\Fun}}\A^1_{\Fun}$ has two closed points $[0:1]$ and $[1:0]$ and one generic point $[1:1]$. Similarly, the points of $\P^2_\Fun$ correspond to all combinations $[x_0:x_1:x_2]$ with $x_i=0$ or $1$ with exception of $x_0=x_1=x_2=0$. These blue schemes can be illustrated as in Figure \ref{figure: a1,a2,p1,p2}.
\begin{figure}[h]
 \begin{center}
  \includegraphics{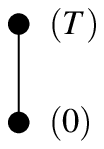} \hspace{0,5cm} \includegraphics{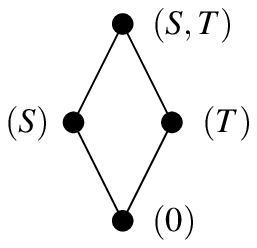}   \hspace{0,5cm} \includegraphics{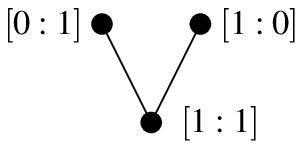} \hspace{0,5cm} \includegraphics{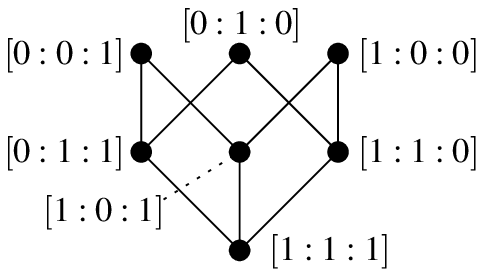}
  \caption{The blue scheme $\A^1_\Fun$, $\A^2_\Fun$, $\P^1_\Fun$ and $\P^2_\Fun$ (from left to right)}
  \label{figure: a1,a2,p1,p2}
 \end{center} 
\end{figure}
\end{ex}


\subsection{Closed immersions}
\label{subsection:closed_subschemes}

An important tool to describe all points of a blue scheme are closed immersions into known blue schemes. We generalize the notion of closed immersions as introduced in \cite{CHMWW11} to blue schemes.

\begin{df}
 A morphism $\varphi:X\to Y$ of blue schemes is a \emph{closed immersion} if $\varphi$ is a homeomorphism onto its image and for every affine open subset $U$ of $Y$, the inverse image $V=\varphi^{-1}(U)$ is affine in $X$ and $\varphi^\#(U):\Gamma(\cO_Y,U)\to\Gamma(\cO_X,V)$ is surjective. A \emph{closed subscheme of $Y$} is a blue scheme $X$ together with a closed immersion $X\to Y$.
\end{df}

\begin{rem}
 In contrast to usual scheme theory, it is in general not true that the image of a closed immersion $\varphi:X\to Y$ is a closed subset of $Y$. Consider, for instance, the diagonal embedding $\Delta:\A^1_\Fun\to\A^1_\Fun\times_\Fun\A^1_\Fun=\A^2_\Fun$, which corresponds to the blueprint morphism $\Fun[T_1,T_2]\to\Fun[T]$ that maps both $T_1$ and $T_2$ to $T$. Then the inverse image of the $0$-ideal is the $0$-ideal, and the inverse image of the ideal $(T)$ is $(T_1,T_2)$. But the set $\{(0),(T_1,T_2)\}$ is not closed in $\A^2_\Fun$ as illustrated in Figure \ref{figure: diagonal of a1}. 
\begin{figure}[h]
 \begin{center}
  \includegraphics{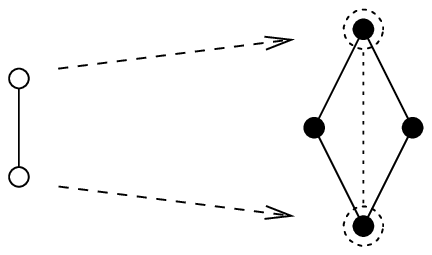}
  \caption{The diagonal embedding $\Delta:\A^1_\Fun\to\A^2_\Fun$}
  \label{figure: diagonal of a1}
 \end{center} 
\end{figure}
\end{rem}

\begin{lemma}\label{lemma: surjections are closed immersions}
 Let $f:B\to C$ be a surjective morphism of blueprints. Then $f^\ast:\Spec C\to \Spec B$ is a closed immersion.
\end{lemma}

\begin{proof}
 Put $X=\Spec B$, $Y=\Spec C$ and $\varphi=f^\ast: Y\to X$. We first show that $\varphi$ is injective. Since $f:B\to C$ is surjective, $f(f^{-1}(\fp))=\fp$ for all $\fp\subset C$. If $\fp$ and $\fp'$ are prime ideals of $C$ with $\varphi(\fp)=\varphi(\fp')$, then $\fp=f(f^{-1}(\fp))=f(f^{-1}(\fp'))=\fp'$. Thus $\varphi$ is injective.

 For to show that $\varphi$ is a homeomorphism onto its image, we have to verify that every open subset $V$ of $Y$ is the inverse image $\varphi^{-1}(U)$ of some open subset $U$ of $X$. It suffices to verify this for basic opens. Let $V_a=\{\fp\in Y|a\notin\fp\}$ for some $a\in C$. Then there is a $b\in B$ such that $f(b)=a$ and thus $V_a=\varphi^{-1}(U_b)$ for $U_b=\{\fq\in X|b\notin\fq\}$. Hence $\varphi$ is a homeomorphism onto its image.

 Affine opens of $X=\Spec B$ are of the form $U\simeq\Spec(S^{-1}B)$ for some multiplicative subset $S$ of $B$. The inverse image $V=\varphi^{-1}(U)$ is then of the form $V\simeq\Spec(f(S)^{-1}C)$, and thus affine. Since $f:B\to C$ is surjective, also the induced map $S^{-1}f:S^{-1}B\to f(S)^{-1} C$ is surjective. Thus $\varphi$ is a closed immersion.
\end{proof}

If $A$ is a monoid (with $0$), then we consider $A$ as the blueprint $B=\bpgenquot{A}{\emptyset}$. Since $\bpgenquot{A}{\emptyset}\to\bpquot A\cR$ is surjective for any pre-addition $\cR$ on $A$, we have the following immediate consequence of the previous lemma.

\begin{cor} \label{cor: Spec B is a subset of Spec A}
 If $B=\bpquot A\cR$ is a representation of the blueprint $B$, then $\Spec B \subset \Spec A$. \qed
\end{cor}

 Let $f:B\to C$ be a morphism of blueprints. We say that $C$ is \emph{finitely generated over $B$ (as a blueprint)} or that $f$ is \emph{of finite type} if $f$ factorizes through a surjective morphism $B[T_1,\dotsc,T_n]\to C$ for some $n\in\N$. If $C$ is finitely generated over a blue field, then $C$ has finitely many prime ideals and thus $\Spec C$ is finite. 

 Let $\varphi:X\to S$ be a morphism of blue schemes. We say that $X$ is \emph{locally of finite type over $S$ (as a blue scheme)} if for every affine open subset $U$ of $X$ that is mapped to an affine open subset $V$ of $S$ the morphism $\varphi^\#(V):\Gamma(\cO_S,V)\to\Gamma(\cO_X,U)$ between sections is of finite type. We say that $X$ is \emph{of finite type over $S$ (as a blue scheme)} if $X$ is locally finitely generated and compact. If $X$ is (locally) of finite type over a blue field $\kappa$, i.e.\ $X\to\Spec\kappa$ is (locally) of finite type, then $X$ is (locally) finite.

\begin{ex} \label{ex: subschemes of affine space over f1}
 We can apply Corollary \ref{cor: Spec B is a subset of Spec A} to describe the topological space of affine blue schemes of finite type over $\Fun$. It is easily seen that the prime ideals of the free blueprint $\Fun[T_1,\dotsc. T_n]$ are of the form $\fp_I=(T_i)_{i\in I}$ where $I$ is an arbitrary subset of $\un=\{1,\dotsc, n\}$. Every blueprint $B$ that is finitely generated over $\Fun$ has a representation $B=\bpquot{\Fun[T_1,\dotsc,T_n]}{\cR}$, then every prime ideal of $B$ is also of the form $\fp_I$ (where it may happen that $T_i\=T_j$ if the representation of $B$ is not proper).

 More precisely, $\fp_I$ is a prime ideal of $B=\bpquot{\Fun[T_1,\dotsc,T_n]}{\cR}$ if and only if for all additive relations $\sum a_i\=\sum b_j$ in $B$, either all terms $a_i$ and $b_j$ are contained in $\fp_I$ or at least two of them are not contained in $\fp_I$.

 Since $\A^n_\Fun$ is finite, $\Spec B$ is so, too, and the topology of $\Spec B$ is completely determined by the inclusion relation of prime ideals of $B$.
\end{ex}


\subsection{Reduced blueprints and closed subschemes}
\label{subsection: reduced blueprints}

In this section, we extend the notions of reduced rings and closed subschemes to the context of blueprints and blue schemes. Since all proofs have straight forward generalizations, we forgo to spell them out and restrict ourselves to state the facts that are needed in this paper.

\begin{df}
 Let $B$ be a blueprint and $I\subset B$ an ideal. The \emph{radical $\Rad(I)$ of $I$} is the intersection $\bigcap\fp$ of all prime ideals $\fp$ of $B$ that contain $I$. The \emph{nilradical $\Nil(B)$ of $B$} is the radical $\Rad(0)$ of the $0$-ideal of $B$. 
\end{df}

\begin{rem}
 If $B$ is a ring, then $\Rad(I)$ equals the set 
 \[
 \sqrt I \quad = \quad \{ \ a\in B \ | \ a^n\in I\text{ for some } n>0 \ \}.
 \]
 The inclusion $\sqrt I\subset\Rad(I)$ holds for all blueprints and $\Rad(I)\subset\sqrt I$ holds true if $B$ is with $-1$. The latter inclusion is, however, not true in general as the following example shows.

 Let $B=\bpgenquot{\Fun[S,T,U]}{S\=T+U}$ and $I=(S^2, T^2)=\{S^2b,T^2b|b\in B\}$. Then $\Rad(I)=(S,T,U)$ while $\sqrt I=\{Sb,Tb|b\in B\}$ does not contain $U$.
\end{rem}

If, however, $I$ is the $0$-ideal, then the equality $\sqrt 0=\Rad(0)$ holds true for all blueprints.

\begin{lemma}
 Let $B$ be a blueprint. Then the following conditions are equivalent.
 \begin{enumerate}
  \item $\Nil(B)=0$;
  \item $\sqrt 0=0$;
  \item $0$ is a prime ideal of $B$.
 \end{enumerate}
 If $B$ satisfies these conditions, then $B$ is said to be \emph{reduced}. \qed
\end{lemma}

We define $B^\red=B/\Nil(B)$ as the quotient of $B$ by its nilradical, which is a reduced blueprint. Every morphism from $B$ into a reduced blueprint factors uniquely through the quotient map $B\to B^\red$.

\begin{lemma}
 The universal morphism $f:B\to B^\red$ induces a homeomorphism $f^\ast:\Spec B^\red\to \Spec B$ between the underlying topological spaces of the spectra of $B$ and $B^\red$. \qed
\end{lemma}

\begin{prop}
 Let $X$ be a blue scheme with structure sheaf $\cO_X$. Then the following conditions are equivalent.
 \begin{enumerate}
  \item $\cO_X(U)$ is reduced for every open subset $U$ of $X$.
  \item $\cO_X(U_i)$ is reduced for all $i\in I$ where $\{U_i\}_{i\in I}$ is an affine open cover of $X$.
 \end{enumerate}
 If $X$ satisfies these conditions, then $X$ is said to be \emph{reduced}. \qed
\end{prop}

\begin{cor}
 A blueprint $B$ is reduced if and only its spectrum $\Spec B$ is reduced.\qed
\end{cor}

Let $X$ be a blue scheme. We define the reduced blue scheme $X^\red$ as the underlying topological space of $X$ together with the structure sheaf $\cO_X^\red$ that is defined by $\cO_X^\red(U)=\cO_X(U)^\red$. It comes together with a closed immersion $X^\red\to X$, which is a homeomorphism between the underlying topological spaces.

More generally, there is for every closed subset $V$ of $X$ a reduced closed subscheme $Y$ of $X$ such that the inclusion $Y\to X$ has set theoretic image $V$ and such that every morphism $Z\to X$ from a reduced scheme $Z$ to $X$ with image in $V$ factors uniquely through $Y\hookrightarrow X$. We call $Y$ the \emph{(reduced) subscheme of $X$ with support $V$}.

\begin{rem}
 Note that $Y$ is not the smallest subscheme of $X$ with support $V$ since in general, there quotients of blueprints that are the quotient by an ideal. For example consider the blue field $\{0\}\cup\mu_n$ where $\mu_n$ is a group of order $n$ together with the surjective blueprint morphism $\{0\}\cup\mu_n\to \{0\}\cup\mu_m $ where $m$ is a divisor of $n$. Then $\Spec(\{0\}\cup\mu_m)\to\Spec(\{0\}\cup\mu_n)$ is a closed immersion of reduced schemes with the same topological space, which consists of one point.
\end{rem}


\subsection{Mixed characteristics}
\label{subsection:mixed_characteristics}

A major tool for our studies of the topological space of a blue scheme $X$ are morphisms $\Spec k\to X$ from the spectrum of a semifield $k$ into $X$, whose image is a point $x$ of $X$. In particular, the characteristics that $k$ can assume are an important invariant of $x$. Note that unlike fields (in the usual sense), a blue field might admit morphisms into fields of different characteristics. For instance, the blue field $\Fun=\{0,1\}$ embeds into every field. In this section, we investigate the behaviour of blue fields and their characteristics.

\begin{df}
 Let $B$ be a blueprint. The \emph{(arithmetic) characteristic $\kar B$ of $B$} is the characteristic of the ring $B^+_\Z$.
\end{df}

We apply the convention that the characteristic of the zero ring $\{0\}$ is $1$. Thus a ring is of characteristic $1$ if and only it is the zero ring. As the examples below show, there are, however, non-trivial blueprints of characteristic $1$.

With this definition, the arithmetic characteristic of a blueprint $B=\bpquot{A}{\cR}$ is finite (i.e.\ not equal to $0$) if and only if there is an additive relation of the form
\[
 \sum a_i + \underbrace{1+\dotsb+1}_{n\text{-times}} \quad \= \quad \sum a_i
\]
in $\cR$ with $n>0$, and $\kar B$ is equal to the smallest such $n$. 

A \emph{prime semifield} is a semifield that does not contain any proper sub-semifield. Prime semifields are close to prime fields, which are $\Q$ or $\F_p$ where $p$ is a prime. Indeed, $\F_p$ are prime semifields since they do not contain any smaller semifield. The rational numbers $\Q$ contain the smaller prime semifield $\Q_\pos$ of non-negative rational numbers. There is only one more prime semifield, which is the idempotent semifield $\B_1=\bpgenquot{\{0,1\}}{1+1\=1}$ (cf.\ \cite{Lescot09} and \cite[p.\ 13]{CC09}). Note that semifields, and, more generally, semirings $B$ that contain $\B_1$ are idempotent, i.e.\ $a+a\=a$ for all $a\in B$. 

Every semifield $k$ contains a unique prime semifield, which is generated by $1$ as a semifield. If $k$ contains $\F_p$, then $\kar k=p$ and $k$ is a field since it is with $-1$. If $k$ contains $\B_1$, then $\kar k=1$. If $k$ contains $\Q_\pos$, then $k^+_\Z$ is either a field of characteristic $0$ or the zero ring $\{0\}$. Thus the characteristic of $k$ is either $0$ or $1$. In the former case, $k\to k^+_\Z$ is a morphism into a field of characteristic $0$. To see that the latter case occurs, consider the example $k=\bpgenquot{\Q_\pos(T)}{T+1\=T}$ where $\Q_\pos(T)$ are all rational functions $P(T)/Q(T)$ where $P(T)$ and $Q(T)$ are polynomials with non-negative rational coefficients. Indeed, $k^+_\Z=\{0\}$ since $1\=(T+1)-T\=0$; it is not hard to see that $k$ contains $\Q_\pos$ as constant polynomials.



\begin{df}
 Let $B$ be a blueprint. An integer $p$ is called a \emph{potential characteristic of $B$} if there is a semifield $k$ of characteristic $p$ and a morphism $B\to k$. We say that $B$ is \emph{of mixed characteristics} if $B$ has more than one potential characteristic, and that $B$ is \emph{of indefinite characteristic} if all primes $p$, $0$ and $1$ are potential characteristics of $B$. A blueprint $B$ is almost of indefinite characteristic, if all but finitely many primes $p$ are potential characteristics of $B$.
\end{df}

We investigate the potential characteristics of semifields.

\begin{lemma}\label{lemma-semifields-have-1-or-2-geometric-char}
 Let $k$ be a semifield. Then there is a morphism $k\to\B_1$ if and only if $k$ is without an additive inverse $-1$ of $1$. Consequently, $\kar k$ is the only potential characteristic of $k$, unless $k$ is of arithmetic characteristic $0$, but without $-1$. In this case, $k$ has potential characteristics $0$ and $1$. 
\end{lemma}

\begin{proof}
 Let $k$ be a semifield. Then the map $f:k\to\B_1$ that sends $0$ to $0$ and every other element to $1$ is multiplicative. If $k$ is with $-1$, then $1+(-1)\=0$ in $k$, but $1+1\n= 0$ in $\B_1$; thus $f$ is not a morphism in this case. If $k$ is without $-1$, then for every relation $\sum a_i\=\sum b_j$ in $k$ neither sum is empty. Since $\sum 1\=\sum 1$ holds true in $\B_1$ if neither sum is empty, $f$ is a morphism of semifields. This proves the first statement of the lemma.

 Trivially, the arithmetic characteristic of a semifield $k$ is a potential characteristic of $k$. If $k$ contains $-1$, then $k$ is a field and has a unique characteristic. Since there is no morphism from an idempotent semiring into a cancellative semifield, semifields of characteristic $1$ have only potential characteristic $1$. The only case left out, is the case that $k$ is of arithmetic characteristic $0$, but is without $-1$. Then there is a morphism $k\to \B_1$, and thus $k$ has potential characteristic $0$ and $1$. 
\end{proof}

Let $B$ be a blueprint of arithmetic characteristic $n>1$. Since every morphism $B\to k$ into a semifield $k$ factorizes through $B^+$, which is with $-1=n-1$, the semifield $k$ is a field and every potential characteristic $p$ of $B$ is a divisor of $n$. This generalizes trivially to the cases $n=0$ and $n=1$. The reverse implication is not true since $1$ divides all other characteristics. Even if we exclude $p=1$ as potential characteristic, the reverse implication does also not hold for blueprints of arithmetic characteristic $0$, as the example $B=\Q$ and, more general, every proper localization of $\Z$, witnesses. However, it is true for blueprints of finite arithmetic characteristic.

\begin{lemma}\label{lemma-divisors-of-char-are-geometric-char}
 Let $B$ be a blueprint of characteristic $n\geq1$. If $p$ is a prime divisor of $n$, then $p$ is a potential characteristic of $B$. 
\end{lemma}

\begin{proof}
 If $p$ divides $n$, then $n>1$ and $p=1+\dotsb +1$ generates a proper ideal in $B^+_\Z$. Thus $B^+_\Z/(p)$ is a ring of characteristic $p$ and, in particular, not the zero ring. Therefore, there is a morphism $B^+_\Z/(p)\to k$ into a field $k$ of characteristic $p$. The composition $B\to B^+_\Z\to B^+_\Z/(p)\to k$ verifies that $p$ is a potential characteristic of $B$.
%
\end{proof}

If $G$ is an abelian semigroup, then we denote by $B[G]$ the \emph{(blue) semigroup algebra of $G$ over $B$}, which is the blueprint $\bpquot A\cR$ with $A=B\times G$ and $\cR=\gen{\sum(a_i,1)\=\sum(b_j,1)|\sum a_i\=\sum b_j\text{ in }B}$. Since there are morphisms $B\to B[G]$, which maps $b$ to $(b,1)$, and $B[G]\to B$, which maps $(b,g)$ to $b$, the potential characteristics of $B$ and $B[G]$ are the same. Thus every blue field of the form $\Fun[G]$ is of indefinite characteristic. More generally, we have the following. Recall from \cite{blueprints1} that $\Funn$ (for $n\geq1$) is the blue field $\bpquot {(0\cup\mu_n)}\cR$ where $\mu_n$ is a cyclic group with $n$ elements and $\cR$ is generated by the relations $\sum_{\zeta\in H}\zeta\=0$ where $H$ varies through all non-trivial subgroups of $\mu_n$.

\begin{lemma}\label{lemma-funn-is-of-indefinite-characteristic}
 Let $G$ be an abelian semigroup and $n\geq1$. Then $\Funn[G]$ has all potential characteristics but $1$ unless $n=1$, in which case $\Fun[G]$ is of indefinite characteristic.
\end{lemma}

\begin{proof}
 Since there is a morphism $\Funn[G]\to\Funn$ that maps all elements of $G$ to $1$, it suffices to show that $\Funn$ is of indefinite characteristic.  Let $\zeta_n$ be a primitive root of unity. Then $\Funn$ embeds into $\Q[\zeta_n]$ and thus $0$ is a potential characteristic of $\Funn$. Let $p$ be a prime that does not divide $n$. Then $\Funn$ embeds into the algebraic closure $\overline{\F_p}$of $\F_p$, and $p$ is a potential characteristic of $\Funn$. 

 The last case is that $p$ is a prime that divides $n$. Then we can define a unique multiplicative map $f:\Funn\to\overline{\F_p}$ whose kernel consists of those $\zeta\in\Funn$ whose multiplicative order is divisible by $p$. We have to verify that this map induces a map between the pre-additions. It is enough to verify this on generators of the pre-addition of $\Funn$. Let $H$ be a non-trivial subgroup of $\mu_n$ whose order is not divisible by $p$. Then $H$ is mapped injectively onto the non-trivial subgroup $f(H)$ of $\overline{\F_p}^\times$, and we have $\sum_{\zeta\in H}f(\zeta)=\sum_{\zeta'\in f(H)}\zeta'=0$ in $\overline{\F_p}$. If $H$ is a subgroup of $\mu_n$ whose order is divisible by $p$, then the kernel of the restriction $f:H\to \overline{\F_p}^\times$ is of some order $p^k$ with $k\geq1$. Thus
 \[
  \sum_{\zeta\in H} f(\zeta) \quad = \quad \sum_{\zeta'\in f(H)} (\underbrace{\zeta'+\dotsb+\zeta'}_{p^k\text{-times}}) \quad = \quad 0.
 \]
 This shows that $f:\Funn\to\overline{\F_p}$ is a morphism of blueprints and that $p$ is a potential characteristic of $\Funn$. If $n\neq 1$, then $\F_{1^n}^+=\Z[\zeta_n]$ is with $-1$ where $\zeta_n$ is a primitive $n$-th root of unity. Thus there is no blueprint morphism $\Funn\to k$ into a semi-field of characteristic $1$ unless $n=1$.
\end{proof}

To conclude this section, we transfer the terminology from algebra to geometry.

\begin{df}
 Let $X$ be a blue scheme, $x$ a point of $X$ and $\kappa(x)$ be the residue field of $x$. The \emph{(arithmetic) characteristic $\kar(x)$ of $x$} is the arithmetic characteristic of $\kappa(x)$. We say that $p$ is a \emph{potential characteristic of $x$} if $p$ is a potential characteristic of $\kappa(x)$, and we say that $x$ is of \emph{mixed} or of \emph{indefinite characteristics} if $\kappa(x)$ is so.
\end{df}

By a \emph{monoidal scheme}, we mean a $\cM_0$-scheme in the sense of \cite{blueprints1}. A monoidal scheme is characterized by its coordinate blueprints, which are blueprints with trivial pre-addition. 

\begin{cor}\label{cor-points-in-monoidal-schemes-are-of-indefinite-characteristic}
 Let $X$ be a monoidal scheme. Then every point of $X$ is of indefinite characteristic.
\end{cor}

\begin{proof}
 This follow immediately from Lemma \ref{lemma-funn-is-of-indefinite-characteristic} since the residue field of a point in a monoidal scheme is of the form $\Fun[G]$ for some abelian group $G$.
\end{proof}

\begin{ex}
 We give two examples to demonstrate certain effects of potential characteristics under specialization. Let $B_1=\bpgenquot{\Fun[T]}{T+T\=0}$. Then $B_1$ has two prime ideals $x_0=(0)$ and $x_T=(T)$. The residue field $\kappa(x_0)=\bpgenquot{\Fun[T^{\pm1}]}{T+T\=0}\simeq\F_2[T^{\pm1}]$ has only potential characteristic $2$ since $1+1\=T^{-1}(T+T)\=0$, while the residue field $\kappa(x_T)=\Fun$ is of indefinite characteristic.

 The blueprint $B_2=\bpgenquot{\Fun[T]}{1+1=T}$ has also two prime ideals $x_0=(0)$ and $x_T=(T)$. The residue field $\kappa(x_0)=\bpgenquot{\Fun[T^{\pm1}]}{1+1\=T}$ has all potential characteristics except for $2$ since $1+1\=T$ is invertible, while the residue field $\kappa(x_T)=\bpgenquot{\Fun}{1+1\=0}= \F_2$ has only characteristic $2$. We illustrate the spectra of $B_1$ and $B_2$ together with their residue fields in Figure \ref{figure: char of points}.
\begin{figure}[h]
 \begin{center}
  \includegraphics{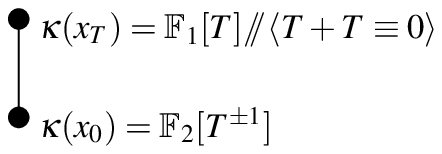} \hspace{1cm} \includegraphics{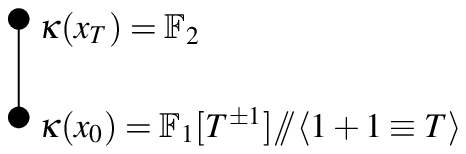}
  \caption{The spectra of $B_1$ and $B_2$ together with their respective residue fields}
  \label{figure: char of points}
 \end{center} 
\end{figure}
\end{ex}


\subsection{Fibres and image of morphisms from (semiring) schemes}
\label{subsection:fibres-of-morphisms}

The fibre of a morphism $\varphi:Y\to X$ of schemes over a point $x\in X$ is defined as the fibre product $\{x\}\Stimes_X Y$. The canonical morphism $\{x\}\Stimes_X Y\to Y$ is an embedding of topological spaces. In this section, we extend this result to blue schemes. Recall from \cite[Prop.\ 3.27]{blueprints1} that the category of blue schemes contains fibre products.


Let $\varphi:Y\to X$ be a morphism of blue schemes and $x\in X$. The \emph{fibre of $\varphi$ over $x$} is the blue scheme $\varphi^{-1}(x)=\{x\}\Stimes_X Y$ and the \emph{topological fibre of $\varphi$ over $x$} is the subspace $\varphi^{-1}(x)^\top=\{y\in Y|\varphi(y)=x\}$ of $Y$. The following lemma justifies the notation since $\varphi^{-1}(x)^\top$ is indeed canonically homeomorphic to the underlying topological space of $\varphi^{-1}(x)$.

\begin{lemma}\label{lemma-algebraic-and-topological-fibre-of-a-morphism}
 Let $\varphi:Y\to X$ be a morphism of blue schemes and $x\in X$. Then the canonical morphism $\varphi^{-1}(x)\to Y$ is a homeomorphism onto $\varphi^{-1}(x)^\top$.
\end{lemma}

\begin{proof}
 Since the diagram
 \[
  \xymatrix{\varphi^{-1}(x)\ar[rr]\ar[d] && Y\ar[d]^\varphi\\ \{x\}\ar[rr] && X}
 \]
 commutes, the image of $\varphi^{-1}(x)\to Y$ is contained in $\varphi^{-1}(x)^\top$. Given a point $y\in \varphi^{-1}(x)^\top$, consider the canonical morphism $\Spec\kappa(y)\to Y$ with image $y$, and the induced morphism $\Spec\kappa(y)\to\Spec\kappa(x)$ of residue fields, which has image $\{x\}\subset X$. The universal property of the tensor product implies that both morphisms factorize through a morphism $\Spec\kappa(y)\to\varphi^{-1}(x)$, which shows that the canonical map $\varphi^{-1}(x)\to\varphi^{-1}(x)^\top$ is surjective.

 We have to show that $\varphi^{-1}(x)\to\varphi^{-1}(x)^\top$ is open. Since this is a local question, we may assume that $X=\Spec B$ and $Y=\Spec C$ are affine blue schemes with coordinate blueprints $B$ and $C$. Then $\varphi^{-1}(x)=\Spec(\kappa(x)\otimes_BC)$ and $\kappa(x)=S^{-1}B/\fp(S^{-1}B)$ where $S=B-\fp$ and $\fp=x\in\Spec B$. Let $f=\Gamma(\varphi,X): B\to C$. Then
 \begin{multline*}
  \kappa(x) \ \otimes_B \ C \quad = \quad \bigl(  S^{-1}B \ / \ \fp(S^{-1}B) \bigr) \ \otimes_B \ C \quad \simeq \quad \bigl(  S^{-1}B \ / \ \fp(S^{-1}B)  \bigr) \ \otimes_{S^{-1}B} \ S^{-1}B \ \otimes_B \ C \\
      \simeq \quad \bigl(  S^{-1}B \ / \ \fp(S^{-1}B)  \bigr) \ \otimes_{S^{-1}B}  \ f(S)^{-1} C \quad \simeq \quad f(S)^{-1}C \ / \ f(\fp)(f(S)^{-1}C),
 \end{multline*}
 which is the quotient of a localization of $C$. Note that the last two isomorphisms follow easily from the universal property of the tensor product combined with the universal property of localizations and quotients, completely analogous to the case of rings. This proves that $\varphi^{-1}(x)\to Y$ is a topological embedding. 
\end{proof}

\begin{prop}\label{prop-fibres-of-alpha}
 Let $X$ be a blue scheme and $x\in X$. Let $\alpha_X:X^+\to X$ and $\beta_X:X^+_\Z\to X$ be the canonical morphisms. Then the canonical morphisms $\Spec\kappa(x)^+\to\alpha(x)^{-1}$ and $\Spec\kappa(x)^+_\Z\to\beta(x)^{-1}$ are isomorphisms. 
\end{prop}

\begin{proof}
 We prove the proposition only for $\alpha_X$. The proof for $\beta_X$ is completely analogous. Since the statement is local around $x$, we may assume that $X$ is affine with coordinate blueprint $B$, and $x=\fp$ is a prime ideal of $B$. Then $X^+=\Spec B^+$, and we have to show that the canonical map $\kappa(x)\otimes_B B^+\to \kappa(x)^+$ is an isomorphism.  Note that the canonical map $B\to B^+$ is injective, so we may consider $B$ as a subset of $B^+$. Let $S=B-\fp$. The same calculation as in the proof of Lemma \ref{lemma-algebraic-and-topological-fibre-of-a-morphism} shows that 
 \[
  \kappa(x)\otimes_B B^+ \quad \simeq \quad S^{-1}B^+ \ / \ \fp(S^{-1}B^+) \quad \simeq \quad S^{-1}B^+ \ / \ (\fp (S^{-1}B))^+.
 \]
 Recall from \cite[Lemma 2.18]{blueprints1} that $B^+_\Z/I^+_\Z\simeq (B/I)^+_\Z$ where $I$ is an ideal of $B$ and $I^+_\Z$ is the ideal of $B^+_\Z$ that is generated by the image of $I$ in $B^+_\Z$. In the same way it is proven for a blueprint that $B^+/I^+\simeq (B/I)^+$ where $I^+$ is the ideal of $B^+$ that is generated by the image of $I$ in $B^+$. We apply this to derive
 \[
  S^{-1}B^+ \ / \ (\fp (S^{-1}B))^+ \quad \simeq \quad \Big( \ S^{-1}B \ / \ \fp(S^{-1}B) \ \Bigr)^+ \quad = \quad \kappa(x)^+,
 \]
 which finishes the proof of $\Spec\kappa(x)^+\simeq \alpha(x)^{-1}$.
\end{proof}

The potential characteristics of the points of a blue scheme are closely related to the fibres of the canonical morphism from its semiring scheme as the following lemma shows.

\begin{lemma}\label{lemma-geom-char-and-the-fibres-of-alpha}
 The canonical morphism $\alpha_X:X^+\to X$ is surjective. The potential characteristics of a point $x\in X$ correspond to the potential characteristics of the points $y$ in the fibre of $\alpha_X$ over $x$.
\end{lemma}

\begin{proof}
 The morphism $\alpha_X:X^+\to X$ is surjective for the following reason. The canonical morphism $B\to B^+$ is injective. In particular, $\kappa(x)\to\kappa(x)^+$ is injective for every point $x$ of $X$. This means that $\kappa(x)^+$ is non-trivial, and thus $\alpha^{-1}(x)\simeq\kappa(x)^+$ non-empty. This shows the first claim of the lemma.

 If $x=\alpha(y)$ for some $y$ in the fibre of $\alpha_X$ over $x$, then there is a morphism $\kappa(x)\to\kappa(y)$ between the residue fields, and the potential characteristics of the semifield $\kappa(y)$ are potential characteristics of the blue field $\kappa(x)$. On the other hand, if $\kappa(x)\to k$ is a map into a semifield $k$ of characteristic $p$, then this defines a morphism $\Spec k\to X$ with image $x$, which factors through $X^+$. Thus the map $\kappa(x)\to k$ factors through $\kappa(x)\to \kappa(y)$ for some $y$ in the fibre of $\alpha_X$ over $x$. Thus the latter claim of the lemma.
\end{proof}

\begin{rem}
 By the previous lemma, every point $x$ of a blue scheme $X$ lies in the image of some $\alpha_{X,k}:X^+\Stimes_\N k\to X$ where $k$ is a semifield, which can be chosen to be an algebraically closed field if it is not an idempotent semifield. This shows that the geometry of a blue scheme is dominated by algebraic geometry over algebraically closed fields and idempotent geometry, by which I mean geometry that is associated to idempotent semirings. There are various (different) viewpoints on this: idempotent analysis as considered by Kolokoltsov and Maslov, et al.\ (see, for instance, \cite{Maslov94}), tropical geometry as considered Itenberg, Mikhalkin, et al.\ (see, for instance, \cite{Gathmann06}, \cite{Mikhalkin06} and, in particular, \cite[Chapter 2]{Mikhalkin10}) and idempotent geometry that attempts to mimic $\Fun$-geometry (see \cite{CC09}, \cite{Lescot09} and \cite{Takagi10}). These theories might find a common background in the theory of blue schemes.
\end{rem}

\begin{lemma} \label{lemma: quotients of cancellative blueprints are cancellative}
 Let $B$ be a cancellative blueprint and $I\subset B$ be an ideal of $B$. Then the quotient $B/I$ is cancellative.
\end{lemma}

\begin{proof}
 We first establish the following claim: two elements $a,b\in B$ define the same class $\overline a\=\overline b$ in $B/I$ if and only if there are elements $c_k,d_l\in I$ such that $a+\sum c_k\=b+\sum d_l$ in $B$. Per definition, $\overline a\=\overline b$ if and only if there is a sequence of the form
 \[
  a \ \= \ \sum c_{1,k} \ \sim^I_\N \ \sum d_{1,k} \ \= \ \sum c_{2,k} \ \sim^I_\N \quad \dotsb \quad \sim^I_\N \ \sum d_{n,k} \ \= \ b 
 \]
 where $\sum c_k\sim^I_\N\sum d_k$ if for all $k$ either $c_k=d_k$ or $c_k,d_k\in I$ (cf.\ \cite[Def.\ 2.11]{blueprints1}). If we add up all additive relations in this sequence, we obtain 
 \[
  a+\sum c_{i,k} \ \= \ b+\sum d_{i,k}.
 \]
 Since $B$ is cancellative, we can cancel all terms $c_{i,k}\=d_{i,k}$ that appear on both sides, and stay over with a relation of the form
 \[
  a+\sum\tilde c_k\ \= \ b+\sum \tilde d_k
 \]
 with $\tilde c_k,\tilde d_k\in I$. This shows one direction of the claim. To prove the reverse direction, consider a relation of the form $a+\sum c_k\ \= \ b+\sum d_l$ with $c_k, d_l\in I$. Then we have
 \[
  a \ \= \ a +\sum 0 \ \sim_\N^I \ a+\sum c_k \ \= \ b+\sum d_l \ \sim^I_\N \ b+\sum 0 \ \= \ b,
 \]
 which shows that $\overline a\=\overline b$ in $B/I$.
 
 With this fact at hand, we can prove that $B/I$ is cancellative. Consider a relation of the form
 \[
  \sum a_i + c_0 \ \= \ \sum b_j +d_0
 \]
 in $B$ where $\overline c_0\=\overline d_0$ in $B/I$. We have to show that $\sum\overline a_i\=\sum\overline b_j$ in $B/I$. By the above fact, $\overline c_0\=\overline d_0$ if and only if there are $c_k,d_l\in I$ such that $c_0+\sum c_k\=d_0+\sum d_l$. Adding this equation to the above equation, with left and right hand side reversed, yields
 \[
  \sum a_i + c_0 + d_0+\sum d_l \ \= \ \sum b_j +d_0 +c_0+\sum c_k.
 \]
 Since $B$ is cancellative, we can cancel the term $c_0+d_0$ on both sides and obtain the sequence
 \[
  \sum\overline a_i \ \= \sum\overline a_i + \sum\overline d_l \ \= \ \sum\overline b_j + \sum\overline c_k \ \= \ \sum\overline b_j
 \]
 in $B/I$, which proves that $B/I$ is cancellative.

\end{proof}

\begin{lemma}\label{lemma-beta-surjective-for-cancellative-scheme}
 If $X$ is cancellative, then the canonical morphism $\beta_X:X^+_\Z\to X$ is surjective.
\end{lemma}

\begin{proof}
 This is a local question, so we may assume that $X=\Spec B$ for a cancellative blueprint $B$. Since localizing preserves cancellative blueprints (see \cite[section 1.13]{blueprints1}), $B_\fp$ is also cancellative for every prime ideal $\fp$ of $B$. The residue field at $x=\fp$ is $\kappa(x)=B_\fp/\fp B_\fp$, which is cancellative by the Lemma \ref{lemma: quotients of cancellative blueprints are cancellative}. This it is a subblueprint of the (non-zero) ring $\kappa(x)_\Z^+$. Thus the canonical morphism $\Spec \kappa(x)_\Z^+\to\Spec\kappa(x)\hookrightarrow X$ has image $\{x\}$ and factors through $X_\Z^+$ by the universal property of the scheme $X_\Z^+$.
\end{proof}

\begin{rem}
 Every sesquiad (see \cite{Deitmar11}) can be seen as a cancellative blueprint. A prime ideal of a sesquiads is the intersection of a prime ideal of the prime ideal of its universal ring with the sesquiad. The previous lemma shows that the sesquiad prime ideals coincide with its blueprint prime ideals.
\end{rem}

While the points of potential characteristic $p\neq 1$ are governed by usual scheme theory, the points of potential characteristic $1$ in a fibre $\alpha^{-1}(x)$ are of a particularly simple shape.

\begin{lemma}\label{lemma: points of potential characteristic 1 in fibres of alpha}
 Let $x\in X$ be a point with potential characteristic $1$. Then $\alpha^{-1}_X(x)$ is irreducible with generic point $\eta$, which is the only point of $\alpha^{-1}(x)$ with potential characteristic $1$. If $X$ is cancellative, then $\eta$ has also potential characteristic $0$.
\end{lemma}

\begin{proof}
 Let $\kappa$ be the residue field of $x$. Since $x$ has potential characteristic $1$, there is a morphism into a semifield $k$ of characteristic $1$. By the universal property of $\kappa\to\kappa^+$, this morphism factors through a morphism $f:\kappa^+\to k$. Every element of $\kappa^+$ is of the form $\sum a_i$ where $a_i$ are units of $\kappa$. 

 Consider the case that $f(\sum a_i)=0$. Unless the sum is trivial, it is of the form $a+\sum a'_j$ for some unit $a$ of $\kappa$. Then $b'=f(\sum a'_j)$ is an additive inverse of $b=f(a)$ in $k$. Since units are mapped to units, $b$ is a unit of $k$ and therefore $1=b^{-1}b$ has the additive inverse $-1=b^{-1}b'$. But this is not possible in a semifield of characteristic $1$. Therefore we conclude that $\sum a_i$ has to be the trivial sum and that the kernel of $f:\kappa^+\to k$ is $0$. This shows that $0$ is a prime ideal, that $\alpha^{-1}(x)$ is irreducible with generic point $\eta=0$ and that $\eta$ is the only point of $\alpha^{-1}(x)$ with potential characteristic $1$.

 By Lemma \ref{lemma-beta-surjective-for-cancellative-scheme}, the canonical morphism $\beta:X^+_\Z\to X^+$ is surjective if $X^+$ is cancellative, which is the case if $X$ is cancellative. Therefore every point $x$ of $X^+$ with potential characteristic $1$ has at least one other potential characteristic, which must be $0$ since $\kappa(x)$ is a semifield without $-1$ (cf.\ Lemma \ref{lemma-semifields-have-1-or-2-geometric-char}). This proves the last claim of the lemma.
\end{proof}


\subsection{The topology of fibre products}
\label{subsection:products}

In this section, we investigate the topological space of the product of two blue schemes. The canonical projections of the fibre product of blue schemes are continuous, and thus induce a universal continuous map into the product of the underlying topological spaces. In contrast to the product of two varieties over an algebraically closed field, which surjects onto the product of the underlying topological spaces, the product of two blue schemes injects into the product of the underlying topological spaces.


\begin{prop}\label{prop: tau is an embedding}
 Let $X_1\to X_0$ and $X_2\to X_0$ be morphisms of blue schemes. Then the canonical map
 \[
  \tau: \quad X_1 \ \times_{X_0} \ X_2 \quad \longrightarrow \quad X_1 \ \toptimes_{X_0} \ X_2
 \]
 is an embedding of topological spaces.
\end{prop}

\begin{proof}
 We have to show that $\tau$ is a homeomorphism onto its image. Since the claim of the proposition is a local question, we may assume that $X_i=\Spec B_i$ are affine with $B_i=\bpquot{A_i}{\cR_i}$. Then there are morphisms $j_1:B_0\to B_1$ and $j_2:B_0\to B_2$, and we have $X_1 \times_{X_0} X_2 =\Spec B_1\otimes_{B_0}B_2$ with 
 \[
  B_1 \ \otimes_{B_0} \ B_2 \quad = \quad \bpgenquot{A_1\times A_2}{\ \cR_1\times\{1\},\ \{1\}\times\cR_2,\ (a_0a_1,a_2)\=(a_1,a_0a_2) \ | \ a_i\in B_i \ }. 
 \]
 Note that this is not a proper representation of $B_1\otimes_{B_0}B_2$. Since we are only concerned with topological properties of $\Spec B_1 \otimes_{B_0} B_2$, this is legitimate (cf.\ Section \ref{subsection:notation-and-conventions}).

 We begin to show injectivity of $\tau$. Let $\fp$ be a prime ideal of $B$. Then $\tau(\fp)=(\fp_0,\fp_1,\fp_2)$ where $\fp_i=\iota_i^{-1}(\fp)$ is a prime ideal of $B_i$ and $\iota_i:B_i\to B_1\otimes_{B_0}B_2$ is the canonical map that sends $a$ to $j_1(a)\tensor 1=1\tensor j_2(a)$ if $i=0$,  that sends $a$ to $a\tensor 1$ if $i=1$ and that sends $a$ to $1\tensor a$ if $i=2$. Since for $a_1\tensor a_2=(a_1\tensor 1)\cdot (1\tensor a_2)\in\fp$ either $a_1\tensor 1\in\fp\cap\iota_1(B_1)$ or $1\tensor a_2\in\fp\cap\iota_2(B_2)$, the prime ideal $\fp$ equals the set $\{a_1\tensor a_2| a_i\in\fp_i\}$. Thus $\fp$ is uniquely determined by $\tau(\fp)$.

 We show that $\tau$ is a homeomorphism onto its image. Given a basic open $U=U_{a_1}\toptimes_{X_0}U_{a_2}$ of $X_1 \toptimes_{X_0} X_2$ where $a_i\in B_i$ and $U_{a_i}=\{\fp_i\in X_i|a_i\notin\fp_i\}$ is the according basic open of $X_i$ for $i=1,2$. Then 
 \begin{multline*}
  \tau^{-1}(U) \  = \ \{ \ \fp \ \in \ X_1 \ \times_{X_0} \ X_2 \ | \ a_1\tensor 1\notin\fp\text{ and }1\tensor a_2\notin\fp \ \}  \\
    = \ \{ \ \fp \ \in \ X_1 \ \times_{X_0} \ X_2 \ | \ a_1\tensor a_2\notin\fp \ \} \ = \ U_{a_1\tensor a_2},
 \end{multline*}
 which is a basic open of $X_1 \times_{X_0} X_2$. Thus $\tau$ is continuous. Since every basic open of $X_1 \times_{X_0} X_2$ is of the form $U_{a_1\tensor a_2}$ for some $a_1\in B_1$ and $a_2\in B_2$, $\tau$ is is indeed a homeomorphism onto its image.
\end{proof}

Therefore, we can regard $X_1 \times_{X_0} X_2$ as a subspace of $X_1 \toptimes_{X_0} X_2$, and we denote a point $x$ of $X_1 \times_{X_0} X_2$ by the coordinates $(x_1,x_2)$ of $\tau(x)$ where $x_1\in X_1$ and $x_2\in X_2$. 

In the rest of this section, we investigate the image of $\tau$ in the case $X_0=\Spec\Fun$.

\begin{lemma}\label{lemma: potential characteristics of b_1 tensor b_2}
 Let $B_1$ and $B_2$ be blueprints. Then $p$ is a potential characteristic of $B_1\otimes_\Fun B_2$ if and only if $p$ is a potential characteristic of both $B_1$ and $B_2$. Consequently, $B_1\otimes_\Fun B_2=\{0\}$ if and only if $B_1$ and $B_2$ have no potential characteristic in common.
\end{lemma}

\begin{proof}
 Since there are canonical maps $B_i\to B_1\otimes_\Fun B_2$ for $i=1,2$, every potential characteristic of $B_1\otimes_\Fun B_2$ is a potential characteristic of both $B_1$ and $B_2$. 

 Conversely, let $p$ be a common potential characteristic of $B_1$ and $B_2$. In case $p\neq 1$, there are morphisms $B_i\to k_i$ into fields $k_1$ and $k_2$ of characteristic $p$. The compositum $k$ of $k_1$ and $k_2$ is a field of characteristic $p$ that contains $k_1$ and $k_2$ as subfields. This yields morphisms $f_i: B_i\to k$ and thus a morphism $f:B_1\otimes_\Fun B_2\to k$ (note that there is a unique map $\Fun\to k$, which factorizes through $f_1$ and $f_2$). Thus $p$ is a potential characteristic of $B_1\otimes_\Fun B_2$.

 If $p=1$, then $B_1$ and $B_2$ are both without $-1$, and there are morphisms $f_i:B_i\to \B_1$ by Lemma \ref{lemma-semifields-have-1-or-2-geometric-char}. Therefore there is a morphism $B_1\otimes_\Fun B_2\to\B_1$, and $1$ is a potential characteristic of $B_1\otimes_\Fun B_2$.
%

 This shows in particular that $B_1\otimes_\Fun B_2\neq \{0\}$ if $B_1$ and $B_2$ have a potential characteristic in common. If there is no morphism $B_1\otimes_\Fun B_2\to k$ into a semifield $k$, then there is no morphism $B_1\otimes_\Fun B_2\to \kappa$ into any blue field $\kappa$. This means that $\Spec B_1\otimes_\Fun B_2$ is the empty scheme and $B_1\otimes_\Fun B_2$ is $\{0\}$.
\end{proof}

\begin{ex}
 While $B_1$ and $B_2$ possess all potential characteristics of $B_1\otimes_{\kappa_0} B_1$ for an arbitrary blue field $\kappa_0$, the contrary is not true in general. 

 For instance, consider the tensor product $\Funsq\otimes_{\Fun[T^{\pm1}]}\Funsq$ with respect to the two morphisms $f_1:\Fun[T^{\pm1}]\to\Funsq$ with $f_1(T)=1$ and $f_2:\Fun[T^{\pm1}]\to\Funsq$ with $f_2(T)=-1$. We have that $\Funsq\otimes_\Fun\Funsq=(\Funsq)_\inv=\Funsq$, which is represented by $\{0\tensor0,1\tensor1,1\tensor(-1)\}$. The tensor product $\Funsq\otimes_{\Fun[T^{\pm1}]}\Funsq$ is a quotient of $\Funsq$, and we have 
 \[
  1\tensor(-1) \ = \ 1\tensor(1\cdot f_2(T)) \ = \ (1\cdot f_1(T))\tensor1 \ = \ 1\tensor1.
 \]
 Thus $1\tensor1$ is its own additive inverse and $\Funsq\otimes_{\Fun[T^{\pm1}]}\Funsq=\F_2$, the field with two elements. While $\Funsq$ and $\Fun[T^{\pm1}]$ have both indefinite characteristic, $\F_2$ has characteristic $2$.
\end{ex}

\begin{thm} \label{thm-image-of-tau-in-x_1-times-x_2}
 Let $X_1$ and $X_2$ be blue schemes. Then the embedding $\tau:X_1\times_\Fun X_2\to X_1 \toptimes X_2$ is a homeomorphism onto the subspace
 \[
  \{ \ (x_1,x_2)\in X_1 \toptimes X_2 \ | \ x_1\text{ and }x_2\text{ have a common potential characteristic}\ \}.
 \]
\end{thm}

\begin{proof}
 Note that since the underlying topological space of $X_0=\Spec\Fun$ is the one-point space, we have $X_1\toptimes_{X_0}X_2=X_1\toptimes X_2$. Since $\tau$ is an embedding (cf.\ \ref{prop: tau is an embedding}), we have only to show that the image of $\tau$ is as described in the theorem.

 Let $x=(x_1,x_2)\in X_1\toptimes X_2$. Write $\kappa_i$ for $\kappa(x_i)$ and $\kappa$ for $\kappa_1\otimes_\Fun\kappa_2$. If $x\in X_1\times_\Fun X_2$, then $\kappa=\kappa(x)$. The canonical morphism $\kappa\to\kappa^+$ witnesses that $p=\kar\kappa^+$ is a potential characteristic of $\kappa$. By Lemma \ref{lemma: potential characteristics of b_1 tensor b_2}, the potential characteristics of $\kappa$ (or, equivalently, $x$) correspond to the potential characteristic of $\kappa_1$ and $\kappa_2$ (or, equivalently, $x_1$ and $x_2$). Therefore, $x_1$ and $x_2$ have a potential characteristic in common.

 If, conversely, $x_1$ and $x_2$ have a common potential characteristic $p$, then $p$ is also a potential characteristic of $\kappa$ by Lemma \ref{lemma: potential characteristics of b_1 tensor b_2}. This means that there exists a morphism $\kappa\to k$ into a semifield $k$. The morphism $\Spec k\to \Spec\kappa$ has image $x=(x_1,x_2)$, and thus $(x_1,x_2)\in X_1\times_\Fun X_2$. 
\end{proof}

For later reference, we state the following fact, which follows from the local definition of the fibre product. We use the shorthand notation $\Gamma X$ for the global sections $\Gamma(X,\cO_X)$ of $X$.

\begin{lemma}\label{lemma: global sections of fibre products}
 Let $X\to Z$ and $Y\to Z$ be two morphisms of blue schemes. Then 
 \[
  \Gamma (X\times_ZY) \quad \simeq \quad \Gamma X \ \otimes_{\Gamma Z} \ \Gamma Y.
 \]
 \vspace{-40pt}\\ \ \qed
\end{lemma}


\subsection{Relative additive closures}
\label{subsection: relative additive closures}

Let $f:B\to C$ be a morphism. The \emph{additive closure of $B$ in $C$ w.r.t.\ to $f$} is the subblueprint
\[
 f^+(B) \quad = \quad \{ \ c\in C \ | \ c\=\sum f(a_i)\text{ for }a_i\in B\}
\]
of $C$. Note that this is indeed a subblueprint of $C$ since for $c,d\in f^+(B)$, i.e.\ $c\=\sum f(a_i)$ and $d\=\sum f(b_j)$, the product $cd=\sum f(a_ib_j)$ is an element of $f^+(B)$. 

If $B$ is a subblueprint of $C$ and $\iota:B\hookrightarrow C$ the inclusion, then we call $\iota^+(B)$ briefly the \emph{additive closure of $B$ in $C$}. The subblueprint $B$ is \emph{additively closed in $C$} if $B=\iota^+(B)$.

We list some immediate properties of relative additive closures. Let $f:B\to C$ be a blueprint morphism. Then $f^+(B)$ is additively closed in $C$. More precisely, $f^+(B)$ is the smallest additively closed subblueprint $B'$ of $C$ such that the morphism $f:B\to C$ factors through $B'\hookrightarrow C$. If $C$ is a semiring, then $f^+(B)$ is isomorphic to the universal semiring $f(B)^+$ associated with $f(B)$ (considered as a subblueprint of $C$).

\begin{lemma}\label{lemma: relative additive closure of a morphism}
 For any commutative diagram
 \[
  \xymatrix@R=0,6pc@C=3pc{ & B\ar[rr]^g\ar[dd]_f &&  \tilde B \ar[dd]^{\tilde f}\\ \\ & C\ar[rr]^{\tilde g} && \tilde C,}
 \]
 there exists a unique blueprint morphism $f^+(g):f^+(B)\to f^+(\tilde B)$ such that the diagram
 \[
  \xymatrix@R=1pc@C=3pc{ & B\ar[rr]^g\ar[dd]|\hole_(0.65)f\ar[dl] &&  \tilde B \ar[dd]^{\tilde f}\ar[dl] \\  f^+(B) \ar[rr]^(0,65){f^+(g)}\arincl[dr] && \tilde f^+(\tilde B) \arincl[dr]\\ & C\ar[rr]^{\tilde g}  && \tilde C,}
 \]
 commutes.
\end{lemma}

\begin{proof}
 The uniqueness of $f^+(g)$ follows from the injectivity of $\tilde f^+(\tilde B)\hookrightarrow \tilde C$. For $c\in f^+(B)$ define $f^+(g)(c)$ as $\tilde g(c)$, which is a priori an element of $\tilde C$. Since $c\=\sum f(a_i)$ for certain $a_i\in B$, we have that $\tilde g(c)\=\sum \tilde g(f(a_i)) \=\sum \tilde f(g(a_i))$, thus $\tilde g(C)$ is indeed an element of $\tilde f^+(\tilde B)$. This shows that $f^+(g):f^+(B)\to f^+(\tilde B)$ is a blueprint morphism with the desired property.
\end{proof}

Let $B$ be a blueprint and $\iota:B\to B_\inv$ the base extension from $\Fun$ to $\Funsq$ where we write $B_\inv=B\otimes_\Fun\Funsq$. The \emph{inverse closure of $B$} is the subblueprint $\hat B=\iota^+(B)$ of $B_\inv$. A blueprint $B$ is \emph{inverse closed} if $B\simeq \hat B$. 

Note that since $(B_\inv)_\inv= B_\inv$, the inverse closure $\hat B$ of $B$ is inverse closed. The previous lemma extends the association $B\mapsto \hat B$ naturally to a functor $(\blanc)^\hatexp:\bp\to \bp$ whose essential image are the inverse closed blueprints. Further note that the inverse closure $\hat B$ of $B$ equals the intersection of $B_\inv$ with $B^+_\canc$ inside $B^+_\Z$. In particular note that $\hat B$ is cancellative.

\begin{lemma} \label{lemma: tensor product of inverse closures}
 Let $B_1$ and $B_2$ be blueprints. Then $(B_1\otimes_\Fun B_2)^\hatexp \simeq\hat B_1\otimes_\Fun\hat B_2$.
\end{lemma}

\begin{proof}
 Since $(B_1\otimes_\Fun B_2)_\canc=B_{1,\canc}\otimes_\Fun B_{2,\canc}$ and $\hat B=(B_\canc)^\hatexp $, we can assume that $B_1$ and $B_2$ are cancellative. Therefore, we can consider $(B_1\otimes_\Fun B_2)^\hatexp $ and $\hat B_1\otimes_\Fun\hat B_2$ as subblueprints of $(B_1\otimes_\Fun B_2)_\inv=B_{1,\inv}\otimes_\Fun B_{2,\inv}$ that both contain $B_1\otimes_\Fun B_2$ as a subblueprint.
 
 Let $a\otimes b$ be an element of $B_1\otimes_\Fun B_2$. we have to show that its additive inverse $-(a\otimes b)$ is contained in $(B_1\otimes_\Fun B_2)^\hatexp $ if and only if it is contained in $\hat B_1\otimes_\Fun\hat B_2$.

 Assume that $-(a\otimes b)$ is contained in $(B_1\otimes_\Fun B_2)^\hatexp $. Then there is an additive relation of the form $a\otimes b+\sum c_k\otimes d_k\=0$ in $B_1\otimes_\Fun B_2$. By the definition of the tensor product $B_1\otimes_\Fun B_2$, this must come from an additive relation of the form $a+\sum \tilde c_k\=0$ in $B_1$ or an additive relation of the form $b+\sum \tilde d_k\=0$ in $B_2$. Thus $-a\in \hat B_1$ or $-b\in\hat B_2$. In either case, $(-a)\otimes b=-(a\otimes b)=a\otimes (-b)$ is an element of $\hat B_1\otimes_\Fun\hat B_2$.

 Assume that $-(a\otimes b)$ is contained in $\hat B_1\otimes_\Fun\hat B_2$. By symmetry of the argument, we may assume that $-a\in\hat B_1$, i.e.\ we have an additive relation $a+\sum c_k\=0$ in $B_1$. Thus $a\otimes b+\sum c_k\otimes b$ in $B_1\otimes B_2$, which shows that $-(a\otimes b)$ is an element of $(B_1\otimes_\Fun B_2)^\hatexp $.
\end{proof}

If $Z=\Spec B$ is an affine blue scheme, then we define $\hat Z=\Spec\hat B$. It comes together with a morphism $\gamma_Z:\hat Z\to Z$ induced by the blueprint morphism $B\to \hat B$. 

\begin{rem}
 The inverse closure $(\blanc)^\hatexp $ of a blueprint does not behave well with localizations. It seems that there is no (meaningful) extension of $(\blanc)^\hatexp $ from affine blue schemes to all blue schemes. To illustrate the incompatibility with localizations, consider the subblueprint $B=\Fun[T]$ of $C=\bpgenquot{\Fun[T,S]}{ST+S\=0}$, which is additively closed in $C$. Let $\fq$ be the ideal of $C$ that is generated by $T$. Then $C_\fq=\bpgenquot{\Fun[T,S^{\pm1}]}{ST+S\=0}\simeq\Funsq[S^{\pm1}]$. The additive closure of $B$ in $C_\fq$ (w.r.t.\ the canonical morphism $f:B\hookrightarrow C\to C_\fq$) is $f^+(B)\simeq\Funsq$ while the localization $B_\fp$ at the prime ideal $\fp=\fq\cap B$ of $B$ is equal to $B=\Fun[T]$ itself.
\end{rem}


\subsection{The unit field and the unit scheme}
\label{subsection: unit field and unit scheme}

The units of a ring form naturally a group. In certain cases like polynomial rings over fields or discrete valuation rings of positive characteristics, the unit group together with $0$ forms a field; but, for a general ring, this is not true. However, the unit group together with $0$ and the restriction of the (pre-)addition of the ring has always the structure of a blue field, which leads to the following definition.

Let $B=\bpquot A\cR$ be a blueprint. The \emph{unit field of $B$} is the blue field $B^{\px}=\bpquot{A^\times\cup\{0\}}{\cR^\px}$ where $\cR^\px=\cR\vert_{A^\times\cup\{0\}}$ is the restriction of $\cR$ to the submonoid $A^\times\cup\{0\}$ of $A$. It comes together with a canonical inclusion $u:B^\px\to B$ of blueprints.

Let $X$ be a blue scheme and $B=\Gamma X$ its global sections. By \cite[Lemma 3.25]{blueprints1}, there exists a canonical morphism $X\to\Spec B$ that factors every morphism from $X$ to an affine blue scheme in a unique way. The \emph{unit scheme of $X$} is the blue scheme $X^\px=\Spec B^\px$ together with the morphism
\[
 \upsilon: \ X \ \longrightarrow \ \Spec B \ \stackrel{u^\ast}{\longrightarrow} \ \Spec B^\px \ = \ X^\px.
\]
The blue field $\Fpx(X)=B^\px$ is called the \emph{unit field of $X$}. The unit scheme $X^\px$ consists of one point $\eta$, which is corresponds to the unique prime ideal $\{0\}$ of the unit field $\Fpx(X)$. 

For a point $x$ of $X$, we write $\Fpx(x)$ for the unit field of the reduced closed subscheme $\barx$ of $X$ whose support is the closure of $x$. We call $\Fpx(x)$ the \emph{unit field at $x$}. There is a canonical morphism $\psi:\Fpx(x)\to\Gamma\barx\to\cO_{X,x}\to\kappa(x)$ into the residue field of $x$, which is, in general, neither injective nor surjective. If, however, $X$ is a reduced scheme that consists of only one point $x$, then $\psi:\Fpx(x)\to\kappa(x)$ is an isomorphism. This means, in particular, that 
\[
 \Fpx(X)=\Fpx(X^\px)=\Fpx(\eta)=\kappa(\eta)
\]
where $\eta$ is the unique point of $X^\px$.

Note that since a morphism $f:B\to C$ of blueprints sends $0$ to $0$ and units to units, it induces a morphism $f^\px:B^\px\to C^\px$ between the unit fields. Thus taking the unit field is an idempotent endofunctor of the category of blueprints whose essential image is the full subcategory of blue fields. Similarly, taking the unit scheme of a blue scheme is an idempotent endofunctor of the category of blue schemes. Note further that the category of unit schemes is dual to the category of blue fields since unit schemes are affine. 

A blueprint $B$ is \emph{generated by its units} if $u^+(B^\px)=B$ for $u:B^\px\to B$. This is equivalent to saying that $u^+(B^\px)=B$ induces an isomorphism $u^+:(B^\px)^+\to B^+$ of semirings. A blue scheme $X$ is \emph{generated by its units} if $\upsilon:X\to X^\px$ induces an isomorphism $\upsilon^+:X^+\to (X^\px)^+$ of semiring schemes.


\section{The Tits category}
\label{section: the tits category}

In this section, we will introduce Tits morphisms between blue schemes, which will be the technical core of the theory of Tits-Weyl models of algebraic groups. As a first task, we introduce the rank space of a blue scheme. With this, we are prepared to define Tits morphisms and to investigate their relationship to morphisms (in the usual sense), which we also call \emph{locally algebraic morphisms}.


\subsection{The rank space}
\label{subsection: rank space}

Let $X$ be a blue scheme and $x$ a point of $X$. In the following, we will understand by $\barx$ the closure of $x$ in $X$ together with its structure as a reduced closed subscheme (see Section \ref{subsection: reduced blueprints}). 

\begin{df}
 A point $x$ of $X$ is \emph{pseudo-Hopf} if $x$ is almost of indefinite characteristic, $\barx$ is affine, $\barx_\inv$ is generated by its units and $\barx_\Z^+$ is a flat scheme.
\end{df}

\begin{rem}
 If $x$ is pseudo-Hopf and $F=\Fpx(x)$ is the unit field of $x$, then $\Gamma \barx_\Z^+$ is a quotient of the Hopf algebra $\Z[F^\times]$, namely, by the ideal
 \[
  I \quad = \quad \{ \ \sum a_i-\sum b_j \ | \ \sum a_i\=\sum b_j\text{ in }F \ \}.
 \]
\end{rem}

Recall from Section \ref{subsection: relative additive closures} that for an affine blue scheme $Z=\Spec B$, we have $\hat Z=\Spec\hat B$ together with $\gamma_Z:\hat Z\to Z$. If $x\in X$ is a point such that $\barx$ is affine, then this yields the morphism
\[
 \begin{array}{cccccc}
   \rho_x: & \prerk x & \stackrel{\gamma_\barx}{\longrightarrow} & \barx & \stackrel{\iota_\barx}\longrightarrow & X.
 \end{array}
\]

\begin{df}
 Let $X$ be a blue scheme and $x$ a point of $X$. The \emph{rank $\rk\, x$ of $x$} is the dimension of the scheme $\barx_\Q^+$ over $\Q$.

 Let $X$ be connected. Then the \emph{rank of $X$} is 
 \[
  \rk\,X \quad = \quad \inf\ \{ \ \rk\,x \ | \ x\text{ is pseudo-Hopf}\ \} 
 \]
 if $X$ has a pseudo-Hopf point, and $\rk\,X=0$ otherwise. Let $\cZ(X)$ be the set of all pseudo-Hopf points of $X$ whose rank equals $\rk\,X$. The \emph{pre-rank space of $X$} is 
 \[
  X^\prk \quad = \quad \coprod_{x\in\cZ(X)}\ \prerk x
 \]
 and the \emph{rank space of $X$} is 
 \[
  X^\rk \quad = \quad \coprod_{x\in\cZ(X)}\ \prerk x^\px.
 \]
 If $X=\coprod X_i$ is the disjoint union of connected schemes $X_i$, then 
 \[
  X^\prk \quad = \quad \coprod X_i^\prk \hspace{1cm} \text{and} \hspace{1cm} X^\rk \quad = \quad \coprod X_i^\rk
 \]
 are the pre-rank space and the rank space of $X$.
\end{df}

We describe some immediate consequences of these definitions. The canonical morphisms $\rho_x:\prerk x\to\barx\to X$ define a morphism $\rho_X:X^\prk\to X$ and the canonical morphisms $\upsilon_x:\prerk x\to\prerk x^\px$ into the unit scheme define a morphism $\upsilon_X:X^\prk\to X^\rk$. Thus we obtain for every blue scheme $X$ the diagram
\[
 X^\rk \quad\stackrel{\upsilon_X}{\longleftarrow}\quad X^\prk \quad\stackrel{\rho_X}{\longrightarrow}\quad X.
\]

By the definition of pseudo-Hopf points, $\upsilon_X:X^\prk\to X^\rk$ induces an isomorphism $\upsilon_{X,\Z}^+:X^{\prk,+}_\Z\to X^{\rk,+}_\Z$ of schemes where we use the shorthand notations $X^{\prk,+}_\Z=(X^\prk)^+_\Z$ and $X^{\rk,+}=(X^\rk)^+_\Z$. Thus we obtain a commutative diagram
\[
 \xymatrix@C=4pc{X^{\rk,+}_\Z \ar[d]^{\beta_{X^\rk}}     &   X^{\prk,+}_\Z \ar@{=}[l]\ar[r]^{\rho^+_{X,\Z}}\ar[d]^{\beta_{X^\prk}}   &  X^{+}_\Z \ar[d]^{\beta_{X}} \\
           X^\rk                                         &   X^\prk \ar[l]_{\upsilon_X}\ar[r]^{\rho_X}                                    &  X\ .\hspace{-5pt} } 
\]
In the following, we identify $X^{\rk,+}_\Z$ with $X^{\prk,+}_\Z$ via $\upsilon_{X,\Z}^+$, which allows us to consider $\rho_{X,\Z}^+$ as a morphism from $X^{\rk,+}_\Z$ to $X^+_\Z$. If $\upsilon_X:X^\prk\to X^\rk$ is an isomorphism, then we say that \emph{the rank space $X^\rk$ lifts to $X$} and we may define $\tilde\rho_X:X^\rk\to X$ as $\rho_X\circ\upsilon_X^{-1}$. If additionally $\rho_X$ is a closed immersion, then we say that \emph{the rank space $X^\rk$ embeds into $X$}.

We turn to an investigation of the rank spaces. For this, we introduce the notion of blue schemes of pure rank.

\begin{df}
 A blue scheme $X$ is \emph{of pure rank} if it is discrete and reduced, if all points are pseudo-Hopf and if $\prerk x\to\barx$ is an isomorphism for all $x\in X$. We denote the full subcategory of $\BSch$ whose objects are blue schemes of pure rank by $\rkBSch$.
\end{df}

If $X$ is of pure rank, then every $x\in X$ has all all potential characteristics with the possible exception of $1$ since $\{x\}^+_\Z=\barx^+_\Z$ is a flat non-empty scheme. A scheme of pure rank is cancellative since for every connected component $\{x\}$, the blueprint $\Gamma\prerk x\simeq\Gamma \barx=\Gamma\{x\}$ is cancellative.

\begin{prop}\label{prop: characterisation of rank spaces}\quad 
 \begin{enumerate}
  \item The rank space of a blue scheme is of pure rank. 
  \item If $X$ is a scheme of pure rank, then $X^\rk$ lifts to $X$ and $\tilde\rho_X:X^\rk\to X$ is an isomorphism.
 \end{enumerate}
\end{prop}

\begin{proof}
 We show \eqref{part1}. Let $X$ be a blue scheme. Since its rank space is the disjoint union of spectra of blue fields, $X^\rk$ is discrete and reduced. Before we show that $x$ is pseudo-Hopf, we show that $\prerk x\to\barx$ is an isomorphism for all $x\in X^\rk$. By definition of the rank space, there is a pseudo-Hopf point $y\in X$ such that $\{x\}=\prerk y^\px$. If we denote $\Gamma\overline y$ by $B$, then $\barx=\Spec\hat B^\px$ and we have to show that the natural morphism $\hat B^\px\to (\hat B^\px)^\hatexp $ is an isomorphism. In the case that $\hat B^\px$ is with $-1$, the unit field $\hat B^\px$ is with inverses and equals its additive closure in $(\hat B^\px)_\inv$. In case that $\hat B^\px$ is without $-1$, $\hat B^\px$ does not contain the additive inverse of any element $b$. Thus $\hat B^\px$ equals the image of $\hat B^\px$ in $(\hat B^\px)_\inv$, which is the same as $\hat B^\px_\canc$, and if $\sum a_i\=0$ in $B$, then $a_i\=0$ for all $i$. This means that $(\hat B^\px)^\hatexp \simeq (B^\px_\canc)^\hatexp \simeq B^\px_\canc\simeq\hat B^\px$.

 We show that every point $x$ of $X^\rk$ is pseudo-Hopf. Clearly, $\barx=\{x\}$ is affine for every $x\in X^\rk$. Let $y$ be a pseudo-Hopf point of $X$ such that $\{x\}=\prerk y^\px$. Then $\barx^+_\Z\simeq\prerk{y}^+_\Z$ is a non-empty flat scheme and $x$ is almost of indefinite characteristic. The blueprint $\Gamma\prerk x_\inv=\Gamma\barx_\inv$ is generated by its units since $\Gamma\barx_\inv\subset(\Gamma\barx^+_\inv)^\px$. Thus $x$ is pseudo-Hopf, which finishes the proof of \eqref{part1}.

 We show \eqref{part2}. If $X$ is of pure rank, then every point $x$ is pseudo-Hopf of minimal rank in its component, i.e.\ $\cZ(X)=X$. Since $X$ is discrete and reduced, $\Fpx(\barx)\simeq\cO_X(\{x\})\simeq\kappa(x)$ is a blue field for all $x\in X$. Since $\prerk x\simeq\barx$, we have isomorphisms
 \[
  \prerk x^\px \quad \stackrel\sim\longleftarrow \quad \prerk x \quad \stackrel\sim\longrightarrow \quad \barx \quad \stackrel\sim\longrightarrow \quad \Spec\cO_X(\{x\})
 \]
 and, consequently, $X^\rk\simeq X^\prk\simeq X$. This completes the proof of the proposition.
\end{proof}

We give a series of examples of blue schemes and their rank spaces.

\begin{ex}[Tori]
 The key example of rank spaces are tori over $\Fun$. Let $X=\G_{m,\Fun}^r$ be the spectrum of $B=\Fun[T_1^{\pm1},\dotsc,T_r^{\pm1}]$. Then $X$ consists of one point $\eta$, namely, the $0$-ideal of $B$, and $\eta$ is of indefinite characteristic, $\overline\eta=\Spec B$ is affine, $B_\inv$ is generated by its units since $B^\px_\inv=B_\inv$, and $B_\Z^+=\Z[T_1^{\pm1},\dotsc,T_r^{\pm1}]$ is a free $\Z$-module. Therefore $\eta$ is pseudo-Hopf and we have isomorphisms $X^\rk\simeq X^\prk\simeq X$. 

 Note that the rank of $X$ is $r$, which equals the rank of the group scheme $\SG_{m,\Z}^r$. This is a first instance for the meaning of the rank of a blue scheme. We will see later that, more generally, the rank of a ``Tits-Weyl model'' of a reductive group scheme equals the reductive rank of the group scheme (see Theorem \ref{thm: properties of tits-weyl groups}).
\end{ex}

\begin{ex}[Monoidal schemes]\label{ex: rank space of a monoidal scheme}
 If $X$ is a monoidal scheme, then every point $x$ of $X$ is of indefinite characteristic and $\Gamma \barx^+_\Z$ is a free $\Z$-module. The scheme $\barx_\inv$ is generated by its units if and only if $\barx=\{x\}$, i.e.\ if and only if $x$ is a closed point of $X$. In this case, $\barx=\{x\}$ is an affine blue scheme. Thus the pseudo-Hopf points of $X$ are its closed points. Therefore, the rank space of a monoidal scheme $X$ lifts to $X$. If $X$ is locally of finite type, then $X^\rk$ embeds into $X$.

 The closed points that belong to the rank space, i.e.\ that are of minimal rank, are easily determined since the rank of a point $x$ of a monoidal scheme $X$ equals the free rank of the unit group $\cO_{X,x}^\times$ of the stalk $\cO_{X,x}$ at $x$. For example, the projective space $\P_\Fun^n$ has $n+1$ closed points, which are all of rank $0$. Thus $(\P^n_\Fun)^\rk$ consists of $n+1$ points, which are all isomorphic to $\Spec\Fun$.
\end{ex}

\begin{ex}[Semiring schemes]
 If $X$ is a semiring scheme, then none of its points is almost of indefinite characteristic. Thus both the pre-rank space and the rank space of $X$ are empty.
\end{ex}

\begin{ex}
 The following are four examples that demonstrate certain effects that can occur for blue schemes and their rank spaces. The first example shows that pseudo-Hopf points are in general not closed, a fact that we have to consider in case of $\Fun$-models of adjoint groups. Let $B=\bpgenquot{\Fun[T]}{T\=1+1}$ and $X=\Spec B$. Then $X$ has two points $\eta=(0)$ and $x=(T)$. The closed subscheme $\barx$ is isomorphic to $\F_2$, which is of characteristic $2$ and not a free $\Z$-module. The closed subscheme $\overline\eta$ is $B$ itself and thus affine. The point $\eta$ has all potential characteristics except for $2$. The unit field of $B_\inv$ is $B_\inv^\px=\Fun$, thus $B^+_\Z\simeq\Z\simeq(B^\px)^+_\Z$, which shows $\eta$ is pseudo-Hopf. Thus $X$ is of rank $0$ and $\cZ(X)=\{\eta\}$ is not closed in $X$. The morphism $\rho_X:X^\prk\to X$ is an isomorphism, but the rank space $X^\rk$ does not lift to $X$.

 The second example extends the first example in a way such that the morphism $\rho_X$ is no longer injective. Let $B=\bpgenquot{\Fun[S,T]}{S+T\=1+1}$ and $X=\Spec B$. Then $X$ has four points $\eta=(0)$, $x=(S)$, $y=(T)$ and $z=(S,T)$. For similar reasons as in the first example, the pseudo-Hopf points of $X$ are $x$ and $y$, which are both of rank $0$. Thus $\cZ(X)=\{x,y\}$ and $X^\prk=\prerk x\cup\prerk y$. Both, the closed point of $\prerk x$ and the closed point of $\prerk y$ are mapped to $z$. Thus $\rho_X$ is not injective.

 The third example presents a blue scheme in which one pseudo-Hopf point lies in the closure of another pseudo-Hopf point. Let $B=\bpgenquot{\Funsq[S,T]}{T^2\=1,S\=T+1}$ and $X=\Spec B$. Then $X$ has two points $\eta=(0)$ and $x=(S)$. We have $\barx=\Spec\Funsq$, which means that $x$ is pseudo-Hopf of rank $0$. The point $\eta$ has all potential characteristics except for $1$, $\overline\eta=\Spec B$ is affine, $B^\px_\inv=\bpgenquot{\Funsq[T]}{T^2\=1}=\Funsq[\mu_2]$ is generated by its units (where $\mu_2$ is the cyclic group with two elements) and its extension to $\Z$ is the flat ring $B^+_\Z=\Z[\mu_2]$. Thus $\eta$ is also pseudo-Hopf of rank $0$. This means that $\cZ(X)=\{\eta,x\}$ and $X^\prk=\overline\eta\cup\barx\simeq X\cup\Spec\Funsq$ does not map injectively to $X$. The rank space of $X$ is $X^\rk\simeq \Spec\Funsq[\mu_2]\cup\Spec\Funsq$.

 The forth example shows that in general $\upsilon_X^+:X^{\prk,+}\to X^{\rk,+}$ is not an isomorphism. Let $B=\bpgenquot{\Fun[S,T^{\pm1}]}{T\=S+1+1}$. Then $X=\Spec B$ has two points $x=(S)$ and $\eta=(0)$. The scheme $\barx$ is the spectrum of $\bpgenquot{\Fun[T^{\pm1}]}{T\=1+1}$, whose base extension to $\Z$ is the localization $\Z_{(2)}$, which is not a flat ring. The point $\eta$ is easily seen to be pseudo-Hopf. Thus $X^\prk=\prerk\eta=\Spec B$ and $X^\rk=\prerk\eta^\px=\Spec\Fun[T^{\pm1}]$. The embedding $\N[T^{\pm1}]\to \bpgenquot{\N[S,T^{\pm1}]}{T\=S+1+1}$ is not surjective, thus $\upsilon_X^+:X^{\prk,+}\to X^{\rk,+}$ is not an isomorphism.
\end{ex}


\subsection{Tits morphisms}
\label{subsection: tits morphisms}

\begin{df}
 Let $X$ and $Y$ be blue schemes. A \emph{Tits morphism $\varphi:X\to Y$} is a pair $(\varphi^\rk,\varphi^+)$ where $\varphi^\rk:X^\rk\to Y^\rk$ is a morphism between the rank spaces of $X$ and $Y$ and $\varphi^+:X^+\to Y^+$ is a morphism between the universal semiring schemes of $X$ and $Y$ such that the diagram
 \[
\xymatrix@C=6pc{X^{\rk,+}_\Z  \ar[r]^{\varphi^{\rk,+}_\Z} \ar[d]_{\rho_{X,\Z}^+}      & Y^{\rk,+}_\Z  \ar[d]^{\rho_{Y,\Z}^+} \\ 
           X^+_\Z  \ar[r]^{\varphi^+_\Z}                                      & Y^+_\Z  }
 \]
 commutes.

 If $\varphi:X\to Y$ and $\psi:Y\to Z$ are two Tits morphisms, then the composition $\psi\circ\varphi:X\to Z$ is defined as the pair $(\psi^\rk\circ\varphi^\rk,\psi^+\circ\varphi^+)$. The \emph{Tits category} is the category $\TSch$ whose objects are blue schemes and whose morphisms are Tits morphisms.
\end{df}

To make a clear distinction between Tits morphisms between blue schemes and morphisms in the usual sense, we will often refer to the latter kind of morphism as \emph{locally algebraic morphisms} (cf.\ \cite[Thm.\ 3.23]{blueprints1} for the fact that locally algebraic morphisms are locally algebraic). 

\begin{rem}
 For a wide class of blue schemes $X$, the base extension $\upsilon^+_X: X^{\prk,+}\to X^{\rk,+}$ of $\upsilon_X: X^\prk\to X^\rk$ is already an isomorphism and we can consider $\rho_X^+$ as a morphism from $X^{\rk,+}$ to $X^+$. If this is the case for $X$ and $Y$, then a pair $(\varphi^\rk,\varphi^+)$ as above is a Tits morphism if and only if the diagram
 \[
\xymatrix@C=6pc{X^{\rk,+}  \ar[r]^{\varphi^{\rk,+}} \ar[d]_{\rho_{X}^+}      & Y^{\rk,+}  \ar[d]^{\rho_{Y}^+} \\ 
           X^+  \ar[r]^{\varphi^+}                                      & Y^+  }
 \]
 commutes. In fact, all of the blue schemes that we will encounter in the rest of the paper, will be of this sort.
\end{rem}

The Tits category comes together with two important functors: the base extension $(\blanc)^+:\TSch\to\Sch^+$ to semiring schemes, which sends a blue scheme $X$ to $X^+$ and a Tits morphism $\varphi:X\to Y$ to $\varphi^+:X^+\to Y^+$; and the extension $(\blanc)^\rk:\TSch\to\rkBSch$ to blue schemes of pure rank, which sends a blue scheme $X$ to its rank space $X^\rk$ and a Tits morphism $\varphi:X\to Y$ to $\varphi^\rk:X^\rk\to Y^\rk$. 

The former functor allows us to define the base extensions $(-)_k^+:\TSch\to\Sch^+_k$ for every semiring $k$ or, more generally, the base extension $\blanc\Sotimes S:\TSch\to\Sch^+_S$ for every semiring scheme $S$. 

The latter functor allows us to define the \emph{Weyl extension $\cW:\TSch\to\Sets$} from the Tits category to the category of sets that associates to each blue scheme $X$ the underlying set of its rank space $X^\rk$ and to each Tits morphism $\varphi:X\to Y$ the underlying map of the morphism $\varphi:X^\rk\to Y^\rk$ between the rank spaces.

Both locally algebraic morphisms and Tits morphisms between two blue schemes $X$ and $Y$ have an base extension to semiring scheme morphisms between $X^+$ and $Y^+$. The classes of semiring scheme morphisms between $X^+$ and $Y^+$ that are the respective base extensions of locally algebraic morphisms and of Tits morphisms between $X$ and $Y$ are, in general, different. For example, the natural embedding $\iota_\N:\G_{m,\N}\to\A^1_\N$ descends to a locally algebraic morphism $\iota_\Fun:\G_{m,\Fun}\to\A^1_\Fun$, but there is no Tits morphism $\tilde\iota:\G_{m,\Fun}\to\A^1_\Fun$ with $\tilde\iota^+=\iota_\N$. As we will see in the following, Tits morphisms are more flexible in other aspects, which will allow us to descend the group laws of many group schemes to ``$\Fun$-models'' of the group scheme, which is not the case for locally algebraic morphisms. 

In the following, we will investigate the case that a locally algebraic morphism $\varphi:X\to Y$ of blue schemes defines a Tits morphism. Namely, if $\varphi$ maps $\cZ(X)$ to $\cZ(Y)$, then we can define a morphism $\varphi^\prk:X^\prk\to Y^\prk$ by $\varphi^\prk\vert_{\prerk x}=(\varphi\vert_{\barx})^\hatexp $ for $x\in\cZ(X)$. This defines a morphism $\varphi^\prk$ between the pre-rank spaces of $X$ and $Y$ since $y=\varphi(x)\in\cZ(Y)$ and thus the morphism $\varphi$ restricts to a morphism $\varphi\vert_{\barx}:\barx\to \overline y$ to which we can apply the functor $(\blanc)^\hatexp $. This definition of $\varphi^\prk$ behaves well with composition, i.e.\ if $\varphi$ is as above and $\psi:Y\to Z$ is a locally algebraic morphism that maps $\cZ(Y)$ to $\cZ(Z)$, then $(\psi\circ\varphi)^\prk=\psi^\prk\circ\varphi^\prk$.


 Let $\varphi:X\to Y$ be a locally algebraic morphism of blue schemes that maps $\cZ(X)$ to $\cZ(Y)$ and $\varphi^\prk:X^\prk\to Y^\prk$ the corresponding morphism between the pre-rank spaces of $X$ and $Y$. Then applying the functor $(\blanc)^\rk$ to the connected components of $X^\prk$ yields a morphism $\varphi^\rk:X^\rk\to Y^\rk$ between the rank spaces of $X$ and $Y$.

 We say that a locally algebraic morphism $\varphi:X\to Y$ that maps $\cZ(X)$ to $\cZ(Y)$ is \emph{Tits} or that $\varphi$ is a \emph{locally algebraic Tits morphism}. We denote the category of blue schemes together with locally algebraic Tits morphisms by $\BTSch$. The following proposition justifies the terminology.

\begin{prop}\label{prop: morphisms that are tits}
 Let $\varphi:X\to Y$ be a locally algebraic morphism of blue schemes that maps $\cZ(X)$ to $\cZ(Y)$. Then the pair $(\varphi^\rk,\varphi^+)$ is a Tits morphism from $X$ to $Y$. 
\end{prop}

\begin{proof}
 The base extension of the commutative diagram
 \[
  \xymatrix@C=6pc{X^{\prk}  \ar[r]^{\varphi^{\prk}} \ar[d]_{\rho_{X}}      & Y^{\prk}  \ar[d]^{\rho_{Y}} \\ 
           X  \ar[r]^{\varphi}                                      & Y  }
 \]
 to semiring schemes yields the commutative diagram
 \[
\xymatrix@C=6pc{X^{\rk,+}_\Z  \ar[r]^{\varphi^{\rk,+}_\Z} \ar[d]_{\rho_{X,\Z}^+}      & Y^{\rk,+}_\Z  \ar[d]^{\rho_{Y,\Z}^+} \\ 
           X^+_\Z  \ar[r]^{\varphi^+_\Z}                                      & Y^+_\Z,  }
 \]
 which proves the lemma.
\end{proof}

Let $X$ and $Y$ be two blue schemes. We define the \emph{set $Y^\sT(X)$ of $X$-rational Tits points of $Y$} as the set $\Hom_\cT(X,Y)$ of Tits morphisms from $X$ to $Y$. We denote the set of locally algebraic morphisms $X\to Y$ of blue schemes by $\Hom(X,Y)$. 

Since the rank space of a semiring scheme is empty, we have the following immediate consequence of the previous proposition.

\begin{cor}\label{cor: tits morphisms between semiring schemes}
  Let $X$ be a semiring scheme and $Y$ a blue scheme. Let $\alpha_Y:Y^+\to Y$ the base extension morphism. Then the map
        \[ 
           \begin{array}{ccc}
              \Hom_\cT(X,Y) & \longrightarrow & \Hom(X,Y) \\
              (\varphi^\rk,\varphi^+) & \longmapsto & \alpha_Y\circ\varphi^+    
           \end{array} 
        \]
        is a bijection. This means in particular that $\SSch$ embeds as a full subcategory into $\TSch$. \qed
\end{cor}

Since the rank space of a blue scheme $X$ of pure rank is isomorphic to $X$ itself, a Tits morphism $\varphi:X\to Y$ between two blue schemes $X$ and $Y$ of pure rank is determined by the morphism $\varphi^\rk:X\to Y$. Therefore, also $\rkBSch$ is a full subcategory of $\TSch$. 

\begin{ex}[$\Fun$-rational Tits points of monoidal schemes]
 Given a monoidal scheme $X$, then for all its points $x$, the schemes $\barx=\prerk x$ and $\prerk x^\px$ are also monoidal. Since $\prerk x$ is generated by its units if and only if $\prerk x=\prerk x^\px$, the rank space $X^\rk$ lifts to $X$ (if $X$ is locally of finite type, $X^\rk$ embeds into $X$). Thus a Tits morphism $\varphi:Y\to X$ from a scheme $Y$ of pure rank is already determined by $\varphi^\rk:Y^\rk\to X^\rk$. This holds, in particular, for $Y=\ast_\Fun$. Since a blue field that is a monoid admits precisely one morphism to $\Fun$, the $\Fun$-rational Tits points of $X$ correspond to the points of the rank space $X^\rk$. These correspond, in turn, to the set $\cZ(X)$ of pseudo-Hopf points of minimal rank in $X$, which is the image of $\tilde\rho_X:X^\rk\to X$.

 Note that in case of a connected monoidal scheme, the set of $\Fun$-rational Tits points coincides with the sets of $\Fun$-rational points as defined \cite{L09}.
\end{ex}

The following proposition characterizes those semiring scheme morphisms that are base extensions of Tits morphisms. Note that if $X$ is a blue scheme and $x\in\cZ(X)$ is a pseudo-Hopf point of minimal rank, then $\Gamma\barx^+_\Z=\Gamma\prerk x^+_\Z=\Fpx(\prerk x)^+_\Z$. In particular, the cancellative blue field $\Fpx(\prerk x)$ is a subblueprint of $\Gamma\barx^+_\Z$.


\begin{prop}\label{prop: semiring scheme morphisms that descend to tits morphisms}
 Let $X$ and $Y$ be two blue schemes and $\varphi^+:X^+\to Y^+$ a morphism. Then there exists a morphism $\varphi^\rk:X^\rk\to Y^\rk$ between the rank spaces of $X$ and $Y$ such that $(\varphi^\rk,\varphi^+)$ is a Tits morphism from $X$ to $Y$ if and only if there is a map $\varphi_0:\cZ(X)\to \cZ(Y)$ such that for all $x\in \cZ(X)$ and $y=\varphi_0(x)$,
 \begin{enumerate}
  \item $\varphi^+\Bigl(\,\rho^+_{X}\bigr(\,\barx^+\,\bigr)\,\Bigr) \ \subset \ \rho^+_{Y}\bigl(\,\overline y^+\,\bigl)$ \quad and
  \item the blueprint morphism $f_x=\Gamma(\varphi^+\vert_{\barx^+})_\Z^+:\Gamma\bary^+_\Z\to \Gamma\barx^+_\Z$ maps $\Fpx(\prerk y)\subset \Gamma\bary^+_\Z$ to $\Fpx(\prerk x)\subset \Gamma\barx^+_\Z$.
 \end{enumerate}
 If $\rho_{Y}^+:Y^{\rk,+}\to Y^+$ is injective, then $\varphi^\rk$ is uniquely determined by $\varphi^+$.
\end{prop}

\begin{proof}
 For every point $x\in\cZ(X)$, the scheme $\prerk x^\px$ consists of one point, which denote by $\tilde x$. The association $x\to \tilde x$ is a bijection between $\cZ(X)$ and the points of $X^\rk$. Similarly, we denote by $\tilde y$ the point of $Y^\rk$ that corresponds to $y\in\cZ(Y)$.

 Given a Tits morphism $(\varphi^\rk,\varphi^+)$ from $X$ to $Y$, define $\varphi_0(x)=y$ if $\varphi^\rk(\tilde x)=\tilde y$. Evidently, this map satisfies \eqref{part1} and \eqref{part2}.

 Given a morphism $\varphi^+:X^+\to Y^+$ and a map $\varphi_0:\cZ(X)\to\cZ(Y)$ that satisfies \eqref{part1} and \eqref{part2}, we define $\varphi^\rk(\tilde x)=\tilde y$ (as a map) if $\varphi_0(x)=y$. The morphism $(\varphi^\rk)^\#$ between the structure sheaves is determined by the blueprint morphisms $\Gamma(\varphi^\rk\vert_{\{\tilde x\}})=f_x\vert_{\Fpx(\prerk y)}:\Fpx(\prerk y)\to\Fpx(\prerk x)$. The pair $(\varphi^\rk,\varphi^+)$ is clearly a Tits morphism from $X$ to $Y$. 

 Assume that $\rho_Y^+:Y^{\rk,+}\to Y^+$ is injective. Then the map $\varphi_0:\cZ(X)\to \cZ(Y)$ is uniquely determined by the condition that there must be a $y\in \cZ(Y)$ for every $x\in\cZ(X)$ such that $\varphi^+$ restricts to a morphism $\varphi^+\vert_{\barx^+}:\barx^+\to\bary^+$. This determines $\varphi^\rk:\tilde x\mapsto\tilde y$ as a map. If $\varphi^\rk:X^\rk\to Y^\rk$ can be extended to a morphism, then property \eqref{part2} of the proposition applied the scheme morphism $(\varphi^+\vert_{\barx^+})^+_\Z :\barx^+_\Z\to\bary^+_\Z$ shows that the morphism $\varphi^\rk$ is uniquely determined by $\varphi^+$. This shows the additional statement of the proposition.
\end{proof}


\section{Tits monoids}
\label{section: tits monoids}

In this section, we introduce the notion of a Tits monoid as a monoid in the Tits category. We start with a reminder on groups and monoids in Cartesian categories. Then we show that the Tits category $\TSch$ as well as some other categories and functors between them are Cartesian. This allows us to introduce the objects that will be in the focus of our attention for the rest of the paper: Tits-Weyl models of smooth affine group schemes $\cG$ of finite type. Roughly speaking, a Tits-Weyl model of $\cG$ is a Tits monoid $G$ such that $G_\Z^+$ is isomorphic to $\cG$ as a group scheme and such that $\cW(G)$ is isomorphic to the Weyl group of $\cG$.


\subsection{Reminder on Cartesian categories}
\label{subsection: reminder on cartesian categories}

A \emph{Cartesian category} is a category $\cC$ that contains finite products and a terminal object $\ast_\cC$. A \emph{Cartesian functor} is a (covariant) functor between Cartesian categories that commutes with finite products and sends terminal objects to terminal objects. The importance of Cartesian categories is that they admit to define group objects, and the importance of Cartesian functors is that they send group objects to groups objects. In the following, we will expose some facts on (semi-)group objects. All this is general knowledge and we stay away from proving facts. For more details, see, for instance, \cite[Section 1]{L09}.

 \subsubsection*{Semigroups}

Let $\cC$ be a Cartesian category. A \emph{semigroup in $\cC$} is a pair $(G,\mu)$ where $G$ is an object in $\cC$ and $\mu:G\times G\to G$ is a morphism such that the diagram
\[
 \xymatrix{G\times G\times G\ar[rr]^{\mu\times\id}\ar[d]_{\id\times\mu}&& G\times G\ar[d]^\mu\\ G\times G\ar[rr]^\mu&& G }
\]
commutes. We often suppress $\mu$ from the notation and say that $G$ is a semigroup object in $\cC$. We call $\mu$ the \emph{semigroup law of $G$}.

An \emph{(both-sided) identity} for a semigroup $G$ is a morphism $\epsilon:\ast_\cC\to G$ such that the diagrams
\[
 \xymatrix@C=4pc{G\times \ast_\cC\ar[r]^{(\id,\epsilon)}\ar[dr]_{\pr_1}& G\times G\ar[d]^\mu\\ & G }\qquad\text{and}\qquad\xymatrix@C=4pc{\ast_\cC\times G\ar[r]^{(\epsilon,\id)}\ar[dr]_{\pr_2}& G\times G\ar[d]^\mu\\ & G }
\]
commute. An identity for $G$ is unique. If $G$ is with an identity, we say that $G$ is a \emph{monoid in $\cC$} and that $\mu$ is its \emph{monoid law}.

A \emph{group in $\cC$} is a monoid $(G,\mu)$ with identity $\epsilon:\ast_\cC\to G$ that has an \emph{inversion}, i.e.\ a morphism $\iota:G\to G$ such that the diagrams
\[
 \xymatrix{G\ar[r]^{\Delta}\ar[d]& G\times G\ar[r]^{(\id,\iota)}& G\times G\ar[d]^\mu\\ \ast_\cC\ar[rr]^\epsilon&& G } \qquad\text{and}\qquad \xymatrix{G\ar[r]^{\Delta}\ar[d]& G\times G\ar[r]^{(\iota,\id)}& G\times G\ar[d]^\mu\\ \ast_\cC\ar[rr]^\epsilon&& G }
\]
commute. An inversion is unique. If $G$ is a group, we call $\mu$ its \emph{group law}.

A pair $(G,\mu)$ is a semigroup (monoid / group) in $\cC$ if and only if $\Hom_\cC(X,G)$ together with the composition induced by $\mu$ is a semigroup (monoid / group) in $\Sets$ for all objects $X$ in $\cC$. 

Let $\cF:\cC\to\cD$ be a Cartesian functor and $(G,\mu)$ a semigroup in $\cC$. Then $(\cF(G),\cF(\mu))$ is a semigroup in $\cD$, and $\cF$ maps an identity to an identity and an inversion to an inversion. For every object $X$ in $\cC$, the map 
\[
 \Hom_\cC(X,G) \quad \longrightarrow\quad \Hom_\cD(\cF(X),\cF(G))
\]
is a semigroup homomorphism, which maps an identity to an identity and inverses to inverses if they exist.

A \emph{homomorphism of semigroups} $(G_1,\mu_1)$ and $(G_2,\mu_2)$ in $\cC$ is a morphism $\varphi:G_1\to G_2$ such that the diagram
\[
 \xymatrix{G_1\times G_1\ar[rr]^{\mu_1}\ar[d]_{(\varphi,\varphi)} && G_1\ar[d]^\varphi \\  G_2\times G_2\ar[rr]^{\mu_2} && G_2 }
\]
commutes. If $G_1$ is with an identity $\epsilon_1$ and $G_2$ is with an identity $\epsilon_2$, then a semigroup homomorphism $\varphi:G_1\to G_2$ is called \emph{unital} (or \emph{monoid homomorphism}) if the diagram
\[
 \xymatrix@C=5pc@R=0.6pc{  &  G_1\ar[dd]^\varphi \\ \ast_\cC\ar[ur]^{\epsilon_1}\ar[dr]_{\epsilon_2}\\ & G_2}
\]
commutes. If $G_2$ is a group, then every semigroup homomorphism $\varphi:G_1\to G_2$ is unital. A Cartesian functor $\cF:\cC\to\cD$ sends (unital) semigroup homomorphisms to (unital) semigroup homomorphisms.

\subsubsection*{Monoid and group actions}

Let $(G,\mu)$ be a monoid with identity $\epsilon:\ast_\cC\to G$ and $X$ an object in $\cC$. A \emph{(unitary left) action of $G$ on $X$ in $\cC$} is a morphism $\theta: G\times X\to X$ such that the diagrams
\[
 \xymatrix{G\times G\times X\ar[rr]^{(\id,\theta)}\ar[d]_{(m,\id)}&& G\times X\ar[d]^\theta\\ G\times X\ar[rr]^\theta&& X} 
  \qquad\text{and}\qquad 
 \xymatrix{\ast_\cC\times X\ar[rr]^{(\epsilon,\id)}\ar[drr]_{\pr_2}&& G\times X\ar[d]^\theta\\ && X }
\]
commute. Let $\cF:\cC\to\cD$ be a Cartesian functor. Then $\cF$ sends an action $\theta$ of $G$ on $X$ in $\cC$ to an action $\cF(\theta)$ of $\cF(G)$ on $\cF(X)$ in $\cD$. If $\theta$ is unitary, then $\cF(\theta)$ is unitary. If $G$ is a monoid, then we call a unitary action $\theta: G\times X\to X$ also a \emph{monoid action}, if $G$ is a group, then we call $\theta$ a \emph{group action}.

\subsubsection*{Semidirect products of groups}

The \emph{direct product of groups $(G_1,m_1)$ and $(G_2,m_2)$} in a Cartesian category $\cC$ is the product $G_1\times G_2$ together with the pair $m=(m_1,m_2)$ as group law, which is easily seen to define a group object. 

Let $(N,m_N)$ and $(H,m_H)$ be groups in $\cC$ and let $\theta:H\times N\to N$ be a group action that respects the group law $m_N$ of $N$, i.e.\ if we define the \emph{change of factors along $\theta$} as
\[
 \xymatrix@C=3,5pc{\chi_\theta:\quad H\times N\ar[r]^{(\Delta,\id)}& H\times H\times N\ar[r]^{(\id,\theta)}&H\times N\ar[r]^\chi& N\times H,}
\]
 then the diagram 
\[
 \xymatrix@C=4pc@R=0,8pc{H\times N\times N\ar[r]^{(\id,m_N)}\ar[dd]_{(\chi_\theta,\id)}&H\times N\ar[dr]^\theta\\ &&N\\ N\times H\times N\ar[r]^{(\id,\theta)}& N\times N\ar[ur]_{m_N}} 
\]
commutes. Then the morphism
\[
 \xymatrix@C=4pc{m_\theta:\quad N\times H\times N\times H\ar[r]^{\quad(\id,\chi_\theta,\id)}&N\times N\times H\times H\ar[r]^{\quad\quad(m_N,m_H)}&N\times H}
\]
is a group law for $G=N\times H$. We say that $G$ is the \emph{semidirect product of $N$ with $H$ w.r.t.\ $\theta$} and write $G=N\rtimes_\theta H$. The group object $N$ is a normal subgroup of $G$ with quotient group $H$, and $H$ is a subgroup of $G$ that acts on $N$ by conjugation. The conjugation $H\times N\to N$ equals $\theta$. If $\theta: H\times N\to N$ is the canonical projection to the second factor of $H\times N$, then $N\rtimes_\theta H$ is equal to the direct product of $N$ and $H$ (as a group).

If $\cF:\cC\to \cD$ is a Cartesian functor and if $G=N\rtimes_\theta H$ in $\cC$, then $\cF(G)=\cF(N)\rtimes_{\cF(\theta)} \cF(H)$ in $\cD$.


\subsection{The Cartesian categories and functors of interest}
\label{subsection: sch_t is cartesian}

In this section, we show that the Tits category $\TSch$ is Cartesian, which allows us to consider monoids and group objects in this category. We will further investigate certain Cartesian functors to and from $\TSch$.

%


In order to prove that $\TSch$ is Cartesian, we have to verify that certain constructions behave well with products.

\begin{lemma}\label{lemma: unit scheme of products}
 Let $X$ and $Y$ be two blue schemes. Then $(X\times Y)^\px\simeq X^\px\times Y^\px$.
\end{lemma}

\begin{proof}
 Since $(X\times Y)^\px=\Spec\Gamma(X\times Y)^\px$ and $X^\px\times Y^\px=\Spec\bigl( \Gamma X^\px\otimes_\Fun\Gamma Y^\px \bigr)$, we prove the lemma by establishing an isomorphism between the corresponding blueprints of global section. By Lemma \ref{lemma: global sections of fibre products}, $\Gamma(X\times Y)=\Gamma X\otimes_\Fun\Gamma Y$. Let $\Gamma X=\bpquot{A_X}{\cR_X}$ and $\Gamma Y=\bpquot{A_Y}{\cR_Y}$ be proper representations of the global sections of $X$ and $Y$, respectively. Then 
 \[
  \Gamma X\otimes_\Fun\Gamma Y=\bpquot{A_X\times A_Y}{\cR}
 \]
 for $\cR=\gen{\cR_X\times\{1\},\{1\}\times\cR_Y}$ and 
 \[
  (\Gamma X\otimes_\Fun\Gamma Y)^\px=\bpquot{\{0\}\cup(A_X\times A_Y)^\times}{\cR'}
 \] 
 for $\cR'=\cR\vert_{\{0\}\cup(A_X\times A_Y)^\times}$. Since $(A_X\times A_Y)^\times=A_X^\times\times A_Y^\times$, the above expression equals
 \[
  \bpgenquot{(\{0\}\cup A_X^\times)\times(\{0\}\cup A_Y^\times)}{(\cR_X\vert_{\{0\}\cup A_X^\times})\times\{1\},\{1\}\times (\cR_Y\vert_{\{0\}\cup A_Y^\times})},
 \]
 which is $\Gamma X^\px\otimes_\Fun\Gamma Y^\px$.
\end{proof}

\begin{lemma}\label{lemma: product of blueprints generated by their units}
 Let $B_1$ and $B_2$ be two blueprints and $B=B_1\otimes_\Fun B_2$ their tensor product. Assume that both $B_{1,\Z}^+$ and $B_{2,\Z}^+$ are non-zero and free as $\Z$-modules. Then the canonical inclusion
 \[
  u_\Z^+: \quad (B^\px)^+_\Z \quad \longrightarrow \quad B^+_\Z
 \]
 is an isomorphism if and only if the canonical inclusions $u_{i,\Z}^+:(B_i^\px)^+_\Z\to B_{i,\Z}^+$ are isomorphisms for $i=1,2$.
\end{lemma}

\begin{proof}
 Since $(B^\px)^+_\Z = (B_1^\px)^+_\Z\Sotimes_\Z (B_2^\px)^+_\Z$ (by the previous lemma) and $B^+_\Z = B_{1,\Z}^+\Sotimes_\Z B_{2,\Z}^+$, the inclusion $u^+_\Z$ is clearly an isomorphism if both $u_{1,\Z}^+$ and $u_{2,\Z}^+$ are so.

 Assume that $u^+_\Z$ is an isomorphism. Since $B_\Z^+= B_{1,\Z}^+\Sotimes_\Z B_{2,\Z}^+$ is non-zero and free, the isomorphic blueprint $(B^\px)^+_\Z=(B_1^\px)^+_\Z\Sotimes_\Z (B_2^\px)^+_\Z$ is non-zero and free. Thus both factors $(B_1^\px)^+_\Z$ and $(B_2^\px)^+_\Z$ are non-zero and free. Therefore we obtain a commutative diagram
 \[
  \xymatrix@C=6pc{(B_1^\px)^+_\Z\Sotimes_\Z (B_2^\px)^+_\Z  \ar[r]^\sim_{ u_\Z^+}   &  B_{1,\Z}^+\Sotimes_\Z B_{2,\Z}^+ \\ (B_i^\px)^+_\Z  \arincl{[r]^{u_{i,\Z}^+}} \arincl{[u]}  &   B_{i,\Z}^+  \arincl{[u]}}
 \]
 of inclusions of free $\Z$-modules for $i=1,2$ where the morphisms on the top is an isomorphism. If we choose a basis $(a_i)$ for $(B_1^\px)^+_\Z$ and a basis $(b_j)$ for $(B_2^\px)^+_\Z$, then $(a_i\otimes b_j)$ is a basis for $(B_1^\px)^+_\Z\Sotimes_\Z (B_2^\px)^+_\Z= B_{1,\Z}^+\Sotimes_\Z B_{2,\Z}^+$. Thus $(a_i)$ is a basis for $B_{1,\Z}^+$ and $(b_j)$ is a basis for $B_{2,\Z}^+$, which proves that $u_{1,\Z}^+$ and $u_{2,\Z}^+$ are isomorphisms.
\end{proof}

\begin{prop}\label{prop: product of rank spaces}
 Let $X_1$ and $X_2$ be two blue schemes. Then there are canonical identifications
 \[ 
  \cZ(X_1\times X_2)=\cZ(X_1)\times\cZ(X_2), \quad (X_1\times X_2)^\prk = X_1^\prk\times X_2^\prk \quad\text{and}\quad (X_1\times X_2)^\rk = X_1^\rk\times X_2^\rk
 \]
 such that 
 \[
  \xymatrix@C=4pc{\cZ(X_1\times X_2) \arincl{[r]\ar[d]^{\pr_i}}   &   X_1\times X_2 \ar[d]^{\pr_i} \\
                  \cZ(X_i)\arincl{[r]}  & X_i}
 \]
 commutes as a diagram in $\Sets$ and
 \[
  \xymatrix@C=6pc{X_1^\rk\times X_2^\rk \ar[d]^{\pr_i}   &  X_1^\prk\times X_2^\prk \ar[l]_{\upsilon_{X_1\times X_2}}  \ar[r]^{\rho_{X_1\times X_2}} \ar[d]^{\pr_i}   &  X_1\times X_2 \ar[d]^{\pr_i} \\
                   X_i^\rk    &  X_i^\prk \ar[l]_{\upsilon_{X_i}}  \ar[r]^{\rho_{X_i}}    &  X_i }
 \]
 commutes as a diagram in $\BSch$ for $i=1,2$.
\end{prop}

\begin{proof}
 If $X_1=\coprod X_{1,k}$ and $X_2=\coprod X_{2,l}$ are the respective decompositions of $X_1$ and $X_2$ into connected components, then $X_1\times X_2=\coprod X_{1,k}\times X_{2,l}$ is the decomposition of $X_1\times X_2$ into connected components. Since these decompositions are compatible with the canonical projections $\pr_i:X_1\times X_2\to X_i$, we can assume for the proof that $X_1$, $X_2$ and $X_1\times X_2$ are connected.

 Recall that $\cZ(X)$ are the pseudo-Hopf points of a (connected) blue scheme $X$ that are of minimal rank, i.e.\ of rank equal to $\rk X$. By Lemma \ref{lemma: potential characteristics of b_1 tensor b_2}, the point $(x_1,x_2)\in X_1\times X_2$ is of almost indefinite characteristic if and only if both $x_1\in X_1$ and $x_2\in X_2$ are points that are of almost indefinite characteristic. Conversely, if $x_1\in X_1$ and $x_2\in X_2$ are points of almost indefinite characteristic, then the point $(x_1,x_2)$ exists in $X_1\times X_2$ by Theorem \ref{thm-image-of-tau-in-x_1-times-x_2}. The closed subscheme $\overline{(x_1,x_2)}$ is affine if and only if both $\overline{x_1}$ and $\overline{x_2}$ are affine. By Lemma \ref{lemma: product of blueprints generated by their units}, the scheme $\overline{(x_1,x_2)}^+_\Z$ is flat and non-empty if and only if both $\overline{x_1}^+_\Z$ and $\overline{x_2}^+_\Z$ are flat and non-empty. This shows that $(x_1,x_2)$ is pseudo-Hopf if and only if both $x_1$ and $x_2$ are pseudo-Hopf. To complete the proof of $\cZ(X_1\times X_2)=\cZ(X_1)\times\cZ(X_2)$, note that $(x_1,x_2)$ is of minimal rank if and only if both $x_1$ and $x_2$ are of minimal rank. Since $\pr_i(x_1,x_2)=x_i$, it is clear from the preceding that the first diagram of the proposition is commutative.

 Since $\cZ(X_1\times X_2)=\cZ(X_1)\times\cZ(X_2)$, we have an isomorphism
 \[
  (X_1\times X_2)^\prk \ = \ \coprod_{(x_1,x_2)\in\cZ(X_1\times X_2)} \prerk{(x_1,x_2)} \ \simeq \ \Bigl( \coprod_{x_1\in\cZ(X_1)} \prerk{x_1} \Bigr) \times \Bigl( \coprod_{x_2\in\cZ(X_2)} \prerk{x_2} \Bigr) \ = \ X_1^\prk\times X_2^\prk
 \]
 by Lemma \ref{lemma: tensor product of inverse closures}.
 It is obvious that this identification makes the right square of the second diagram in the proposition commutative.

 By the preceding and Lemma \ref{lemma: unit scheme of products}, we have canonical isomorphisms
 \[
  (X_1\times X_2)^\rk \ = \ ((X_1\times X_2)^\prk)^\px \ \simeq \ (X_1^\prk\times X_2^\prk)^\px \ \simeq \ (X_1^\prk)^\px\times (X_2^\prk)^\px \ = \ X_1^\rk\times X_2^\rk.
 \]
 It is obvious that the left square of the second diagram of the proposition commutes.
\end{proof}

As a side product of the equality $\cZ(X_1\times X_2)=\cZ(X_1)\times\cZ(X_2)$, we have the following fact.

\begin{cor}
 Let $X_1$ and $X_2$ be connected blue schemes. Then $\rk (X_1\times X_2)=\rk X_1+\rk X_2$.\qed
\end{cor}

For brevity, we will denote $\Spec B$ by $\ast_B$, which should emphasize that $\ast_B$ is the terminal object in $\Sch_B$, the category of blue schemes with base scheme $\ast_B=\Spec B$. In particular, $\ast_\Fun$ is the terminal object of $\BSch$. Note that $\ast_B$ is the terminal object of both $\Sch_B$ and $\Sch_B^+$ if $B$ is a semiring.

\begin{thm}\label{thm: sch_t is cartesian}
 The category $\TSch$ is Cartesian. Its terminal object is $\ast_\Fun$ and the product of two blue schemes in $\TSch$ is represented by the product in $\BSch$.
\end{thm}

\begin{proof}
 We begin to show that $\ast_\Fun$ is terminal. First note that $\ast_\Fun$ is of pure rank, i.e.\ $\ast_\Fun^\rk=\ast_\Fun$, and that $\ast_\Fun^+=\ast_\N$. Let $X$ be a blue scheme. Then there are a unique morphism $\varphi^\rk:X^\rk\to\ast_\Fun^\rk$ and a unique morphism $\varphi^+:X^+\to\ast_\N$. Thus uniqueness is clear. It is easily verified that $(\varphi^\rk,\varphi^+)$ is a Tits morphism. 

 To prove that the product of two blue schemes $X_1$ and $X_2$ in $\BSch$ together with the canonical projections $\pr_i:X_1\times X_2\to X_i$ (which are Tits by Proposition \ref{prop: product of rank spaces}) represents the product in $\TSch$, consider two Tits morphisms $\varphi_1:Y\to X_1$ and $\varphi_2:Y\to X_2$ for a blue scheme $Y$, i.e.\ $\varphi_i=(\varphi_i^\rk,\varphi_i^+)$ for $i=1,2$. We define $\varphi^\rk$ as $\varphi_1^\rk\times\varphi_2^\rk:Y^\rk\to X_1^\rk\times X_2^\rk$ and $\varphi^+$ as $\varphi_1^+\times\varphi_2^+:Y^+\to X_1^+\times X_2^+$. We have to show that the pair $\varphi=(\varphi^\rk,\varphi^+)$ is a Tits morphism $\varphi:Y\to X_1\times X_2$. Once this is shown, it is clear that $\varphi_i=\pr_i\circ\varphi$ and that $\varphi$ is unique with this property.

 To verify that $\varphi$ is Tits, consider for $i=1,2$ the diagram
 \[
  \xymatrix@R=2pc@C=4pc{                &  Y^+_\Z \ar[rr]^{\varphi^+_\Z} \ar[ddrr]|{\textcolor{white}{\rule{4em}{1.5em}}}^(0.35){\varphi^+_{i,\Z}} &          &   X_{1,\Z}^+\Stimes X_{2,\Z}^+ \ar[dd]_{\pr_i}   \\
                         Y_\Z^{\rk,+} \ar[ur]^{\rho_{Y,\Z}^+} \ar[rr]^{\varphi_\Z^{\rk,+}} \ar[ddrr]_{\varphi_{i,\Z}^{\rk,+}}  &        &   X_{1,\Z}^{\rk,+}\Stimes X_{2,\Z}^{\rk,+}\ar[dd]_{\pr_i}\ar[ur]_{\rho^+_{X,\Z}} \\
                                        &          &         &   X_{i,\Z}^+  \\
                                         &         &   X_{i,\Z}^{\rk,+} \ar[ur]_{\rho^+_{X_i,\Z}}  }  
 \]
 which we know to commute up to the top square. The top square commutes because both $\rho_{X,\Z}^+\circ\varphi^{\rk,+}_\Z$ and $\varphi_\Z^+\circ\rho_{Y,\Z}^+$ equal the canonical morphisms $Y_\Z^{\rk,+}\to X_{1,\Z}^+\Stimes X_{2,\Z}^+$ that is associated to the morphisms $\rho_{X_i,\Z}^+\circ\varphi^{\rk,+}_{i,\Z}=\varphi_{i,\Z}^+\circ\rho_{Y,\Z}^+:Y\to X_{i,\Z}^+$ for $i=1,2$. This shows that $\varphi$ is Tits.
\end{proof}

\begin{prop}\label{prop: sch_fun-t is cartesian}
 The category $\BTSch$ is Cartesian. Its terminal object is $\ast_\Fun$ and the product of two blue schemes in $\BTSch$ is represented by the product in $\BSch$.
\end{prop}

\begin{proof}
 Since $\ast_\Fun$ is of pure rank, the unique morphism $\varphi:X\to \ast_\Fun$ is Tits for each blue scheme $X$ by Proposition \ref{prop: morphisms that are tits}. Thus $\ast_\Fun$ is a terminal object in $\BTSch$.

 We show that the product $X_1\times X_2$ of two blue schemes $X_1$ and $X_2$ in $\BSch$ represents the product in $\BTSch$. First note that the canonical projections $\pi_i:X_1\times X_2\to X_i$ are Tits by Proposition \ref{prop: product of rank spaces}. Let $\varphi_1:Y\to X_1$ and $\varphi_2:Y\to X_2$ be two locally algebraic Tits morphisms and $\varphi=\varphi_1\times\varphi_2:Y\to X_1\times X_2$ the canonical morphism. If $y\in\cZ(Y)$, then $\varphi(y)=(\varphi_1(y),\varphi_2(y))$ is an element of $\cZ(X_1)\times\cZ(X_2)=\cZ(X_1\times x_2)$. Thus $\varphi$ is Tits. This shows that $X_1\times X_2$ is the product of $X_1$ and $X_2$ in $\BTSch$,
\end{proof}

\begin{prop}\label{prop: sch^rk_fun is cartesian}
 The category $\rkBSch$ is Cartesian. Its terminal object is $\ast_\Fun$ and the product of two blue schemes in $\rkBSch$ is represented by the product in $\BSch$.
\end{prop}

\begin{proof}
 Since $\rkBSch$ is a full subcategory of $\BSch$, it suffices to show that the terminal object of $\BSch$ and the product of two schemes of pure rank (taken in $\BSch$) are in $\rkBSch$. The terminal object $\ast_\Fun$ is of pure rank. If $X_1$ and $X_2$ are of pure rank, i.e.\ discrete, reduced and of almost indefinite characteristic, then $X_1\times X_2$ is also discrete, reduced and of almost indefinite characteristic. This proves the proposition.
\end{proof}

We collect the results of this section in the following theorem, which gives an overview of the Cartesian categories and the Cartesian functors between them, which will be of importance for the rest of this paper. Before we can state it, we fix some notation. We denote by $\iota:\rkBSch\hookrightarrow\BTSch$ and $\iota:\BTSch\hookrightarrow\BSch$ the inclusions as subcategories.

The functor $\cT:\BTSch\to\TSch$ is the identity on objects and sends a locally algebraic Tits morphism $\varphi:X\to Y$ to the Tits morphism $(\varphi^\rk,\varphi^+):X\to Y$ (cf.\ Proposition \ref{prop: morphisms that are tits}). Since a morphism $f:B\to C$ of blueprints is uniquely determined by the morphism $f^+: B^+\to C^+$ of semirings and since morphisms of blue schemes are locally algebraic (see \cite[Thm.\ 3.23]{blueprints1}), a morphism $\varphi:X\to Y$ of blue schemes is uniquely determined by its base extension $\varphi^+:X^+\to Y^+$. This means that $\cT:\BTSch\to\TSch$ is faithful and that we can, in fact, consider $\BTSch$ as a subcategory of $\TSch$.

The functor $\cW:\TSch\to\Sets$ is the Weyl extension, which factors through $\rkBSch$ (cf.\ Section \ref{subsection: tits morphisms}). For any semiring $k$, the  base extension $(\blanc)^+_k:\TSch\to\Sch_k^+$ to semiring schemes over $k$ factors through $\SSch$.

\begin{thm}\label{thm: overview of cartesian categories and functors}
 The diagram 
 \[
  \xymatrix@C=2,5pc@R=1,5pc{\rkBSch \ar[rrrr]^{\id} \arincl{[dr]^\iota} &        &       &       & \rkBSch \ar[rr]^\cW && \Sets \\
                                & \BTSch \arincl{[rr]^\cT}\arincl{[dr]^\iota} &       & \TSch \ar[ur]^(0.4){(\blanc)^\rk} \ar[dr]^{(\blanc)^+} \\
                                &        & \BSch \ar[rr]^{(\blanc)^+} &       & \Sch_\N^+ \ar[rr]^{(\blanc)_k^+ } && \Sch_k^+}
 \]
 is an essentially commutative diagram of Cartesian categories and Cartesian functors where $k$ is an arbitrary semiring. 
\end{thm}

\begin{proof}
 All categories are Cartesian: the terminal object of $\Sets$ is the one-point set $\ast$ and the product is the Cartesian product of sets; the terminal object of $\BSch$ is $\ast_\Fun$ and the product is the fibre product $\blanc\times_\Fun\blanc$; for any semiring $k$, the terminal object of $\Sch_k^+$ is $\ast_k$ and the product is the fibre product $\blanc\Stimes_k\blanc$; that $\TSch$, $\BTSch$ and $\rkBSch$ are Cartesian is subject of in Theorem \ref{thm: sch_t is cartesian}, Proposition \ref{prop: sch_fun-t is cartesian} and Proposition \ref{prop: sch^rk_fun is cartesian}, respectively.

 All functors are Cartesian: since the terminal object and the product in $\TSch$, $\BTSch$ and $\rkBSch$ coincides with the terminal object and the product in $\BSch$, the inclusions $\iota:\rkBSch\hookrightarrow\BTSch$, $\iota:\BTSch\hookrightarrow\BSch$ and $\cT:\BTSch\to\TSch$ are Cartesian; the identity functor $\id:\rkBSch\to\rkBSch$ is evidently Cartesian; $(\blanc)^\rk:\TSch\to\rkBSch$ and is Cartesian by Proposition \ref{prop: product of rank spaces} and since $\ast_\Fun$ is of pure rank; $(\blanc)^+:\TSch\to\SSch$ and $(\blanc)^+:\BSch\to\SSch$ are Cartesian since $(X\times Y)^+=X^+\Stimes Y^+$; $(\blanc)_k^+:\SSch\to \Sch_k^+$ is Cartesian since $(X\Stimes Y)^+_k=X^+_k\Stimes_k Y^+_k$; and $\cW:\rkBSch\to\Sets$ is Cartesian since the underlying set of $\ast_\Fun$ is the one-point set $\ast$ and since every point of a scheme of pure rank is of almost indefinite characteristic and therefore $\cW(X_1\times X_2)$ is the Cartesian product of $X_1,X_2\in\rkBSch$ by Theorem \ref{thm-image-of-tau-in-x_1-times-x_2}.

 The composition $(\blanc)^\rk\circ\cT\circ\iota:\rkBSch\to\rkBSch$ is isomorphic to the identity functor because $X^\rk\simeq X$ for a blue scheme of pure rank and $\varphi^\rk=\varphi$ for a morphism between blue schemes of pure rank. The functors $(\blanc)^+\circ\cT:\BTSch\to\SSch$ and $(\blanc)^+\circ\iota:\BTSch\to\SSch$ are isomorphic because in both cases a blue scheme $X$ is sent to $X^+$ and a locally algebraic Tits morphism $\varphi:X\to Y$ is sent to $\varphi^+: X^+\to Y^+$. This finishes the proof of the theorem.
\end{proof}


\subsection{Tits-Weyl models}
\label{subsection: tits-weyl models}

\begin{df}
 A \emph{Tits monoid} is a monoid in $\TSch$. The \emph{Weyl monoid} of a Tits monoid $(G,\mu)$ is the monoid $(\cW(G),\cW(\mu))$ in $\Sets$. In case that $\cW(G)$ is a group, we call it also the \emph{Weyl group of $G$}.
\end{df}

Often, we will suppress the semigroup law from the notation if it is not necessarily needed. A \emph{Tits model} of a group scheme $\cG$ is a Tits monoid $G$ whose base extension $G_\Z^+$ to schemes is isomorphic to $\cG$ as a group scheme. 

\begin{rem}
 Note that the definition of a Tits model given here differs from that in \cite{Deitmar11}. While we define a Tits model of a group scheme to be a monoid in $\TSch$ whose base extension is isomorphic to the group scheme, the definition of a Tits model of a Chevalley group scheme $\cG$ in \cite{Deitmar11} means a cancellative blue scheme $G$ such that $G_\Z^+\simeq\cG$ (as schemes) and such that the number of morphisms $\ast_\Fun\to G$ coincides with the number of elements in the Weyl group of $\cG$. The notion of a Tits-Weyl model as defined below will combine these two aspects in a certain way.
\end{rem}

Let $(G,\mu)$ be a Tits monoid with identity $\epsilon:\ast_\Fun\to G$. Let $e$ be the image point of $\epsilon^\rk$ in $G^\rk$ and $\fe=\overline{\{e\}}=\Spec\kappa(e)$ the closed subscheme of $G^\rk$ with support $e$. We call $\fe$ the \emph{Weyl kernel of $G$}. 

\begin{lemma}
 Let $(G,\mu)$ be a Tits monoid and $\fe$ its Weyl kernel. The semigroup law $\mu^\rk$ of $G^\rk$ restricts to a semigroup law $\mu_\fe$ of $\fe$, which turns $\fe$ into a commutative group in $\rkBSch$.
\end{lemma}

\begin{proof}
 Let $\epsilon:\ast_\Fun\to G$ be the identity of $G$. Since $\epsilon^\rk$ is a both-sided identity for $\mu^\rk$, the semigroup law of $G^\rk$ restricts to a semigroup law $\mu_\fe^\rk:\fe\times\fe\to\fe$. Since $\fe$ is of pure rank, $\mu_\fe:\fe\times\fe\to\fe$ is a Tits morphism and thus a semigroup law for $\fe$ in $\rkBSch$.

 We verify that this semigroup law is indeed a commutative group law. Its identity is the restriction of $\epsilon^\rk$ to $\epsilon_\fe:\ast_\Fun\to\fe$. Consider the comultiplication $m=\Gamma\mu_\fe:\kappa(e)\to\kappa(e)\otimes_\Fun\kappa(e)$. It sends an element $a\in\kappa(e)$ to an element $b\otimes c=m(a)$ of $\kappa(e)\otimes_\Fun\kappa(e)$. The coidentity yields a commutative diagram
 \[
  \xymatrix@R=1pc@C=2pc{                                                                    &&& \kappa(e)\ar[dlll]_{m}\ar[drrr]^{m}\ar[dd]^{\id} \\ 
                        \kappa(e)\otimes_\Fun\kappa(e)\ar[drr]^{\quad\Gamma\epsilon_\fe\otimes\id} &&&&&                                 & \kappa(e)\otimes_\Fun\kappa(e)\ar[dll]_{\id\otimes\Gamma\epsilon_\fe\quad } ,\\ 
                                        &&\Fun\otimes_\Fun\kappa(e)\ar@{}[r]|(0.62)\simeq     &\kappa(e)          & \kappa(e)\otimes_\Fun\Fun\ar@{}[l]|(0.62)\simeq   }
 \]
 which means that $1\otimes a=1\otimes c$ and $a\otimes 1=b\otimes 1$. Thus, $m(a)=a \otimes a$, which implies that $\mu_\fe$ is commutative. The inverse $\iota_\fe$ of $\mu_\fe$ is defined by the morphism $\Gamma\iota_\fe:\kappa(e)\to\kappa(e)$ that sends $0$ to $0$ and $a$ to $a^{-1}$ if $a\neq 0$.
\end{proof}

\begin{lemma}\label{lemma: fe is diagonalizable}
 Let $G$ be a Tits monoid that is locally of finite type and $\fe$ its Weyl kernel. Then the group scheme $\fe^+_\Z$ is diagonalizable, i.e.\ a closed subgroup of a split torus over $\Z$.
\end{lemma}

\begin{proof}
 Since $\fe^+_\Z$ is affine, the claim of the lemma means that the global sections $B=\kappa(\fe)^+_\Z$ of $\fe^+_\Z$ are a quotient of $\Z[T_1^{\pm1},\dotsc,T_n^{\pm1}]$ by some ideal where $n\in\N$. The global sections $\kappa(\fe)$ of $\fe$ form a blue field, and $B$ is generated by the image of the multiplicative group $\kappa(\fe)^\times$in $B$. Since $B$ is a finitely generated algebra, it is already generated by a finitely generated subgroup $H$ of $\kappa(\fe)^\times$. In other words, $B$ is a quotient of the group ring $\Z[H]$. By the structure theorem for finitely generated abelian groups, $H$ is the quotient of a finitely generated free abelian group of some rank $n$. Thus $\Z[H]$, and therefore $B$, is a quotient of $\Z[T_1^{\pm1},\dotsc,T_n^{\pm1}]$. 

 This shows that $\fe^+_\Z$ is a closed subscheme of a torus. Therefore $\fe_\Z^+$ can neither have a unipotent component nor a semisimple component. As a flat commutative group scheme it must be an extension of a constant group scheme by a torus. Since $\fe$ is commutative, $\fe^+_\Z$ is commutative and therefore diagonalizable.
\end{proof}

The Weyl kernel $\fe$ of $G$ is the identity component of $G^\rk$, i.e.\ the connected component that contains the image of the identity $\epsilon^\rk:\ast_\Fun\to G^\rk$. Thus $\fe_\Z^+$ equals the identity component $(G^{\rk,+}_\Z)_0$ of $G^{\rk,+}_\Z$, which is a normal subgroup of $G^{\rk,+}_\Z$, and we can consider the quotient group $G^{\rk,+}_\Z/\fe^+_\Z$. In the following, we are interested in comparing this quotient to the Weyl group of $G$. 

We recall some notions from the theory of group schemes. Let $\cG$ be a group scheme of finite type. For a torus $T$ of $\cG$, we denote its centralizer by $C(T)$ and its normalizer by $N(T)$. We define $\cW(T)=N(T)/C(T)$, the \emph{Weyl group relative to $T$}, which is quasi-finite, \'etale and separated over $\Spec\Z$ (cf.\ \cite[]{SGA3II}). This means, in particular, that $\Gamma\cW(T)$ is a flat $\Z$-module of finite rank, or, in other words, that $\cW(T)$ is a finite group scheme. Since $\Spec\Z$ has no non-trivial connected finite \'etale extensions (cf.\ \cite[Section 6.4]{Conrad11}), $\cW(T)$ is indeed a constant group scheme over $\Z$. 

A \emph{maximal torus of $\cG$} is a subgroup $T$ of $\cG$ that is a torus such that for every geometric point $\bar s:\Spec\bar k\to\Z$ of $\Spec\Z$, the torus $T_{\bar s}$ is maximal in $\cG_{\bar s}$ (cf.\ \cite[XII.1.3]{SGA3II}). A maximal torus $T$ of $\cG$ is always split (cf.\ \cite[Section 6.4]{Conrad11}). Note that in general, $\cG$ does not have a maximal torus. If $T$ is a maximal torus of $\cG$, then the rank of $T$ is called the \emph{reductive rank of $\cG$} and $C(T)$ is called a \emph{Cartan subgroup of $\cG$}. In case of a maximal torus, we call $\cW(T)$ simply the \emph{Weyl group of $\cG$}. If $\cG$ is affine smooth and $T$ a maximal torus of $\cG$, then the geometric fibre $\cW(T)_{\bs}$ is the Weyl group of $\cG_\bs$, which is also called the \emph{geometric Weyl group (over $\bar k$)}. Since $\cW(T)$ is a constant group scheme, the group $\cW(T)(R)$ of $R$-rational points does not depend on the chosen ring $R$. We call this group the \emph{ordinary Weyl group of $\cG$}.



Let $G$ be a Tits model of $\cG$, i.e.\ we identify $G_\Z^+$ with $\cG$, and let $\fe$ be the Weyl kernel of $G$. A consequence of Lemma \ref{lemma: fe is diagonalizable} is that $\fe^+_\Z$ contains a unique maximal torus $T$ (cf.\ \cite[XII.1.12]{SGA3II}). We call $T$ the \emph{canonical torus of $\cG$ (with respect to $G$)}. Then $\fe^+_\Z$ is contained in the centralizer $C(T)$ of $T$ in $\cG$. Since $\fe^+_\Z$ is a normal subgroup of $G^{\rk,+}_\Z$ and $T$ is the unique maximal torus of $\fe^+_\Z$, the subgroup $G^{\rk,+}_\Z$ normalizes $T$ in $\cG$, which means that $G^{\rk,+}_\Z$ embeds into $N(T)$. Thus we obtain a morphism $\Psi_\fe: G^{\rk,+}_\Z/\fe^+_\Z\to\cW(T)$ of group schemes.


\begin{df}
 Let $\cG$ be an affine smooth group scheme of finite type. A \emph{Tits-Weyl model of $\cG$} is a Tits model $G$ of $\cG$ such that the canonical torus $T$ is a maximal torus of $\cG$ and such that $\Psi_\fe: G^{\rk,+}_\Z/\fe^+_\Z\to\cW(T)$ is an isomorphism of group schemes where $\fe$ is the Weyl kernel of $G$. 
\end{df}


Before we can collect the first properties of a Tits-Weyl model of a group scheme $\cG$, we have to fix some more notation. We define the \emph{rank of $G$} as the rank of the connected blue scheme $G_0$ (as a blue scheme). Note that the rank of each connected component of $G$ is equal to the rank of the identity component $G_0$ of $G$ since each connected component of $G$ is a torsor of $G_0$. 

Let $\cG$ be an affine smooth group scheme of finite type with maximal torus $T$. In general, the Weyl group $W$ cannot be realized as the $\Z$-rational points of a finite subgroup of $\cG$. This is an obstacle to realize $W$ as the $\Fun$-points of a group scheme over $\Fun$ as suggested by Tits in his '56 paper \cite{Tits56} (for more explanation on this, cf.\ \cite[Problem B]{L09} and \cite{CC08}). However, in case $\cG$ is a split reductive group scheme, Tits describes himself in his paper \cite{Tits66} from '66 a certain extension $\widetilde{W}$ of $W$, called the \emph{extended Weyl group} or \emph{Tits group}, which can be realized as the $\Z$-valued points of a finite flat group scheme $\widetilde{\cW}(T)$ of $\cG$. Namely, $\widetilde{\cW}(T)$ is defined as $N(T)(\Z)$-translates of the $2$-torsion subgroup $T[2]$ of $T$ where $N(T)$ is the normalizer of $T$. This yields a short exact sequence of group schemes
\[
 1 \quad \longrightarrow \quad T[2] \quad \longrightarrow \quad \widetilde{\cW}(T) \quad \longrightarrow \quad {\cW}(T) \quad \longrightarrow \quad 1,
\]
and thus an isomorphism $W\simeq \widetilde{\cW}(\Z)/T(\Z)$ since $T(\Z)=T[2](\Z)$ is a $2$-torsion group.

Let $G$ be a Tits monoid and $S$ a blue scheme. Since $G$ is a monoid in $\TSch$, the set $G^\sT(S)=\Hom_\cT(S,G)$ of $S$-rational Tits points of $G$ is a monoid in $\Sets$. If $S=\Spec B$, we also write $G^\sT(B)$ for $G^\sT(S)$.

\begin{thm}\label{thm: properties of tits-weyl groups}
 Let $\cG$ be an affine smooth group scheme of finite type. If $\cG$ has a Tits-Weyl model $G$, then the following properties hold true.
 \begin{enumerate}
  \item\label{part1} The Weyl group $\cW(G)$ is canonically isomorphic to the ordinary Weyl group $W$ of $\cG$.
  \item\label{part2} The rank of $G$ is equal to the reductive rank of $\cG$.
  \item\label{part3} The group $G^\sT(\Fun)$ of $\Fun$-rational Tits points of $G$ is a subgroup of $\cW(G)$.
  \item\label{part4} If $\cG$ is a split reductive group scheme, then $G^\sT(\Funsq)$ is canonically isomorphic to the extended Weyl group $\widetilde W$ of $\cG$.
 \end{enumerate}
\end{thm}

\begin{proof}

 We prove \eqref{part1}. The ordinary Weyl group $W$ equals the group $\cW(T)(\C)$ of $\C$-rational points of the geometric Weyl group over $\C$. The isomorphisms 
 \[
  \cW(\cG)(\C) \ \simeq \ N(T)(\C)/C(T)(\C) \ \simeq \ G^\rk(\C)/\fe(\C)
 \]
 show that the elements of $\cW(\cG)(\C)$ stay in one-to-one correspondence with the connected components of $G^\rk$, which in turn is the underlying set of $\cW(G)$. It is clear that the group structures coincide. 

 We prove \eqref{part2}. Let $\fe$ be the Weyl kernel of $G$. The rank of $G$ equals the dimension of the variety $\fe^+_\Q$ over $\Q$. By Lemma \ref{lemma: fe is diagonalizable}, $\fe^+_\Q$ is a closed subgroup of a split torus, which means that it is an extension of $T_\Q$ by a finite group scheme where $T$ is the maximal torus of $\fe^+_\Z$. Therefore the dimension of $\fe^+_\Q$ equals the rank of $T$, which is the reductive rank of $\cG$ since $T$ is a maximal torus of $\cG$.

 We prove \eqref{part3}. We denote as usual $\Spec\Fun$ by $\ast_\Fun$. A Tits morphism $\varphi:\ast_\Fun\to G$ is determined by the set theoretical image of $\varphi^\rk:\ast_\Fun\to G^\rk$ since there is at most one morphism from a blue field, i.e.\ from the residue field of the image point, to $\Fun$. Note that necessarily $\varphi^+=\varphi^{\rk,+}$. Thus $G^\sT(\Fun)$ is a subset of $\cW(G)$ and it inherits its semigroup structure from $\cW(G)$. Since $\cW(G)$ is a finite group, $G^\sT(\Fun)$ is also a group.

 We prove \eqref{part4}. If $\cG$ is a split reductive group, then the subgroups $T$, $\fe^+_\Z$ and $C(T)$ coincide, and consequently also $N(T)$ and $G_\Z^{\rk,+}$ coincide. We write briefly $N$ for $N(T)$. For a point $x$ of $G^\rk$, the scheme $\{x\}^+_\Z$ is a translate $nT$ of $T$ by some element $n\in N(\Z)$. The scheme $nT$ is isomorphic to $T$, i.e.\ $nT$ is isomorphic to the spectrum of $\Z[T_1^{\pm1},\dotsc,T_r^{\pm1}]$,  where $r$ is the rank of $T$. Its largest blue subfield is $\Funsq[T_1^{\pm1},\dotsc,T_r^{\pm1}]$. The map $\kappa(x)\to\kappa(x)^+_\Z\simeq\Z[T_1^{\pm1},\dotsc,T_r^{\pm1}]$ factorizes through $\kappa(x)_\inv=\kappa(x)\otimes_\Fun\Funsq$ since $\kappa(x)_\Z^+$ is with inverses. Since $\kappa(x)_\inv$ is a blue field with inverses, it must be equal to $\Funsq[T_1^{\pm1},\dotsc,T_r^{\pm1}]\subset\Z[T_1^{\pm1},\dotsc,T_r^{\pm1}]$. Since every morphism $\kappa(x)\to \Funsq$ factorizes uniquely through $\kappa(x)_\inv\simeq\Funsq[T_1^{\pm1},\dotsc,T_r^{\pm1}]$, the morphisms $\kappa(x)\to\Funsq$ stay in one-to-one correspondence with the morphisms $\kappa(x)^+_\Z\to \Z$, i.e.\ with $nT(\Z)=nT[2](\Z)$.

 Note that similar to the case of $\Fun$-rational points, a Tits morphism $\varphi:\Spec\Funsq\to G$ is determined by $\varphi^\rk$. This means that every $\Funsq$-rational Tits point $\varphi:\Spec\Funsq\to G$ with image $x$ is given by a morphism $\kappa(x)\to\Funsq$ of blueprints. Therefore $G^\sT(\Funsq)$ is isomorphic to the subgroup of $G(\Z)$ that is generated by the translates $nT[2](\Z)$ where $n$ ranges through $N(\Z)$. This subgroup is by definition the extended Weyl group $\widetilde W$ of $\cG$. This finishes the proof of the theorem.
\end{proof}


\subsection{Groups of pure rank}
\label{subsection: tits-weyl models of pure rank}

In this section, we will explain first examples of Tits models, namely, of constant group schemes and split tori. All these examples will be of pure rank, thus the group law will be indeed a locally algebraic morphism of blue schemes, which makes the description particularly easy. The Tits monoids appearing in this section are indeed group objects in $\rkBSch$. In case of a torus, or, more generally, of a semidirect product of a torus by a constant group scheme satisfying a certain condition, the described Tits model is a Tits-Weyl model.

 \subsubsection*{Constant groups}

Let $G$ be a finite group. Then the constant group scheme $G_\Z$ that is associated to $G$ is defined as the scheme 
\[
 G_\Z \quad = \quad \Spec\ \prod_{g\in G} \Z
\]
together with the multiplication $\mu_\Z: G_\Z\times G_\Z\to G_\Z$ that is defined by the comultiplication
\[
 \begin{array}{cccc}
  \Gamma\mu_\Z: & \displaystyle\prod_{g\in G}\Z & \longrightarrow & \displaystyle\prod_{g\in G}\Z\Sotimes_\Z\prod_{g\in G}\Z \\[20pt]
                  &   (a_g)_{g\in G}              & \longmapsto     & \displaystyle\sum_{g_1,g_2\in G} a_{g_1g_2} e_{g_1}\otimes e_{g_2}
 \end{array}
\]
where $e_h$ is the element $(a_g)_{g\in G}$ of $\prod_{g\in G}\Z$ with $a_g=1$ if $g=h$ and $a_g=0$ otherwise.

This group scheme descends to a group object $G_\Fun$ in $\rkBSch$. Namely, define the scheme $G_\Fun$ as $\Spec\prod_{g\in G}\Fun$, which is obviously of pure rank. Then, we have indeed canonical isomorphisms $(G_\Fun)_\Z\simeq (G_\Fun)^+_\Z\simeq G_\Z$, which justifies our notation. The group law $\mu_\Z$ descends to the group law $\mu_\Fun:G_\Fun\times G_\Fun\to G_\Fun$ that is defined by the comultiplication
\[
 \begin{array}{cccc}
  \Gamma\mu_\Fun: & \displaystyle\prod_{g\in G}\Fun & \longrightarrow & \displaystyle\prod_{g\in G}\Fun\otimes_\Fun\prod_{g\in G}\Fun\ =\ \bpquot{\bigl(\prod_{g\in G}\Fun\bigr)\times\bigl(\prod_{g\in G}\Fun\bigr)}{\cR}\\[20pt]
                  &   (a_g)_{g\in G}                & \longmapsto     & \displaystyle (a_{g_1g_2})_{g_1,g_2\in G}
 \end{array}
\]
where $\cR$ is the pre-addition that is generated by the relations $(a,0)\=(0,0)\=(0,a)$ for $a\in \prod_{g\in G}\Fun$. 

The morphism $\mu_\Fun:G_\Fun\times G_\Fun\to G_\Fun$ is indeed a group law: its identity is the morphism $\epsilon_\Fun:\ast_\Fun\to G_\Fun$ given by
\[
 \begin{array}{cccc}
  \Gamma\epsilon_\Fun: & \displaystyle\prod_{g\in G}\Fun & \longrightarrow & \Fun \\[20pt]
                       &  (a_g)_{g\in G}                 & \longmapsto     & a_e
 \end{array}
\]
where $e$ is the identity element of $G$ and its inverse is the morphism $\iota_\Fun: G_\Fun\to G_\Fun$ given by
\[
 \begin{array}{cccc}
  \Gamma\epsilon_\Fun: & \displaystyle\prod_{g\in G}\Fun & \longrightarrow &  \displaystyle\prod_{g\in G}\Fun. \\[20pt]
                       &  (a_g)_{g\in G}                 & \longmapsto     & (a_{g^{-1}})_{g\in G}
 \end{array}
\]
This shows that $G_\Fun$ together with $\mu_\Fun$ is a group object in $\rkBSch$ and therefore in $\TSch$. In particular, $G_\Fun$ is a Tits model of $G_\Z$.

The Weyl kernel $\fe$ of $G_\Fun$ is its identity component $G_{\Fun,0}=\Spec\Fun$. Thus the canonical torus of $G_\Fun$ equals the identity component $G_{\Z,0}$ of $G_\Z$, which is a maximal torus of $G_\Z$. Both, its centralizer and its normalizer is the whole group scheme $G_\Z\simeq G^{\rk,+}_\Z$. Thus the morphism $\Psi_\fe: G^{\rk,+}_\Z/\fe^+_\Z\to\cW(T)$ is an isomorphism only for the trivial group scheme $\ast_\Fun$. 


 \subsubsection*{Split tori}

We proceed with the description of a Tits-Weyl model of the split torus $\SG_{m,\Z}^r$ of rank $r$, which is $\Spec\Z[T_1^{\pm1},\dotsc,T_r^{\pm1}]^+$ as a scheme. Its group law $\mu_{\SG_{m,\Z}^r}:\SG_{m,\Z}^r\times \SG_{m,\Z}^r\to \SG_{m,\Z}^r$ is given by the comultiplication
\[
  \begin{array}{cccc}
  \Gamma\mu_{\SG_{m,\Z}^r}: & \Z[T_1^{\pm1},\dotsc,T_r^{\pm1}]^+ & \longrightarrow & \Z[(T_1')^{\pm1},\dotsc,(T_r')^{\pm1},(T_1'')^{\pm1},\dotsc,(T_r'')^{\pm1}]^+ \\
  \end{array}
\]
that maps $T_i$ to $T_i'\otimes T_i''$ for $i=1,\dotsc,r$.

This group scheme has the Tits model $(\G_{m,\Fun}^r,\mu)$ where $\G_{m,\Fun}^r=\Spec\Fun[T_1^{\pm1},\dotsc,T_r^{\pm1}]$ and $\mu: \G_{m,\Fun}^r\times \G_{m,\Fun}^r\to \G_{m,\Fun}^r$ given by the morphism
\[
  \begin{array}{cccc}
  \Gamma\mu: & \Fun[T_1^{\pm1},\dotsc,T_r^{\pm1}] & \longrightarrow & \Fun[(T_1')^{\pm1},\dotsc,(T_r')^{\pm1},(T_1'')^{\pm1},\dotsc,(T_r'')^{\pm1}] \\
  \end{array}
\]
that maps $T_i$ to $T_i'\otimes T_i''$ for $i=1,\dotsc,r$. Note that $\G_{m,\Fun}^r$ has precisely one point, which is of indefinite characteristic, and that $\G_{m,\Fun}^r$ is reduced. This means that $\G_{m,\Fun}^r$ is of pure rank and that $\mu$ is Tits. Its identity is the morphism $\epsilon:\ast_\Fun\to \G_{m,\Fun}^r$ given by the morphism $\Gamma\epsilon:\Fun[T_1^{\pm1},\dotsc,T_r^{\pm1}]\to\Fun$ that maps all elements $a\neq0$ to $1$ in $\Fun$. Its inverse is the morphism $\iota:\G_{m,\Fun}^r\to \G_{m,\Fun}^r$ given by the morphism $\Gamma\iota:\Fun[T_1^{\pm1},\dotsc,T_r^{\pm1}]\to\Fun[T_1^{\pm1},\dotsc,T_r^{\pm1}]$ that maps $T_i$ to $T_i^{-1}$ for $i=1,\dotsc,r$. Thus $\G_{m,\Fun}^r$ is a group object in $\rkBSch$ and therefore in $\TSch$.

The Weyl kernel of $\G_{m,\Fun}^r$ is $\G_{m,\Fun}^r$ itself. The canonical torus $T$ of $\G_{m,\Fun}^r$ is $\SG_{m,\Z}^r$, which is further its own normalizer $N$. Thus $T$ is a maximal torus of $\G_{m,\Fun}^r$ and the morphism $\Psi_\fe: (\G_{m,\Fun}^r)^{\rk,+}_\Z/\fe^+_\Z\to\cW(T)$ is an isomorphism of group schemes. This shows that $\G_{m,\Fun}^r$ is a Tits-Weyl model of $\SG_{m,\Z}^r$. Its Weyl group is the trivial group and consequently $(\G_{m,\Fun}^r)^\cT(\Fun)$ is the trivial group. Since the rank of $\SG_{m,\Z}^r$ is $r$, the group $(\G_{m,\Fun}^r)^\cT(\Funsq)$ is $(\Z/2\Z)^r$.

 \subsubsection*{Semi-direct products of split tori by constant group schemes}

Group schemes $N$ of the form $\SG_{m,\Z}^r\rtimes_\theta G_\Z$ appear as normalizers of split maximal tori in reductive group schemes and will be of a particular interest in the following. We will describe groups in $\rkBSch$ that base extend to the group, but we can already conclude for abstract reasons that a model of $N$ exists in $\rkBSch$ if $\theta$ descends to a morphism in $\rkBSch$. More precisely, the conjugation action $\theta:G_\Z\times^+_\Z \SG_{m,\Z}^r \to \SG_{m,\Z}^r $ restricts to morphisms
\[
 \theta_g: \quad \{g\} \ \times^+_\Z \SG_{m,\Z}^r \quad \simeq \quad \{g\} \ \times^+_\Z \SG_{m,\Z}^r \ \times^+_\Z \{g^{-1}\} \quad \stackrel{\mu\circ(\mu,\id)}{\longrightarrow} \quad \SG_{m,\Z}^r 
\]
for every $g\in G$. This yields blueprint morphisms $\Gamma\theta_g:\Z[T_1^{\pm1},\dotsc,T_r^{\pm1}]\to \Z[T_1^{\pm1},\dotsc,T_r^{\pm1}]$. If the images $\theta_g(T_i)$ are of the form $\prod_{j=1}^r T_j^{e_{i,j}(g)}$ for certain exponents $e_{i,j}(g)\in\Z$ for all $i,j=1,\dotsc,r$ and $g\in G$, then the action $\theta$ descends to an action $\tilde\theta$ of $G_\Fun$ on $\G_{m,\Fun}^r$. Thus we can form the semidirect product $\tilde N=\G_{m,\Fun}^r\rtimes_{\tilde\theta} G_\Fun$ in $\rkBSch$, which is an group scheme whose base extension to rings is $N$. By definition, $\SG_{m,\Z}^r$ is normal in $N$. Thus if the centralizer of $T$ is $T$ itself, then $\tilde N$ is a Tits-Weyl model of $N$. We summarize this in the following statement.

\begin{prop}
 Let $G$ be a group and $\theta:G_\Z\times^+_\Z \SG_{m,\Z}^r \to \SG_{m,\Z}^r$ be a group action that is defined by integers $e_{i,j}(g)$ as above. Then $\theta$ descends to a group action $\tilde\theta:G_\Fun\times\G_{m,\Fun}^r\to \G_{m,\Fun}^r$ and $\tilde N=\G_{m,\Fun}^r\rtimes_{\tilde\theta} G_\Fun$ is a group in $\rkBSch$ whose base extension to $\Z$ is $\tilde N^+_\Z=\SG_{m,\Z}^r\rtimes_\theta G_\Z$.

 If for every $g\in G$ different from the neutral element $e\in G$, the matrix $A(g)=(a_{i,j}(g))_{i,j=1,\dotsc,r}$ is different from the identity matrix, then $\tilde N$ is a Tits-Weyl model of $N$.\qed
\end{prop}


\section{Tits-Weyl models of Chevalley groups}
\label{section: tits-weyl models of chevalley groups}

In this section, we prove for a wide class of Chevalley groups that they have a Tits-Weyl model. Namely, for special linear groups, general linear groups, symplectic groups, special orthogonal groups and all Chevalley groups of adjoint type. As a first step, we establish Tits-Weyl model for the special linear groups. Tits-Weyl models for all other groups of the above list but the adjoint Chevalley groups can be obtained by a general principle for subgroups of the special linear groups, which is formulated in Theorem \ref{thm: tits-weyl models of subgroups}, a central result of this section. Finally, we find Tits-Weyl models of adjoint Chevalley groups by a close examination of explicit formulas for their adjoint representation over algebraically closed fields.

The precise meaning of the term Chevalley group varies within the literature. The original works of Chevalley refer to simple groups (cf.\ \cite{Chevalley55}) and, later, to semisimple groups (cf.\ \cite{Chevalley60}). When we refer to a Chevalley group in this text, we mean, in a more loose sense, a split reductive group scheme. But note that in fact almost all of the Chevalley groups that occur in the following are semisimple. As a general reference for background on Chevalley groups and split reductive group schemes, see SGA3 (\cite{SGA3I},\cite{SGA3II},\cite{SGA3III}), Demazure and Gabriel's book (\cite{Demazure-Gabriel70}) or Conrad's lecture notes (\cite{Conrad11}). There are plenteous more compact and readable accounts of root systems and Chevalley bases of Chevalley groups (for instance, cf.\ \cite{Carter89}).


\subsection{The special linear group}
\label{subsection: special linear group}

In this section, we describe a Tits-Weyl model $\SL_n$ of the special linear group $\SL_{n,\Z}^+$. 

To begin with, consider a closed subscheme of $\SA_\Z^n$ of the form $\cX=\Spec\Z[T_1,\dotsc,T_n]^+/\cI$ where $\cI$ is an ideal of $\Z[T_1,\dotsc,T_n]^+$. The set 
\[
 \cR_\cI \quad = \quad \Bigl\{\ \sum a_i\=\sum b_j \ \Bigl| \ \sum a_i-\sum b_j\ \in \ \cI \ \Bigr\}
\]
is a pre-addition for $\Fun[T_1,\dotsc,T_n]$ and defines a blueprint $B=\bpquot{\Fun[T_1,\dotsc,T_n]}{\cR_\cI}$. We call the blue scheme $X=\Spec B$ an \emph{$\Fun$-model of the scheme $\cX$}. It satisfies $X^+_\Z\simeq \cX$. Further, the canonical morphism $\Fun[T_1,\dotsc,T_n]\to B$ of blueprints defines a closed embedding $\iota:X\to\A_\Fun^n$, and $\iota^+_\Z$ is equal to the embedding of $\cX=X^+_\Z$ as closed subscheme of $\SA_\Z^n$.

The underlying topological space of $X$ is a subspace of the underlying topological space of $\A^n_\Fun$. Recall from Example \ref{ex: subschemes of affine space over f1} that the prime ideals of $\Fun[T_1,\dotsc,T_n]$ are of the form $\fp_I=(T_i)_{i\in I}$ where $I$ ranges through all subsets of $\un=\{1,\dotsc,n\}$. Thus the underlying topological space of $\A^n_\Fun$ is finite and completely determined by the rule $\fp_I\leq \fp_{I'}$ if and only if $I\subset I'$ (cf.\ section \ref{subsection-sober-and-locally-finite-spaces}).

In particular, this applies to the special linear group $\SL_{n,\Z}^+$, i.e.\ the scheme $\Spec\Z[\SL_n]^+$ together with the group law $\mu_\Z^+:\SL_{n,\Z}^+\Stimes_\Z\SL_{n,\Z}^+\to\SL_{n,\Z}^+$ where $\Z[\SL_n]^+=\Z[T_{i,j}]_{i,j\in\un}^+/\cI$ for the ideal $\cI$ that is generated by the element 
\[
 \sum_{\sigma\in S_n} \ \Bigl( \ \sign(\sigma) \cdot \ \prod_{i=1}^n \ T_{i,\sigma(i)} \ \Bigr) \ - \ 1
\]
(which expresses the condition that the determinant of a matrix $(a_{i,j})$ equals $1$) and where $\mu_\Z^+$ is defined by the comultiplication
\[
 \begin{array}{cccc}
    m_\Z^+ \ = \ \Gamma\mu_\Z^+: & \Z[\SL_n]^+& \longrightarrow & \bigl(\Z[\SL_n]^+ \bigr)\Sotimes_\Z \bigl(\Z[\SL_n]^+ \bigr). \\[5pt]
                                               &    T_{i,j}         & \longmapsto         & \displaystyle \sum_{k=1}^n \ T'_{i,k} \otimes T''_{k,j} \qquad
 \end{array}
\]
Thus $\SL_{n,\Z}^+$ is a closed subscheme of $\SA^{n^2}_\Z$, and therefore has an $\Fun$-model $\SL_{n}=\Spec \bpquot{\Fun[T_{i,j}]}{\cR_\cI}$.

Before we describe the group law for $\SL_{n}$ in $\TSch$, we determine the rank space of $\SL_{n}$. Since $\SL_{n}$ is a closed subscheme of $\A^{n^2}$, each point of $\SL_{n}$ is of the form $\fp_I=(T_{i,j})_{(i,j)\in I}$ where $I$ is a subset of $\un^2$. We write $\fp^\sigma=\fp_{I(\sigma)}$ for $I(\sigma)=\un-\{(i,\sigma(i))\}_{i\in\un}$ where $\sigma\in S_n$ is a permutation.

\begin{prop}\label{prop: the rank space of sl_n}
 The underlying set of $\SL_{n}$ is
 \[
  \{ \ \fp_I \ | \ I\subset I(\sigma)\text{ for some }\sigma\in S_n \ \}.
 \]
 The rank of $\SL_{n}$ is $n-1$ and the set of pseudo-Hopf points of minimal rank is
 \[
  \cZ(\SL_{n}) \quad = \quad \{ \ \fp^\sigma \ | \ \sigma \in S_n \ \},
 \]
 which equals the set of closed points of $\SL_{n}$. The residue field of $\fp^\sigma$ is
 \[
  \kappa(\fp^\sigma) \quad = \quad \bpgenquot{\Fun[T_{i,\sigma(i)}^{\pm1}]}{\prod_{i=1}^n T_{i,\sigma(i)}\=1}
 \]
 if $\sigma$ is an element of the alternating group $A_n$ and 
 \[
  \kappa(\fp^\sigma) \quad = \quad \bpgenquot{\Funsq[T_{i,\sigma(i)}^{\pm1}]}{\prod_{i=1}^n T_{i,\sigma(i)}+1\=0}
 \]
 if $\sigma\in S_n- A_n$. The rank space is 
 \[
  \SL_{n}^\rk \quad = \quad \coprod_{\sigma\in S_n} \ \Spec \kappa(\fp^\sigma)
 \]
 and embeds into $\SL_{n}$.
\end{prop}

\begin{proof}
 The pre-addition $\cR$ of the global sections $\Fun[\SL_n]=\bpquot{\Fun[T_{i,j}]}{\cR}$ of $\SL_{n}$ is generated by the relation 
 \begin{equation}
  \sum_{\sigma\in A_n} \ \prod_{i=1}^n \ T_{i,\sigma(i)} \quad \= \quad \sum_{\sigma\in S_n-A_n} \ \prod_{i=1}^n \ T_{i,\sigma(i)} \quad + \quad 1.  \label{eq: det=1}
 \end{equation}
 Thus $\fp_I$ is a prime ideal if and only if there is at least one $\sigma\in S_n$ such that $\prod_{i=1}^n T_{i,\sigma(i)} \notin\fp_I$. In other words, a prime ideal $\fp_I$ of $\Fun[T_{i,j}]_{i,j\in\un}$ generates a prime ideal of $\Fun[\SL_n]$ if and only if there is a $\sigma\in S_n$ such that $\fp_I\subset\fp^\sigma$. Consequently, the closed points of $\SL_n$ are the prime ideals $\fp^\sigma$ for $\sigma\in S_n$. 

 We determine the pseudo-Hopf points of $\SL_{n}$. For a point $\fp_I$, the coordinate blueprint of the closed subscheme $\overline{\fp_I}$ of $\SL_{n}$ is $\Gamma\,\overline{\fp_I}=\bpgenquot{\Fun[T_{i,j}|i,j\in\un]}{\cR}$ whose pre-addition $\cR$ is generated by the relation \eqref{eq: det=1} together with the relations $T_{i,j}\=0$ for $(i,j)\in I$. If $I=I(\sigma)$ for some $\sigma\in S_n$, then $\Gamma\,\overline{\fp^\sigma}=\kappa(\fp^\sigma)$ and in the relation \eqref{eq: det=1} survive only the ``$1$'' and one other term if $T_{i,j}$ is substituted by $0$ for all $(i,j)\in I$, i.e.\ it looks like
 \[
  \prod_{i=1}^n \ T_{i,\sigma(i)} \ \= \ 1 \quad\text{or}\quad \prod_{i=1}^n \ T_{i,\sigma(i)} \ + \ 1 \ \= \ 0
 \]
 depending on the sign of $\sigma$. In both cases, $T_{i,\sigma(i)}$ is invertible in $\Gamma\,\overline{\fp^\sigma}$ for $i=1,\dotsc,n$. Thus $\kappa(\fp^\sigma)$ is as claimed in the proposition. Further, it is clear that $\overline{\fp^\sigma}$ is affine, that $\overline{\fp^\sigma}^\px=\overline{\fp^\sigma}$, that $\overline{\fp^\sigma}^+_\Z$ is a free $\Z$-module and that $\fp^\sigma$ is of indefinite characteristic. Thus $\fp^\sigma$ is pseudo-Hopf. Note that $\fp^\sigma$ is of rank $r$, independently of $\sigma$.

 If $I$ is properly contained in $I(\sigma)$ for some $\sigma\in S_n$, then there are at least two terms besides to the ``$1$'' in relation \eqref{eq: det=1} that are not trivial when $T_{i,j}$ is substituted by $0$ for all $(i,j)\in I$. Therefore, none of the $T_{i,j}$ is invertible and $\overline{\fp_I}^\px=\Fun$. This shows that $\fp_I$ is not pseudo-Hopf in this case.

 We conclude that the rank of $\SL_{n}$ is $r$ and that $\cZ(\SL_{n})=\{ \fp^\sigma | \sigma \in S_n \}$, which equals the set of closed points of $\SL_{n}$. Therefore, $\upsilon_{\SL_{n}}:\SL_{n}^\prk\to \SL_{n}^\rk$ is an isomorphism, and $\SL_{n}^\rk$ embeds into the finite blue scheme $\SL_{n}$ (cf.\ the comments in Section \ref{subsection: rank space}). This also proves the form of the rank space as claimed in the proposition.
\end{proof}

Let $\fe$ be the Weyl kernel of $\SL_{n}$. Then the canonical torus $T$ equals $\fe_\Z^+$, which is the diagonal torus of $\SL_{n,\Z}^+$. Thus $T$ is a maximal torus of $\SL_{n,\Z}$, which equals its own centralizer. The normalizer of $T$ is the subgroup $N=(\SL_{n}^\rk)^+_\Z$ of monomial matrices. 

\begin{thm}\label{thm: the tits-weyl model sl_n} \ 
 \begin{enumerate}
  \item The group law $\mu_N:N\times N\to N$ descends to a unique group law $\mu^\rk:\SL_{n}^\rk\times\SL_{n}^\rk\to\SL_{n}^\rk$ in $\rkBSch$.
  \item The group law $\mu_\Z^+$ of $\SL_{n,\Z}^+$ descends to a unique monoid law $\mu^+:\SL_{n,\N}^+\times\SL_{n,\N}^+\to\SL_{n,\N}^+$ in $\Sch_\N^+$.
  \item The pair $\mu=(\mu^\rk,\mu^+)$ is a Tits morphism $\mu:\SL_{n}\times \SL_{n}\to \SL_{n}$ that makes $\SL_{n}$ a Tits-Weyl model of $\SL_{n,\Z}^+$.
  \item The group $\SL_n^\cT(\Fun)$ of $\Fun$-rational Tits points is isomorphic to the alternating group $A_n$.
  \item\label{part5} For a semiring $B$, the monoid $\SL_n(B)$ is the monoid of all matrices $n\times n$-matrices $(a_{i,j})$ with coefficients $a_{i,j}\in B$ that satisfy the determinant condition \eqref{eq: det=1}.
 \end{enumerate}
\end{thm}

\begin{proof}
 We prove \eqref{part1}. As a scheme, $N=\coprod_{\sigma\in S_n}\Spec\kappa(\fp^\sigma)^+_\Z$, and $\kappa(\fp^\sigma)^+_\Z=\Z[T_{i,\sigma(i)}]_{i=1,\dotsc,n}/ I$ where the ideal $I$ is generated by $\prod_{i=1}^nT_{i,\sigma(i)}+(-1)^{\sign \sigma}$. The group law $\mu_N:N\Stimes N\to N$ is given by the ring homomorphism
 \[
 \begin{array}{cccc}
      \Gamma\mu_N: & \displaystyle \prod_{\sigma\in S_n}\kappa(\fp^\sigma)^+_\Z & \longrightarrow & \displaystyle \biggl(\; \prod_{\tau\in S_n}\kappa(\fp^\tau)^+_\Z \; \biggr)\ \Sotimes_\Z \ \biggl(\; \prod_{\tau'\in S_n}\kappa(\fp^{\tau'})^+_\Z \; \biggr). \\[20pt]
                                               &    T_{i,\sigma(i)}         & \longmapsto         & \displaystyle \sum_{\tau\tau'=\sigma} \ T_{i,\tau(i)}\otimes T_{\tau(i),\sigma(i)}
 \end{array}
 \]
 This descends to a morphism $\mu^\rk:\SL_{n}^\rk\times \SL_{n}^\rk\to\SL_{n}^\rk$ that is defined by 
 \[
 \begin{array}{cccc}
      \Gamma\mu^\rk: & \displaystyle \prod_{\sigma\in S_n}\kappa(\fp^\sigma) & \longrightarrow & \displaystyle \prod_{\tau\in S_n}\kappa(\fp^\tau) \ \otimes_\Fun \ \prod_{\tau'\in S_n}\kappa(\fp^{\tau'}) = \bpquot{\prod_{\tau,\tau'\in S_n} \kappa(\fp^\tau)\times \kappa (\fp^{\tau'})}{\cR} \\[20pt]
                                               &    T_{i,\sigma(i)}         & \longmapsto         & (a_{\tau,\tau'})_{\tau,\tau'\in S_n}
 \end{array}
 \]
 where $\cR$ is the pre-addition that defines the tensor product and where $a_{\tau,\tau'}=(T_{i,,\tau(i)},T_{\tau(i),\sigma(i)})$ if $\tau\tau'=\sigma$ and $a_{\tau,\tau'}=0$ otherwise. This means that the diagram
 \[
  \xymatrix@C=6pc{ {\prod_{\sigma\in S_n}\kappa(\fp^\sigma)^+_\Z} \ar[r]^(0.4){\mu_N} & \biggl(\; \prod_{\tau\in S_n}\kappa(\fp^\tau)^+_\Z \; \biggr)\ \Sotimes_\Z \ \biggl(\; \prod_{\tau'\in S_n}\kappa(\fp^{\tau'})^+_\Z \; \biggr) \\
                  {\prod_{\sigma\in S_n}\kappa(\fp^\sigma)} \ar[r]^(0.4){\mu^\rk}\ar[u] &  {}\prod_{\tau\in S_n}\kappa(\fp^\tau) \ \otimes_\Fun \ \prod_{\tau'\in S_n}\kappa(\fp^{\tau'}) \ar[u] }
 \]
 commutes. Since $\prod_{\sigma\in S_n}\kappa(\fp^\sigma)$ is cancellative, the vertical arrows are inclusions. Consequently, $\mu^\rk$ is uniquely determined by $\mu_N$. It is easily seen that $\mu^\rk$ is a group law in $\rkBSch$ with identity $\epsilon^\rk:\ast_\Fun\to\SL_{n}^\rk$ given by 
 \[
 \begin{array}{cccc}
        \Gamma\epsilon^\rk: & \displaystyle\prod_{\sigma\in S_n}\kappa(\fp^\sigma)  & \longrightarrow &\Fun \\ [20pt]
              &  (a_\sigma)_{\sigma\in S_n}  & \longmapsto    & a_e
  \end{array}
  \]
 where $e\in S_n$ is the trivial permutation and with inverse $\iota^\rk:\SL_{n}^\rk\to\SL_{n}^\rk$ given by
 \[
 \begin{array}{cccc}
       \Gamma\iota^\rk:   & \displaystyle\prod_{\sigma\in S_n}\kappa(\fp^\sigma) & \longrightarrow & \displaystyle\prod_{\sigma\in S_n}\kappa(\fp^\sigma) \\[20pt]
              & T_{i,\sigma(i)}   & \longmapsto    & T^{-1}_{\sigma(i),i}
  \end{array}
  \]
 where we understand the element $T_{i,\sigma(i)}$ of $\kappa(\fp^\sigma)$ as the element $(a_{\sigma'})$ of $\prod_{\sigma\in S_n}\kappa(\fp^\sigma)$ with $a_{\sigma'}=T_{i,\sigma(i)}$ if $\sigma'=\sigma$ and $a_{\sigma'}=0$ otherwise. This shows \eqref{part1}.

 We continue with \eqref{part2}. The group law $\mu_\Z^+:\SL_{n,\Z}^+\Stimes \SL_{n,\Z}^+\to\SL_{n,\Z}^+$ is defined by the ring homomorphism
 \[
  \begin{array}{cccc}
      \Gamma\mu^+_\Z:   & \Z[T_{i,j}]_{i,j\in\un}\,/\,I  & \longrightarrow & \Bigl(\Z[T'_{i,j}]_{i,j\in\un}\,/\,{I'}\Bigr) \ \Sotimes_\Z \ \Bigl( \Z[T''_{i,j}]_{i,j\in\un}\,/\,{I''} \Bigl)\\[15pt]
             & T_{i,j}   & \longmapsto    & \displaystyle\sum_{k=1}^n T'_{i,k} \otimes T''_{k,j}
  \end{array}
 \]
 where the ideals $I$, $I'$ and $I''$ are generated by the relation that expresses that the determinant equals $1$ (as explained in the beginning of this section). Since $\mu_\Z^+$ can be defined without the use of additive inverses, it descends to a morphism $\mu^+:\SL_n^+\times \SL_n^+\to \SL_n^+$. Uniqueness follows, as in the case of $\mu^\rk$, because $\SL_n^+$ is cancellative. It is easily seen that $\mu^+$ is a semigroup law in $\Sch^+_\N$ with identity $\epsilon^+:\ast_\N\to\SL_n^+$ that is given by the blueprint morphism $\Gamma\epsilon^+:\N[\SL_n]\to\N$ that maps $T_{i,j}$ to $1$ if $i=j$ and to $0$ if $i\neq j$. This shows \eqref{part2}. Note that $\mu^+$ does not have an inverse since the inverse of $\mu_\Z^+$ involves additive inverses of the $T_{i,j}$. Note further that $(\mu^+)_\Z^+=\mu_\Z^+$, which justifies the notation.

 We proceed with \eqref{part3}. It is clear from the definitions of $\mu^\rk$ and $\mu^+$ that the diagram
 \[
  \xymatrix@C=6pc{ N\Stimes N \ar[r]^{\mu^{\rk,+}_\Z=\mu_N}\ar[d]_{\rho_{\SL_n,\Z}^+\times\rho_{\SL_n,\Z}^+} & N \ar[d]^{\rho_{\SL_n,\Z}^+} \\
                   \SL_{n,\Z}^+\Stimes \SL_{n,\Z}^+ \ar[r]^{\mu_\Z^+} & \SL_{n,\Z}^+}
 \]
 commutes. Thus $\mu=(\mu^\rk,\mu^+)$ is a Tits morphism that is a semigroup law for $\SL_n$ in $\TSch$ with identity $\epsilon=(\epsilon^\rk,\epsilon^+)$. This shows that $\SL_n$ is a Tits model of $\SL_{n,\Z}^+$. We already reasoned that the canonical torus $T=\fe_\Z^+$ of $\SL_{n}$ is the diagonal torus of $\SL_{n,\Z}^+$, which is a maximal torus and its own centralizer, and that $N=\SL_{n,\Z}^\rk$ is its normalizer. Thus the morphism $\Psi_\fe: \SL_{n}^\rk/\fe_\Z^+ \to N/T$ is an isomorphism of group schemes, which shows that $\SL_n$ is a Tits-Weyl model of $\SL_{n,\Z}^+$. This proves \eqref{part3}.
 
 We proceed with \eqref{part4}. A morphism $\ast_\Fun\to\SL_n^\rk$ is determined by its image point $\fp^\sigma$ and a morphism $\kappa(\fp^\sigma)\to\Fun$, which is necessarily unique. The latter morphism exists if $\sigma\in A_n$ since in this case $\kappa(\fp^\sigma)$ is a monoid (cf.\ Proposition \ref{prop: the rank space of sl_n}). In case, $\sigma\in S_n-A_n$, the residue field $\kappa(\fp^\sigma)$ contains $-1$ and does not admit a blueprint morphism to $\Fun$. 

 We show \eqref{part5}. Let $B$ be a semiring. A morphism $\Spec B\to\SL_n$ is given by a blueprint morphism $f:\bpquot{\Fun[T_{i,j}]}{\cR}\to B$ where $\cR$ is the pre-addition generated by the relation \eqref{eq: det=1}. Such a morphism is determined by the images $a_{i,j}=f(T_{i,j})\in B$ of the generators $T_{i,j}$, and a family of elements $(a_{i,j})$ occurs as images of a blueprint morphism $f$ if and only if the $A_{i,j}$ satisfy relation \eqref{eq: det=1}. It is clear that the multiplication on $SL_n(B)$ that is induced by the monoid law of $\SL_n$ is the usual matrix multiplication. This concludes the proof of the theorem.
\end{proof}


\subsection{The cube lemma}
\label{subsection: The cube lemma}

In the rest of this part of the paper, we will establish Tits-Weyl models of subgroups $\cG$ of $\SL_{n,\Z}^+$. To show that the semigroup law of $\SL_n$ restricts to a given $\Fun$-model $G$ of $\cG$, we will often need to prove the existence of a morphism $h_1:X_1\to Y_1$ that completes a commuting diagram of the form
\begin{equation}\label{c1}
 \xymatrix@C=4pc@R=1pc{                                                         & X_2'  \ar[rr]^(.4){h_2'}\ar[dd]_(.7){f_{X,2}}|\hole &                                          & Y_2'\ar[dd]^{f_{Y,2}}  \\
                        X_1'\ar[ru]^{g_{X}'}\ar[rr]^(.7){h_1'}\ar[dd]_{f_{X,1}} &                                                     & Y_1'\ar[ru]_{g_Y'}\ar[dd]^(.3){f_{Y,1}}                           \\
                                                                                & X_2 \ar[rr]^(.3){h_2}|\hole                         &                                          & Y_2                    \\
                        X_1                                  \ar[ru]^{g_X}      &                                                     & Y_1\ar[ru]_{g_Y}         }  
\end{equation}
to a commuting cube 
\begin{equation}\label{c2}
 \xymatrix@C=4pc@R=1pc{                                                         & X_2'  \ar[rr]^(.4){h_2'}\ar[dd]_(.7){f_{X,2}}|\hole &                                          & Y_2'\ar[dd]^{f_{Y,2}}  \\
                        X_1'\ar[ru]^{g_{X}'}\ar[rr]^(.7){h_1'}\ar[dd]_{f_{X,1}} &                                                     & Y_1'\ar[ru]_{g_Y'}\ar[dd]^(.3){f_{Y,1}}                           \\
                                                                                & X_2 \ar[rr]^(.3){h_2}|\hole                         &                                          & Y_2                    \\
                        X_1 \ar[rr]^(.6){h_1} \ar[ru]^{g_X}                     &                                                     & Y_1\ar[ru]_{g_Y}         }   
\end{equation}
of morphisms. In this section, we provide the necessary hypotheses that yield the morphism in question for the categories $\Sets$, $\Top$ and $\BSch$.

\begin{lemma}[\textbf{The cube lemma for sets}]\label{lemma: cube lemma for sets}
 Consider a commutative diagram of the form \eqref{c1} in the category $\Sets$. If $f_{X,1}:X_1'\to X_1$ is surjective and $g_Y$ injective, then there exists a unique map $h_1:X_1\to Y_1$ such that the resulting cube \eqref{c2} commutes.
\end{lemma}

\begin{proof}
 Let $x\in X_1$. Then there is a $x'\in X'_1$ such that $f_{X,1}(x')=x$. Define $h_1(x)=f_{Y,1}\circ h_1'(x')$. 

 We verify that the definition of $h_1$ does not depend on the choice of $x'$. Let $x'_1$ and $x_2'$ be two elements of $X_1'$ with $f_{X,1}(x'_1)=f_{X,1}(x'_2)=x$. Then 
 \begin{eqnarray*}
  g_Y\circ f_{Y,1}\circ h_1'(x'_i) & = & f_{Y,2} \circ g_Y' \circ h_1'(x'_i) \\
                                   & = & f_{Y,2} \circ h_2' \circ g_X'(x'_i) \\
                                   & = & h_2 \circ f_{X,2} \circ g_X'(x'_i) \\
                                   & = & h_2 \circ g_X \circ f_{X,1}(x'_i) \\
                                   & = & h_2 \circ g_X (x) 
 \end{eqnarray*}
 is the same element in $Y_2$ for $i=1,2$. Since $g_Y$ is injective, $f_{Y,1}\circ h_1'(x'_1)=f_{Y,1}\circ h_1'(x'_2)$ in $Y_1$, which means that the definition of $h_1(x)$ does not depend on the choice of $x'$ in $f_{X,1}^{-1}(x)$.

 By definition of $h_1$, the diagram 
 \[
  \xymatrix@R=2pc@C=4pc{X_1'\ar[r]^{h_1'}\ar[d]_{f_{X,1}}  & Y_1'\ar[d]^{f_{Y,1}}    \\
                        X_1 \ar[r]^{h_1}                   & Y_1     }
 \]
 commutes. The same calculation as above shows that $g_Y\circ h_1(x)=g_Y\circ f_{Y,1}\circ h_1'(x')$ equals $h_2 \circ g_X (x)$, which means that the diagram
 \[
  \xymatrix@R=1pc@C=2pc{  & X_2 \ar[rr]^{h_2}                         &                                          & Y_2     \\
                        X_1 \ar[rr]^{h_1} \ar[ru]^{g_X} &                 & Y_1 \ar[ru]_{g_Y}    }
 \]
 commutes. This proves the lemma.
 \end{proof}

\begin{lemma}[\textbf{The cube lemma for topological spaces}]\label{lemma: cube lemma for topological spaces}
 Consider a commutative diagram of the form \eqref{c1} in the category $\Top$. If $f_{X,1}:X_1'\to X_1$ is surjective and $g_Y$ an immersion, then there exists a unique continuous map $h_1:X_1\to Y_1$ such that the resulting cube \eqref{c2} commutes.
\end{lemma}

\begin{proof}
 By the cube lemma for sets (Lemma \ref{lemma: cube lemma for sets}), a unique map $h_1:X_1\to Y_1$ exists such that the cube \eqref{c2} commutes. We have to show that $h_1$ is continuous. Let $U$ be an open subset of $Y_1$. Then there is an open subset $U'$ of $Y_2$ such that $g_Y^{-1}(U')=U$ since $g_Y$ is an immersion. Then the subset
 \[
  h_1^{-1}(U) \quad = \quad h_1^{-1} \circ g_Y^{-1}(U') \quad = \quad g_X^{-1} \circ h_2^{-1}(U')
 \]
 of $X_1$ is open as an inverse image of $U'$ under a continuous map. This proves the lemma.
\end{proof}

 A \emph{quasi-submersion of blue schemes} is a morphism $f:X\to Y$ that is surjective and satisfies for every affine open subset $U$ of $Y$ that $V=f^{-1}(U)$ is affine and that $f^\#(U):\Gamma(\cO_Y,U)\to\Gamma(\cO_X,V)$ is an inclusion as a subblueprint, i.e.\ $f^\#(U)$ is injective and the pre-addition of $\Gamma(\cO_Y,U)$ is the restriction of the pre-addition of $\Gamma(\cO_X,V)$ to $\Gamma(\cO_Y,U)$ (cf.\ \cite[section 2.1]{blueprints1}).

\begin{lemma}\label{lemma: subblueprints are quasi-submersions}
 If $f:B\hookrightarrow C$ is an inclusion of a subblueprint $B$ of $C$, then $f^\ast:\Spec C\to \Spec B$ is a quasi-submersion.
\end{lemma}

\begin{proof}
 Every affine open $U$ of $X=\Spec B$ is of the form $U=\Spec S^{-1}B$ for some finitely generated multiplicative subset $S$ of $B$. The inverse image $f^\ast(U)$ is isomorphic to $\Spec S^{-1}C$ and therefore affine. We have to show that the induced blueprint morphism $g:S^{-1}B\to S^{-1}C$ is an inclusion of a subblueprint.

 We show injectivity of $g$. If $g(\frac as)=g(\frac{a'}{s'})$, i.e.\ $\frac{f(a)}{s}=\frac{f(a')}{s'}$, for $a,b\in B$ and $s,s'\in S$, then there is a $t\in S$ such that $ts'f(a)=tsf(a')$ in $C$. Since $B$ is a subblueprint of $C$, we have $ts'a=tsa'$ in $B$, which shows that $\frac as=\frac{a'}{s'}$ in $B$. Thus $g:S^{-1}B\to S^{-1}C$ is injective.

 We show that $S^{-1}B$ is a subblueprint of $S^{-1}C$. Consider an additive relation $\sum\frac{f(a_i)}{s_i}\=\sum\frac{f(b_j)}{r_j}$ in $S^{-1}C$. Then there is a $t\in S$ such that $\sum ts^if(a_i)\=\sum tr^jf(b_j)$ in $C$ where $s^i=(\prod_{k\neq i}s_k)\cdot(\prod_{l}r_l)$ and $r^j=(\prod_{k}s_k)\cdot(\prod_{l\neq j}r_l)$. Since $B$ is a subblueprint of $C$, we have $\sum ts^i a_i\=\sum tr^jb_j$ in $B$, which means that $\sum\frac{a_i}{s_i}\=\sum\frac{b_j}{r_j}$ in $S^{-1}B$. This shows that $g:S^{-1}B\to S^{-1}C$ is an inclusion of a subblueprint and finishes the proof of the lemma.
\end{proof}

\begin{lemma}[\textbf{The cube lemma for blue schemes}]\label{lemma: cube lemma for blue schemes}
 Consider a commutative diagram of the form \eqref{c1} in the category $\BSch$. Suppose that $f_{X,1}:X_1'\to X_1$ is a quasi-submersion and $g_Y$ is a closed immersion. Then there exists a unique morphism $h_1:X_1\to Y_1$ of blue schemes such that the resulting cube \eqref{c2} commutes.
\end{lemma}

\begin{proof}
 By the cube lemma for topological spaces (Lemma \ref{lemma: cube lemma for topological spaces}), a unique continuous map $h_1:X_1\to Y_1$ exists such that the cube \eqref{c2} commutes. We have to show that there exists a morphism $\Gamma h_1:\cO_{Y_1}\to\cO_{X_1}$ between the structure sheaves of $X_1$ and $Y_1$. Since a morphism of sheaves can be defined locally, we may assume that all blue schemes in question are affine. 

 If we denote by $\Gamma {Z}$ the coordinate ring of the blue scheme $Z$, then there is a morphism $\Gamma h_1:\Gamma {Y_1}\to\Gamma {X_1}$ of blueprints that completes the commutative diagram
 \[
   \xymatrix@C=4pc@R=1pc{                                     & \Gamma {X_2'}  \ar@{<-}[rr]^(.4){\Gamma h_2'}\ar@{<-}[dd]_(.7){\Gamma f_{X,2}}|\hole &                         & \Gamma {Y_2'}\ar@{<-}[dd]^{\Gamma f_{Y,2}}  \\
                        \Gamma {X_1'}\ar@{<-}[ru]^{\Gamma g_{X}'}\ar@{<-}[rr]^(.7){\Gamma h_1'}\ar@{<-}[dd]_{\Gamma f_{X,1}} &                & \Gamma {Y_1'}\ar@{<-}[ru]_{\Gamma g_Y'}\ar@{<-}[dd]^(.3){\Gamma f_{Y,1}}     \\
                                                              & \Gamma {X_2} \ar@{<-}[rr]^(.3){\Gamma h_2}|\hole                              &                                          &\Gamma { Y_2}                    \\
                        \Gamma {X_1}\ar@{<-}[ru]^{\Gamma g_X}        &                                                                        & \Gamma {Y_1}\ar@{<-}[ru]_{\Gamma g_Y}         }  
 \]
 to a commuting cube. Since $g_Y$ is a closed immersion, $\Gamma g_Y$ is a surjective morphism of blueprints. Since $f_{X,1}$ a quasi-submersion, we may assume further that $f_{X,1}:X_1'\to X_1$ is still surjective and $\Gamma f_{X,1}$ is an injective morphism of blueprints. Then by the cube lemma for sets (Lemma \ref{lemma: cube lemma for sets}), there is a unique map $\Gamma h_1:\Gamma Y_1\to\Gamma X_1$.

 To verify that $\Gamma h_1$ is a morphism of monoids, let $a,b\in\Gamma {Y_1}$. Then there are elements $a',b'\in\Gamma {Y_2}$ such that $\Gamma g_Y(a')=a$ and $\Gamma g_Y(b')=b$. Therefore,
 \[ 
  \Gamma h_1(ab) \ = \ \Gamma g_Y \circ \Gamma h_1(a'b') \ = \ \Gamma h_2 \circ \Gamma g_X(a'b') \ = \ \Bigr(\Gamma h_2 \circ \Gamma g_X(a')\Bigl) \ \cdot \ \Bigr(\Gamma h_2 \circ \Gamma g_X(b')\Bigl),
 \] 
 which is, tracing back the above calculation for both factors, equal to $\Gamma h_1(a) \cdot \Gamma h_1(b)$. Thus $\Gamma h_1$ is multiplicative. Similarly, the calculation $\Gamma h_1(1)=\Gamma g_Y \circ \Gamma h_1(1)=\Gamma h_2 \circ \Gamma g_X(1)=1$ shows that $\Gamma h_1$ is unital.
 
 We are left with showing that $\Gamma h_1$ maps the pre-addition of $\Gamma {Y_1}$ to the pre-addition of $\Gamma {X_1}$. Consider an arbitrary additive relation $\sum a_i\=\sum b_j$ in $\Gamma {Y_1}$. Applying the morphism $\Gamma h_1'\circ \Gamma f_{Y,1}$, we see that 
 \begin{equation*}\label{eqn in proof of cube lemma for blue schemes}
  \sum \ \Gamma h_1'\circ \Gamma f_{Y,1}(a_i) \quad \= \quad \sum \ \Gamma h_1'\circ \Gamma f_{Y,1}(b_j)
 \end{equation*}
 in $\Gamma {X_1'}$. 
 Since $\Gamma f_{X,1}$ is the inclusion of $\Gamma {X_1}$ as a subblueprint of $\Gamma {X_1'}$ and $\Gamma h_1'\circ \Gamma f_{Y,1}=\Gamma f_{X,1}\circ \Gamma h_1$, this relation restricts to $\Gamma {X_1}$, i.e.\ $\sum \Gamma h_1(a_i)\= \sum \Gamma h_1(b_j)$ in $\Gamma {X_1}$. This finishes the proof of the lemma.
\end{proof}


\subsection{Closed subgroups of Tits-Weyl models}
\label{subsection: closed subgroups of tits-weyl models}

In this section, we formulate a criterion for subgroups $\cH$ of a group scheme $\cG$ that yields a Tits-Weyl model of $\cH$. Applied to $\cG=\SL_{n,\Z}^+$, this will establish a variety of Tits-Weyl models of prominent algebraic groups that we discuss in more detail at the end of this section: general linear groups, special orthogonal groups, symplectic groups and some of their isogenies.

If $G=\Spec\bigl(\bpquot A\cR\bigr)$ is an $\Fun$-model of $\cG$ and $\iota_\cH:\cH\hookrightarrow \cG$ a closed immersion, then we can consider the pre-addition $\cR'$ on $A$ that is generated by $\cR$ and all defining relations of $\cH$ in elements of $A$. This defines an $\Fun$-model $H=\Spec\bigl(\bpquot{A}{\cR'}\bigr)$ of $\cH$ together with a closed immersion $\iota:H\hookrightarrow G$ that base extends to $\iota_\cH=\iota_\Z^+$.


Let $\cG$ be an affine smooth group scheme of finite type with an $\Fun$-model $G$. We say that a torus $T\subset\cG$ is \emph{diagonal w.r.t.\ $G$} if for every $x\in G$, the group law $\mu_\cG$ of $\cG$ restricts to morphisms
\[
 T \ \Stimes_\Z \ \barx^+_\Z \quad \longrightarrow \quad \barx^+_\Z \qquad\text{and}\qquad \barx^+_\Z \ \Stimes_\Z \ T \quad \longrightarrow \quad \barx^+_\Z.
\]
If $G$ is the $\Fun$-model of $\cG$ that is associated to an embedding $\iota:\cG\hookrightarrow\SL_{n,\Z}^+$, then a torus $T\subset\cG$ is diagonal w.r.t.\ $G$ if and only if the image $\iota(T)$ is contained in the diagonal torus of $\SL_{n,\Z}^+$. In particular, the canonical torus of $\SL_n$ is diagonal.

If $G$ is a Tits-Weyl model of $\cG$, then we say briefly that the canonical torus $T$ is diagonal if $T$ is diagonal w.r.t.\ $G$.

\begin{thm} \label{thm: tits-weyl models of subgroups}
 Let $\cG$ be an affine smooth group scheme of finite type with a Tits Weyl model $G$. Assume that the canonical torus $T$ is diagonal and that $N=G^{\rk,+}_\Z$ is the normalizer of $T$ in $\cG$. Then $C=\fe^+_\Z$ is the centralizer of $T$ where $\fe$ is the Weyl kernel of $G$. 

 Let $\cH$ be a smooth closed subgroup of $\cG$ and $\iota:H\hookrightarrow G$ the associated $\Fun$-model. Assume that $\tilde T=T\cap\cH$ is a maximal torus of $H$. Assume further that the centralizer $\tilde C$ of $\tilde T$ in $\cH$ is contained in $C\cap\cH$ and that the normalizer $\tilde N$ of $\tilde T$ in $\cH$ is contained in $N\cap\cH$. Then the following holds true.
 \begin{enumerate}
  \item The set of pseudo-Hopf points of $H$ is $\cZ(H)=\iota^{-1}(\cZ(G))$. Thus the closed immersion $\iota: H\hookrightarrow G$ is Tits. 
  \item We have $\tilde C=C\cap \cH=\tilde\fe^+_\Z$ and $\tilde N=N\cap\cH=H^{\rk,+}_\Z$.
  \item The monoid law $\mu_G$ of $G$ restricts uniquely to a monoid law $\mu_H$ of $H$. The pair $(H,\mu_H)$ is a Tits-Weyl model of $\cH$ whose canonical torus $\tilde T$ is diagonal. The Tits morphism $(\iota^\rk,\iota^+):H\hookrightarrow G$ is a monoid homomorphism in $\TSch$. 
  \item If $G^\rk$ lifts to $G$, then $H^\rk$ lifts to $H$.
 \end{enumerate}
\end{thm}

\begin{proof}
 First note that since $\Psi_\fe: G_\Z^{\rk,+}/\fe^+_\Z\stackrel\sim\to N/C=\cW$ is an isomorphism and $G_\Z^{\rk,+}=N$, the centralizer $C$ equals indeed $\fe^+_\Z$. The proof of \eqref{part1}--\eqref{part4} is somewhat interwoven and doesn't follow the order of the statements. 

 Let $x\in H$ and $y=\iota(x)$ be the image in $G$. Then $\barx^+_\Z=\bary^+_\Z\cap\cH$. Therefore, both left and right action of $T$ on $\bary^+_\Z$ restricts to an action of $\tilde T$ on $\barx^+_\Z$. This shows that $\tilde T$ is diagonal w.r.t.\ $H$.

 We show that there is an element $\tilde\fe$ in $H^\rk$ such that $\iota^\rk(\tilde\fe)=\fe$. Let $e\in\cZ(G)$ be the pseudo-Hopf point such that $\fe=\prerk e^\px$. Then $T=\overline e^+_\Z$. Since $\tilde T\subset T$ maps to $H$, there is a point $\tilde e$ in $H$ such that $\iota(\tilde e)=e$. As a closed subgroup of $T=\overline e^+_\Z$, the group scheme $\overline{\tilde e}^+_\Z$ is diagonalizable and its global sections are isomorphic to a group ring. Thus $\overline{\tilde e}^+_\Z$ is flat, $\overline{\tilde e}_\inv$ is generated by its units, $\overline{\tilde e}$ is affine and $\tilde e$ is almost of indefinite characteristic since $e$ is so and the morphism $\tilde T\to\overline e$ factorizes through $\overline{\tilde e}$. Thus $\tilde e$ is pseudo-Hopf of rank $r=\rk \tilde T$. Since $\tilde T$ acts on $\barx^+_\Z$ for all pseudo-Hopf points of $H$, the rank of a pseudo-Hopf point is at least $r$. Thus $\tilde e\in\cZ(H)$, which defines a point $\tilde\fe\in H^\rk$.

 We show that $\tilde C=C\cap \cH=\tilde\fe^+_\Z$. Since $\fe^+_\Z=C$, we have $\tilde\fe^+_\Z=\fe^+_\Z\cap\cH=C\cap \cH$. Since $\tilde\fe$ is an abelian group in $\BSch$, $\tilde\fe^+_\Z$ centralizes $\tilde T$. Thus $\tilde C=\tilde \fe^+_\Z$.

 We show that $\tilde N=N\cap\cH$. By the hypothesis of the theorem, we already know that $\tilde N\subset N\cap\cH$. By the second isomorphism theorem for groups, $\tilde N/\tilde C$ is a subgroup of the constant group scheme $N/C$. Thus $\tilde N$ is isomorphic to a finite disjoint union of copies of $\tilde C$ as a scheme. Therefore, $\tilde N$ is flat and we can investigate $\tilde N$ by considering complex points. Let $n\in(N\cap\cH)(\C)$. Then $n\tilde T(\C)n^{-1}$ is contained in both $\cH(\C)$ and $T(\C)$, whose intersection is $\tilde T(\C)$. Thus $n$ is contained in the normalizer of $\tilde T(\C)$ in $\cH(\C)$, which is $\tilde N(\C)$.

 We show that $\tilde N=H^{\rk,+}_\Z$ and $\cZ(H)=\iota^{-1}(\cZ(G))$. Since $\tilde N=N\cap\cH$, the image of $N\to H$ is $\iota^{-1}(\cZ(G))$. Since $\tilde N$ is the disjoint union of schemes isomorphic to $\tilde C$, every point of $\iota^{-1}(\cZ(G))$ is pseudo-Hopf of rank $r$ for the same reason as the one that showed that $\overline e$ is pseudo-Hopf of rank $r$. Thus $\iota^{-1}(\cZ(G))\subset\cZ(H)$ and $\tilde N\subset H^{\rk,+}_\Z$.

 To show the reverse conclusion, consider an arbitrary pseudo-Hopf point $x$ of $H$. Since we have a left-right double action
 \[
  \tilde T \ \ \Stimes_\Z \ \ \barx^+_\Z \ \ \Stimes_\Z \ \ \tilde T \quad\longrightarrow \quad \barx^+_\Z
 \]
 of $\tilde T$ on $\barx^+_\Z$, the rank of $x$ is at least $r$. If the rank $\rk x=\dim\barx^+_\Q$ is $r$, then $\tilde T(\C)p\tilde T(\C)=p\tilde T(\C)$ for any $p\in\barx(\C)$. This means that for all $t_1,t_2\in\tilde T(\C)$, there is a $t_3\in\tilde T(\C)$ such that $t_1pt_2=pt_3$. Multiplying the latter equation with $t_2^{-1}$ from the right yields $t_1p=pt_3t_2^{-1}$, which shows that $\tilde T(\C)p=p\tilde T(\C)$. This means that $p\in\tilde N(\C)=\cH(\C)\cap N(\C)$, i.e.\ $x\in\cZ(H)$.

 This finishes the proof of \eqref{part2}. Further we have proven that $\iota$ maps $\cZ(H)$ to $\cZ(G)$, which implies that $\iota:H\to G$ is Tits. This finishes the proof of \eqref{part1}.

 We turn to the proof of \eqref{part3}. We show that $\mu_G^+:G^+\Stimes G^+\to G^+$ descends to a morphism $\mu_H:H^+\Stimes H^+\to H^+$. Since $H\hookrightarrow G$ is a closed immersion, we have a surjection $\Gamma G\twoheadrightarrow \Gamma H$ and a surjection $\Gamma G^+\twoheadrightarrow \Gamma H^+$. By Lemma \ref{lemma: surjections are closed immersions}, the morphism $H^+ \to G^+$ is a closed immersion. Since $H$ is cancellative, $\Gamma H^+$ is also cancellative and $\Gamma H^+\to \Gamma H^+_\Z$ is an inclusion as a subblueprint. By Lemma \ref{lemma: subblueprints are quasi-submersions}, the morphism $H^+_\Z\to H^+$ is a quasi-submersion. The same is true for $H^+_\Z\Stimes_\Z H^+_\Z\to H^+\Stimes H^+$. Therefore we can apply the cube lemma for blue schemes (Lemma \ref{lemma: cube lemma for blue schemes}) to the commutative diagram
 \[
  \xymatrix@C=4pc@R=1pc{                               & \cG\Stimes_\Z \cG  \ar[rr]^(.5){\mu_\cG}\ar@{->>}[dd]_(.7){}|\hole &       & \cG\ar@{->>}[dd]^{}  \\
                        \cH\Stimes_\Z \cH \arincl{[ru]^{}}\ar[rr]^(.7){\mu_\cH}\ar@{->>}[dd]_{} &                                  & \cH\arincl{[ru]_{}}\ar@{->>}[dd]^(.3){}     \\
                                                        &  G^+\Stimes G^+\ar[rr]^(.4){\mu_G^+}|(.565)\hole                   &                         &  G^+\;,               \\
                        H^+\Stimes H^+   \arincl[ru]^{}      &                                                     & H^+ \arincl[ru]_{}         } 
 \]
 which shows that $\mu_G^+$ descends to a unique morphism $\mu_H^+:H^+\Stimes H^+\to H^+$. It easy to verify that $\mu_H^+$ is associative. That $\epsilon_G^+:\ast_\N\to G^+$ descends to an identity $\epsilon_H^+:\ast_\N\to H^+$ of $H^+$ is shown similarly. 

 To show that $\mu_G^\rk:G^\rk\times G^\rk\to G^\rk$ of the rank space descends to a morphism $\mu_H^\rk:H^\rk\times H^\rk\to H^\rk$ is more subtle since, in general, $H^\rk\to G^\rk$ is not a closed immersion. 

 By Lemma \ref{lemma: fe is diagonalizable}, the Weyl kernel $\fe$ of $G$ is diagonalizable. Thus $\Gamma\fe^+_\Z$ is a group ring $\Z[\Lambda]$ for an abelian group $\Lambda$ and the group law of $\fe^+_\Z$ comes from the multiplication of $\Lambda$. Therefore the unit field of $\fe_\Z^+$ is $\Fpx(\fe^+_\Z)=\Funsq[\Lambda]=\Fun[\Lambda]_\inv$. Let $e\in\cZ(G)$ be the pseudo-Hopf point of $G$ with $\fe=\prerk e^\px$. Then $\prerk e$ is generated by its units, which means that either $\fe\simeq\Spec\Fun[\Lambda]$ or $\fe\simeq\Funsq[\Lambda]$. The analogous statement is true for $\tilde\fe\in H^\rk$. Since $\tilde\fe^+_\Z\to\fe^+_\Z$ is a closed immersion, $\Gamma\fe^+_\Z\to\Gamma\tilde\fe^+_\Z$ is surjective, and so is $\Gamma\fe_\inv\to\Gamma\tilde\fe_\inv$. By Lemma \ref{lemma: surjections are closed immersions}, the morphism $\tilde\fe_\inv\to\fe_\inv$ is a closed immersion. Since for every point $x\in H^\rk$ and $y=\iota^\rk(x)\in G^\rk$, the scheme $\barx^+_\Z$ is isomorphic to $\tilde\fe^+_\Z$ and $\bary$ is isomorphic to $\fe^+_\Z$, the same argument as above shows that $H^\rk_\inv\to G^\rk_\inv$ is a closed immersion.

 Since $H^\rk_\inv$ is cancellative, $\Gamma H_\inv^\rk\hookrightarrow\Gamma\tilde N$ is a subblueprint, and by Lemma \ref{lemma: subblueprints are quasi-submersions}, $\tilde N\to H^\rk_\inv$ is a quasi-submersion. The same holds for $\tilde N\Stimes_\Z\tilde N\to H^\rk_\inv\times H^\rk_\inv$. Therefore we can apply the cube lemma for blue schemes (Lemma \ref{lemma: cube lemma for blue schemes}) to the commutative diagram
 \[
  \xymatrix@C=4pc@R=1pc{                               & N\Stimes_\Z N \ar[rr]^(.5){\mu_N}\ar@{->>}[dd]_(.7){}|\hole &                                                                & N \ar@{->>}[dd]^{}  \\
                        {\tilde N\Stimes_\Z \tilde N} \arincl[ru]^{}\ar[rr]^(.7){\mu_{\tilde N}}\ar@{->>}[dd]_{} &                      & {\tilde N}\arincl[ru]_{}\ar@{->>}[dd]^(.3){}     \\
                                                                                      &  {G^\rk_\inv\times G^\rk_\inv}\ar[rr]^(.4){\mu_{G,\inv}^\rk}|(.565)\hole                   &                         &  {G^\rk_\inv}, \\
                        {H^\rk_\inv\times H^\rk_\inv}   \arincl[ru]^{}      &                                                     & {H^\rk_\inv} \arincl[ru]      }    
 \]
 which shows that $\mu_{G,\inv}^\rk$ descends to a unique morphism $\mu_{H^\rk_\inv}:H^\rk_\inv\times H^\rk_\inv\to H^\rk_\inv$. 

 We show that $\mu_{H^\rk_\inv}$ descends to a morphism $\mu_H^\rk:H^\rk\times H^\rk\to H^\rk$. Let $B=\Gamma H^\prk$ and $C=\Gamma\bigl(H^\prk\times H^\prk\bigr)$. Then $B^+$ and $C^+$ are cancellative and embed as subblueprints into the rings $B^+_\Z\simeq\Gamma\tilde N$ and $C^+_\Z\simeq\Gamma(\tilde N\Stimes_\Z\tilde N)$, rspectively. By Lemma \ref{lemma: subblueprints are quasi-submersions}, the canonical morphism $\tilde N\Stimes_\Z\tilde N\to H^{\prk,+}\Stimes H^{\prk,+}$ is a quasi-submersion. Since $\hat D^+\simeq D^+_\canc$ for an arbitrary blueprint $D$, the semiring $B^+$ is canonical isomorphic to the coordinate ring of $\coprod_{x\in\cZ(H)}\barx^+_\canc$. Since $G$ is a Tits-Weyl model of $\cG$, the morphism $\coprod_{y\in\cZ(G)}\bary \to G$ is necessarily a closed immersion. Therefore $\coprod_{x\in\cZ(H)}\barx\to H$ is also a closed immersion. Since $\Gamma H^+\to\Gamma\barx^+\to \Gamma\barx_\canc^+=\prerk x^+$ is surjective, the induced morphism $H^{\prk,+}\to H^+$ is a closed immersion. This shows that the commuting diagram
 \[
  \xymatrix@C=4pc@R=1pc{                               & \cH\Stimes_\Z \cH  \ar[rr]^(.5){\mu_\cH}\ar@{->>}[dd]_(.7){}|\hole &       & \cH\ar@{->>}[dd]^{}  \\
                        \tilde N\Stimes_\Z \tilde N \arincl{[ru]^{}}\ar[rr]^(.7){\mu_{\tilde N}}\ar@{->>}[dd]_{} &                                  & \tilde N\arincl{[ru]_{}}\ar@{->>}[dd]^(.3){}     \\
                                                        &  H^+\Stimes_\N H^+\ar[rr]^(.4){\mu_H^+}|(.565)\hole                   &                         &  H^+               \\
                        H^{\prk,+}\Stimes H^{\prk,+}\arincl[ru]^{}      &                                                     & H^{\prk,+}\arincl[ru]_{}         } 
 \]
 satisfies the hypotheses of the cube lemma for schemes, which yields that $\mu_H^+$ descends to a morphism $\mu^{\prk,+}:H^{\prk,+}\Stimes H^{\prk,+}\to H^{\prk,+}$. With $B$ and $C$ as above, we have that 
 \[
  B^\px \ = \ \Gamma H^\rk, \qquad  B^\px_\inv \ = \ \Gamma H^\rk_\inv, \qquad  B^+ \ = \ \Gamma H^{\prk,+} \quad \text{and} \quad B^+_\Z \ = \ \Gamma \tilde N,
 \]
 and analogous identities for $C^\px$, $C^\px_\inv$, $C^+$ and $C^+_\Z$. The morphisms $\mu_{H^\rk_\inv}$, $\mu^{\prk,+}$ and $\mu_{\tilde N}$ yield a commutative diagram 
 \[
  \xymatrix@C=4pc@R=1pc{                                 &  B^\px_\inv\arincl[dd]\ar[rr]^{\Gamma\mu_{H^\rk_\inv}}       &               &  C^\px_\inv \arincl[dd]\\
                            B^\px \arincl[ur]\arincl[dd] &               &  C^\px\arincl[ur]\arincl[dd]         \\
                                                         & B^+_\Z\ar[rr]|\hole^(0.3){\Gamma\mu_{\tilde N}}            &               &  C^+_\Z    \\
                           B^+ \arincl[ur]\ar[rr]^{\Gamma\mu^{\prk,+}}         &                   &   C^+\arincl[ur]}
 \]     
 of cancellative blueprints. Since $C^\px$ is the intersection of $C^\px_\inv$ with $C^+$ inside $C^+_\Z$, this yields the desired morphism $\Gamma\mu^\rk:B^\px\to C^\px$ or, geometrically, $\mu^\rk:H^\rk\times H^\rk\to H^\rk$.

 The associativity of $\mu_H^\rk$ can be easily derived from the associativity of $\mu^\rk_G$ by using the commutativity of certain diagrams. Similarly to the existence of $\mu_H^\rk$, one shows that there are an identity $\epsilon_H^\rk:\ast_\Fun\to H^\rk$ and an inversion $\iota_H^\rk: H^\rk\to H^\rk$ which turn $H^\rk$ into a group object in $\rkBSch$.

 Moreover, it is easy to see that the pairs $\mu_H=(\mu_H^\rk,\mu_H^+)$ and $\epsilon_H=(\epsilon_H^\rk,\epsilon_H^+)$ are Tits morphisms that give $H$ the structure of a Tits monoid. The Weyl kernel of $H$ is $\tilde\fe$ and the canonical torus is $\tilde T$, which is a maximal torus of $\cH$ by hypothesis. Since $\tilde\fe^+_\Z=\tilde C$ and $H^{\rk,+}_\Z=\tilde N$, the morphism $\Psi_{\tilde\fe}: H_\Z^{\rk,+}/\tilde\fe^+_\Z\stackrel\sim\to \tilde N/\tilde C$ is an isomorphism of group schemes. This shows that $H$ is a Tits-Weyl model of $\cH$. It is clear by the definition of $\mu_H$ that $\iota:H\to G$ is a monoid homomorphism in $\TSch$. This shows \eqref{part3}.

 We show \eqref{part4}. Assume that $G^\rk$ lifts to $G$. The existence of the inverse of $\upsilon_{H_\inv}:H^\prk_\inv\to H^\rk_\inv$ follows an application of the cube lemma to 
 \[
  \xymatrix@C=4pc@R=1pc{             & N \ar[rr]^{\id}\ar@{->>}[dd]_(.7){}|\hole                        &    & N \ar@{->>}[dd]   \\
                        {\tilde N} \arincl[ru]^{\iota^{\rk,+}_\Z}\ar[rr]^(.7){\id}\ar@{->>}[dd]_{}   &   & {\tilde N}\arincl[ru]_{\iota^{\rk,+}_\Z}\ar@{->>}[dd]^(.3){}  \\
                                                        &  G^\rk_\inv  \ar[rr]^(.35){\upsilon_{G,\inv}^{-1}}|(.49)\hole                   &                         &  G^\prk_\inv\,.               \\
                        H^\rk_\inv   \arincl[ru]^{\iota^\rk}      &                                                     & H^\prk_\inv \arincl[ru]_{\iota^\prk}         } 
 \]
 Since for $x\in\cZ(X)$, the subscheme $\prerk x$ of $H^\prk$ is with $-1$ if and only if the subscheme $\prerk x^\px$ of $H^\rk$ is with $-1$, it is clear that $\upsilon_{H_\inv}^{-1}$ comes from an isomorphism $\upsilon_{H}^{-1}:H^\rk\to H^\prk$. This completes the proof of the theorem.
\end{proof}

\subsubsection{The general linear group}
 \label{subsubsection: general linear group}
 
 As a first application of Theorem \ref{thm: tits-weyl models of subgroups}, we establish Tits-Weyl models of general linear groups. The Tits-Weyl models $\GL_n$ of $\GL_{n,\Z}^+$ is of importance for all other Chevalley group schemes since one can consider them as closed subgroups of a general linear group. The standard way to embed $\GL_{n,\Z}^+$ as a closed subgroup in $\SL_{n+1,\Z}^+$ is by sending an invertible $n\times n$-matrix $A$ to the $(n+1)\times(n+1)$-matrix whose upper left $n\times n$ block equals $A$, whose coefficient at the very lower right equals $(\det A)^{-1}$ and whose other entries are $0$.

 In other words, if $\Z[\SL_n]^+=\Z[T_{i,j}|i,j\in\underline{(n+1)}]^+/\cI$ is the coordinate ring of $\SL_{n+1,\Z}^+$ where $\cI$ is the ideal generated by $\sum_{\sigma\in S_{n+1}} \Bigl( \sign(\sigma) \cdot \prod_{i=1}^{n+1} T_{i,\sigma(i)} \Bigr) - 1$ (cf.\ Section \ref{subsection: special linear group}), then the closed subscheme $\GL_{n,\Z}^+$ of $\SL_{n+1,\Z}^+$ is defined by the ideal generated by $T_{i,n+1}$ and $T_{n+1,i}$ for $i=1,\dotsc,n$ and by $T_{n+1,n+1}\cdot\sum_{\sigma\in S_n} \Bigl( \sign(\sigma) \cdot \prod_{i=1}^n T_{i,\sigma(i)} \Bigr) - 1$.

 It is clear that this embedding satisfies the hypothesis of Theorem \ref{thm: tits-weyl models of subgroups}. Thus we obtain the Tits-Weyl model $\GL_n$ of $\GL_{n,\Z}^+$. We describe the points of the blue scheme $\GL_n$. Recall from Proposition \ref{prop: the rank space of sl_n} that the points of $\SL_{n+1}$ are of the form $\fp_I=(T_{i,j}|(i,j)\in I)$ for some $I\in \underline{(n+1)}\times\underline{(n+1)}$ such that there is a permutation $\sigma\in S_{n+1}$ with $I\subset I(\sigma)$. By the definition of $\GL_n$, it is clear that a point $\fp_I$ of $\SL_n$ is a point of $\GL_n$, if and only if $\fp_I$ contains $T_{i,n+1}$ and $T_{n+1,i}$ for $i=1,\dotsc,n$, but does not contain $T_{n+1,n+1}$. This means that a point $\fp_I$ of $\SL_{n+1}$ is in $\GL_n$ if and only if $I\subset I(\sigma)$ for a permutation $\sigma\in S_{n+1}$ that fixes $n+1$. 
 
 Since a permutation $\sigma\in S_{n+1}$ fixes $n+1$ if and only if it lies in the image of the standard embedding $\iota:S_n\hookrightarrow S_{n+1}$, a point $\fp_I$ is contained in $\GL_n$ if and only if $I\subset I(\iota(\sigma))$ for some $\sigma\in S_n$. This shows, in particular, that every prime ideal of $\GL_n$ is generated by a subset of $\{T_{i,j}\}_{i,j\in\un}$ and that the rank space of $\GL_n$ equals $\GL_n^\rk=\{\fp^{\iota(\sigma)}|\sigma\in S_n\}$.

 The residue field of $\fp^\sigma$ depends, as in the case of $\SL_n$, on the sign of $\sigma$: if $\sign\sigma$ is even, then $\kappa(\fp^\sigma)\simeq\Fun[T_{i,\sigma(i)}^{\pm 1}]$; if $\sign\sigma$ is odd, then $\kappa(\fp^\sigma)\simeq\Funsq[T_{i,\sigma(i)}^{\pm 1}]$. Thus $\GL_n^\sT(\Fun)$ is equal to the alternating group inside $\cW(\GL_n)=S_n$. The rank of $\GL_n$ is $n$ and the extended Weyl group equals $G^\sT(\Funsq)\simeq(\Z/2\Z)^n\rtimes S_n$.

 We specialize Theorem \ref{thm: tits-weyl models of subgroups} for subgroups of $\cG=\GL_{n,\Z}^+$. Let $\diag(\GL_{n,\Z}^+)$ be the diagonal torus of $\GL_{n,\Z}^+$ and $\mon(\GL_{n,\Z}^+)$ the group of monomial matrices, which is the normalizer of the diagonal torus.

\begin{cor}\label{cor: models of subgroups of gl_n}
 Let $\cG$ be a smooth closed subgroup of $\GL_{n,\Z}^+$ and $G$ the corresponding $\Fun$-model. Assume that $T=\diag(\GL_{n,\Z}^+)\cap\cG$ is a maximal torus of $\cG$ whose normalizer $N$ is contained in $\mon(\GL_{n,\Z}^+)$. Then the following holds true.
 \begin{enumerate}
  \item The monoid law of $\GL_n$ restricts to $G$ and makes $G$ a Tits-Weyl model of $\cG$. 
  \item The closed embedding $G\hookrightarrow\GL_n$ is Tits and a homomorphism of semigroups in $\TSch$.
  \item The rank space $G^\rk$ lifts to $G$ and equals the intersection $\GL_n^\rk\cap G$. 
  \item The canonical torus $T$ of $G$ is diagonal and equals $\fe^+_\Z$ where $\fe$ is the Weyl kernel of $G$. Its normalizer in $G$ is $G^{\rk,+}_\Z$.\qed
 \end{enumerate}
\end{cor}

\subsubsection{Other groups of type $A_n$} 
 \label{subsubsection: other groups of type a_n}

It is interesting to reconsider $\SL_{n,\Z}^+$ as a closed subgroup of $\GL_{n,\Z}^+$. The Tits-Weyl model associated to this embedding is indeed isomorphic to the Tits-Weyl model $\SL_n$ that we described in Theorem \ref{thm: the tits-weyl model sl_n}. The embedding $\SL_n\to\GL_n$ that we obtain from Corollary \ref{cor: models of subgroups of gl_n} is a homeomorphism between the underlying topological spaces.

Corollary \ref{cor: models of subgroups of gl_n} yields a Tits-Weyl model of the adjoint group scheme $\cG$ of type $A_n$ as follows. Let $\Mat_{n,\Z}^+$ be the scheme of the $n\times n$-matrices, which is isomorphic to an affine space $\SA_\Z^{n^2}$. The action of $\cG$ on $\Mat_{n,\Z}^+$ by conjugation has trivial stabilizer. This defines an embedding $\cG\hookrightarrow\GL_{n^2,\Z}^+$ as a closed subgroup. It is easily seen that this embedding satisfies the hypotheses of Corollary \ref{cor: models of subgroups of gl_n} by using the action  of $\SL_{n,\Z}^+$ on $\Mat_{n,\Z}^+$ by conjugation, which factors through the action of $\cG$ on $\Mat_{n,\Z}^+$ via the canonical isogeny $\SL_{n,\Z}^+\to\cG$. This yields a Tits-Weyl model $G$ of $\cG$. 

Note that the construction of Tits-Weyl models of adjoint groups in the Section \ref{subsection: adjoint chevalley groups} also yields a Tits-Weyl model of $\cG$. We compare these two models in Appendix \ref{app: tits-weyl models of type a_1} in the case $n=1$.


\subsubsection{Symplectic groups} 
 \label{subsubsection: symplectic groups}

The symplectic groups $\Sp_{2n,\Z}^+$ have a standard representation in the following form. Let $J$ be the $2n\times 2n$-matrix whose non-zero entries are concentrated on the anti-diagonal with $J_{i,2n-i}=1$ if $1\leq i\leq n$ and $J_{i,2n-i}=-1$ if $n+1\leq i\leq 2n$. Then the set $\Sp_{2n}(k)$ of $k$-rational points can be described as follows for every ring $k$: the elements of $\Sp_{2n}(k)$ correspond to the $2n\times 2n$-matrices $A=(a_{i,j})$ with entries in $k$ that satisfy $AJA^t=J$, i.e.\ the equations
\[
 \sum_{l=1}^n a_{i,l}a_{j,2n+1-l} \qquad = \qquad \sum_{l=n+1}^{2n} a_{i,l}a_{j,2n+1-l} \quad + \quad \delta_{i,2n+1-j}
\]
for all $1\leq i<j \leq 2n$ where $\delta_{i,2n+1-j}$ is the Kronecker symbol. These equations describe $\Sp_{2n,\Z}^+$ as a closed subscheme of $\GL_{2n,\Z}^+$ and thus yield an $\Fun$-model $\Sp_{2n}$. The intersection of $\Sp_{2n,\Z}^+$ with the diagonal torus of $\GL_{2n,\Z}^+$ is a maximal torus of $\Sp_{2n,\Z}^+$, and its normalizer is contained in the group of monomial matrices of $\GL_{2n,\Z}^+$. Thus Theorem \ref{thm: tits-weyl models of subgroups} applies and shows that $\Sp_{2n}$ is a closed submonoid of $\GL_{2n}$ and a Tits-Weyl model of $\Sp_{2n,\Z}^+$.

\subsubsection{Special orthogonal groups} 
 \label{subsubsection: orthogonal groups}


It requires some more thought to define a model of (special) orthogonal groups over the integers. A standard way to do it is the following (cf.\ \cite[Appendix B]{Conrad11} for more details). We first define integral models of the orthogonal groups $\OO_n$. Define for each ring $R$ the quadratic form
\begin{align*}
 q_n(x) \ &= \ \sum_{i=1}^m x_ix_{n+1-i}                     &&\text{for }n=2m,\text{ and}\\
 q_n(x) \ &= \ x_{m+1}^2 \ + \  \sum_{i=1}^m x_ix_{n+1-i}    &&\text{for }n=2m+1
\end{align*}
where $x=(x_1,\dotsc,x_n)\in R^n$. The functor
\[
 \underline{\OO_n}(R) \quad = \quad \{ \ g\in\GL_n(R) \ | \ q_n(gx)=x\text{ for all }x\in R^n \ \}
\]
is representable by a scheme $\OO_{n,\Z}^+$. It is smooth in case $n$ is even, but not for odd $n$, in which case only the base extension $\OO_{n,\Z[1/2]}^+$ to $\Z[1/2]$ is smooth (cf.\ \cite[Thm.\ B.1.8]{Conrad11}). For odd $n$, we define the \emph{special orthogonal group $\SO_{n,\Z}^+$} as the kernel of the determinant $\det:\OO_{n,\Z}^+\to\SG_{m,\Z}$. For even $n$, we define \emph{special orthogonal group $\SO_{n,\Z}^+$} as the kernel of the Dickson invariant $D_q:\OO_{n,\Z}^+\to(\Z/2\Z)_\Z$. Then the scheme $\SO_{n,\Z}^+$ is smooth for all $n\geq 1$ (cf.\ \cite[Thm.\ B.1.8]{Conrad11}). 

We describe a maximal torus and its normalizer of these groups and show that we can apply Theorem \ref{thm: tits-weyl models of subgroups} in the following.

\subsubsection*{The Tits-Weyl model of $\SO_n$ for odd $n$}

We consider the case of odd $n=2m+1$ first. A maximal torus $T(R)$ of $SO_n(R)$ is given by the diagonal matrices with values $\lambda_1,\dotsc,\lambda_n\in R$ on the diagonal that satisfy $\lambda_{n+1-i}=\lambda_i^{-1}$ for all $i\in\un$ and $\prod_{i=1}^n\lambda_i=1$. This means that $\lambda_1,\dotsc,\lambda_m$ can be chosen independently from $R^\times$ and that $\lambda_{m+1}=1$. 

Its normalizer $N(R)$ consists of all monomial matrices $A=(a_{i,j})$ that satisfy the following conditions. Let $\sigma\in S_n$ be the permutation such that $a_{i,j}\neq0$ if and only if $j=\sigma(i)$. Then $A\in N(R)$ if and only if $\det A=1$, if $\sigma(n+1-i)=n+1-\sigma(i)$ and if $a_{n+1-i,\sigma(n+1-i)}=a_{i,\sigma(i)}^{-1}$ for all $i\in \un$. This means, in particular, that $\sigma(m+1)=m+1$ and that $a_{m+1,m+1}=\sign\sigma$. The permutation $\sigma$ permutes the set of pairs $\Lambda_1=\{\lambda_1,\lambda_n\},\dotsc,\Lambda_m=\{\lambda_m,\lambda_{m+2}\}$ and permutes each pair $\Lambda_i$ for $i\in\un$. This means that the quotient $W=N(R)/T(R)$, which corresponds to all permutations $\sigma\in S_n$ that occur, is isomorphic to a signed permutation group. A set of generators is determined by the involutions $s_1=(1,2)(n-1,n),\dotsc,s_{m-1}=(m-1,m)(m+2,m+3)$ and $s_m=(m,m+2)$. The Weyl group $W$ together with the generators $s_1,\dotsc,s_m$ is a Coxeter group of type $B_m$.

Since $\SO_{n,\Z}^+\subset\GL_{n,\Z}^+$ is smooth, its maximal torus $T$ is contained in the diagonal matrices of $\GL_{n,\Z}^+$ and the normalizer $N$ of $T$ is contained in the monomial matrices of $\GL_{n,\Z}^+$, we can apply Theorem \ref{thm: tits-weyl models of subgroups} to yield a Tits-Weyl model $\SO_n$ of $\SO_{n,\Z}^+$.

\subsubsection*{The Tits-Weyl model of $\SO_n$ for even $n$}

We consider the case of even $n=2m$. Since $\OO_{n,\Z}^+$ is smooth, we can ask whether $\OO_{n,\Z}^+$ has a Tits-Weyl model $\OO_n$. This is indeed the case as we will show now. A maximal torus $T(R)$ of $\OO_n(R)$ is given by the diagonal matrices  with values $\lambda_1,\dotsc,\lambda_n\in R$ on the diagonal that satisfy $\lambda_{n+1-i}=\lambda_i^{-1}$ for all $i\in\un$. Its normalizer $N(R)$ in $\OO_n(R)$ consists of all monomial matrices $A=(a_{i,j})$ that satisfy the following conditions. Let $\sigma\in S_n$ be the permutation such that $a_{i,j}\neq0$ if and only if $j=\sigma(i)$. Then $A\in N(R)$ if and only if $\sigma(n+1-i)=n+1-\sigma(i)$ and $a_{n+1-i,\sigma(n+1-i)}=a_{i,\sigma(i)}^{-1}$ for all $i\in \un$ . Thus we can apply Theorem \ref{thm: tits-weyl models of subgroups} to obtain a Tits-Weyl model $\OO_n$ of $\OO_{n,\Z}^+$.

Note that the Weyl group of $\OO_n$ is the same as for $\SO_{n+1}$: a permutation $\sigma$ associated with an element of the normalizer $N$ permutes the set of pairs $\Lambda_1=\{\lambda_1,\lambda_n\},\dotsc,\Lambda_m=\{\lambda_m,\lambda_{m+1}\}$ and permutes each pair $\Lambda_i$ for $i\in\un$. This means the quotient $W=N(R)/T(R)$, which corresponds to all permutations $\sigma\in S_n$ that occur, is isomorphic to a signed permutation group. 

The subgroup $\SO_{n,\Z}^+$ has the same maximal torus $T$ as $\OO_{n,\Z}^+$, but its normalizer in $\SO_{n,\Z}^+$ is a proper subgroup $N'$ of the normalizer $N$ of $T$ in $\OO_{n,\Z}^+$. Namely, a matrix $A
\in N(R)$ is contained in $N'(R)$ if and only if the sign of the associated permutation $\sigma$ is $1$. The Weyl group $W=N'(R)/T(R)$ is isomorphic to the subgroup of all elements of sign $1$ of a signed permutation group. A set of generators is determined by the involutions $s_1=(1,2)(n-1,n),\dotsc,s_{m-1}=(m-1,m)(m+1,m+2)$ and $s_m=(m-1,m+1)(m,m+2)$. The Weyl group $W$ together with $s_1,\dotsc,s_m$ is a Coxeter group of type $D_m$. This shows that we can apply Theorem \ref{thm: tits-weyl models of subgroups} to yield a Tits-Weyl model $\SO_n$ of $\SO_{n,\Z}^+$.

\bigskip
\noindent We summarize the above results.

\begin{thm}
 All of the Chevalley groups in the following list have a Tits-Weyl model: the special linear groups $\SL_{n,\Z}^+$, the general linear groups $\GL_{n,\Z}^+$, the  adjoint Chevalley groups of type $A_n$, the symplectic groups $\Sp_{2n,\Z}^+$, the special orthogonal groups $\SO_{n,\Z}^+$ (for all $n$) and the  orthogonal groups $O_{n,\Z}^+$ (for even $n$). \qed
\end{thm}


\subsection{Adjoint Chevalley groups}
\label{subsection: adjoint chevalley groups}

In this section, we establish Tits-Weyl models of adjoint Chevalley schemes, which come from the action of the Chevalley group on its Lie algebra $\fg$. Namely, the choice of a Chevalley basis allows us to identify the automorphism group of $\fg$ (as a linear space) with $\GL_{n,\Z}^+$ and consider $\cG$ as a subgroup of $\GL_{n,\Z}^+$. The intersection of the diagonal torus of $\GL_{n,\Z}^+$ with $\cG$ is a maximal torus $T$ of $\cG$. However, the normalizer of $T$ is not contained in the subgroup of monomial matrices of $\GL_{n,\Z}^+$ (unless each simple factor of $G$ is of type $A_1$), thus Theorem \ref{thm: tits-weyl models of subgroups} does not apply to this situation. Moreover, we will see that the rank space of the $\Fun$-model $G$ associated with the embedding of $\cG$ into $\GL_{n,\Z}^+$ does not lift (unless each simple factor of $\cG$ is type $A_n$, $D_n$ or $E_n$). Though the situation is more difficult, the formalism of Tits monoids applies to adjoint representations and we will see that there is a monoid law $\mu=(\mu^\rk,\mu^+)$ for $G$ that turns $G$ into a Tits-Weyl model of $\cG$.

\subsubsection*{Chevalley bases and the adjoint action}

Let $\cG$ be a adjoint Chevalley group scheme, i.e.\ a split semisimple group scheme with trivial center, and let $\fg$ be its Lie algebra. Then its adjoint representation $\cG\to\Aut(\fg)$ is a closed embedding as a subgroup (cf.\ \cite[XVI, 1.5(a)]{SGA3II} or \cite[Thm.\ 5.3.5]{Conrad11}). The choice of a Cartan subalgebra $\fh$ of $\fg$ yields a root system $\Phi$. The choice of fundamental roots $\Pi\subset\Phi$ identifies $\Phi$ with the coroots $\Phi^\vee$, which can be seen as a subset $\{h_r| r\in \Phi\}$ of $\fh$, and decomposes $\Phi$ into the positive roots $\Phi^+$ and the negative roots $\Phi^-$. We denote by $\Psi$ the disjoint union of $\Phi$ and $\Pi$. Let $\{e_r\}_{r\in\Psi}$ be the Chevalley basis given by the choices of $\fh$ and $\Pi$. If $r\in\Pi\subset\Psi$, then $e_r$ is the coroot $h_r$. If $r\in\Phi\subset\Psi$, then we write $l_r$ for $e_r$ to avoid confusion with the coroot $h_r$. This leads to the decomposition
\[
 \fg \quad = \quad \bigoplus_{r\in\Pi} \fh_r \quad \oplus \quad \bigoplus_{r\in\Phi} \fl_r
\]
of $\fg$ into $\fh$-invariant $1$-dimensional subspaces of $\fg$ where $\fh_r$ is generated by $h_r$ for $r\in\Pi$ and $\fl_r$ is generated by $l_r$ for $r\in\Phi$. 

The choice of the Cartan subalgebra $\fh$ corresponds to the choice of a maximal torus $T$ of $\cG$. Let $N$ be its normalizer and $\cW=N/T$ its Weyl group. Since $\cG$ is split, the ordinary Weyl group $W$ is isomorphic to $\cW(\Z)=N(\Z)/T(\Z)$ and $\cW\simeq (W)_\Z$. If we choose an ordering on the Chevalley basis, we obtain an isomorphism $\GL(\fg)\simeq\GL_{n,\Z}^+$ where $n=\# \Psi$ is the dimension of $\fg$. Thus, we can realize $\cG$ as a closed subgroup of $\GL_{n,\Z}^+$. Independent of the ordering of the Chevalley basis, the maximal torus $T$ of $\cG$ is the intersection of $\cG$ with the diagonal subgroup $\diag(\GL_{n,\Z}^+)$ of $\GL_{n,\Z}^+$. 

The adjoint action of $T$ factors into actions on each $\fh_r$ for $r\in \Pi$ and $\fl_r$ for $r\in \Phi$. The action on $\fh$ is trivial and the adjoint action of $T$ on $\fl_r$ factorizes through a character of $T$. The adjoint action of $N$ factors into an action on $\fh$ and an action on $\fl=\bigoplus_{r\in \Phi}\fl_r$. The action on $\fh$ has kernel $T$, which means that it factors through the Weyl group $\cW$. The action of $N$ on $\fl$ restricts to $nT\times\fl_r\to\fl_{w(r)}$ for each $n\in N(\Z)$ and each coset $w=nT(\Z)\in W$. More precisely, we have
\begin{align*}
 n.h_r &= h_{w(r)}   & {\text{where}}\quad h_{w(r)} &= \sum_{s\in\Pi} \lambda_{s,r}^w h_s \quad {\text{for certain integers}} \quad \lambda_{s,r}^w \quad \text{and}\\
 n.l_r &=\pm l_{w(r)}\\
\end{align*}
for all $r\in\Phi$, $n\in N(\Z)$ and $w=nT(\Z)$ (cf.\ \cite[Prop.\ 6.4.2]{Carter89}).

 \subsubsection*{The $\Fun$-model associated with a Chevalley basis}

Let $G$ be the $\Fun$-model associated with the closed embedding $\cG\hookrightarrow\GL_{n,\Z}^+$. Then $G$ is a closed blue subscheme of $\GL_n$, and every point of $G$ is of the form $\fp_I$ for some $I\subset\un\times\un$ (cf.\ Section \ref{subsubsection: general linear group}). In this notation, the maximal torus $T$ equals $(\fp^e)_\Z^+$ where $e$ is the trivial permutation. If a point of the form $\fp^\sigma$ is contained in $G$, then it is a closed point since it is closed in $\GL_n$. Note that $T$ is diagonal w.r.t.\ $G$, i.e.\ $T$ acts on $(\fP_I)^+_\Z$ from the left and from the right for every point $\fp_I$ of $G$. Therefore the rank of pseudo-Hopf points is at least equal to the rank $r$ of $T$, which is the same as the rank of $\fp^e$. This shows that the rank of $G$ is $r$.

\begin{lemma} \label{lemma: on the points p_i of g}\ 
 \begin{enumerate}
  \item The $\Fun$-scheme $G$ contains the point $\fp_I$ if and only if there is an algebraically closed field $k$ and a matrix $(a_{i,j})$ in $\cG(k)\subset\GL_n(k)$ such that $a_{i,j}=0$ if and only if $(i,j)\in I$.
  \item The rank of a point $\fp_I$ of $G$ equals $r$ if and only if there is a matrix $(n_{i,j})\in N(\C)$ such that $n_{i,j}=0$ if and only if $(i,j)\in I$.
\end{enumerate}
\end{lemma} 

\begin{proof}
 We show \eqref{part1}. Since $G$ is cancellative, the morphism $\beta:G_\Z^+\to G$ is surjective (cf.\ Lemma \ref{lemma-beta-surjective-for-cancellative-scheme}). Thus there exists for every point $\fp_I$ of $G$ an algebraically closed field $k$ such that $\fp_I$ lies in the image of $\beta_k:G_k^+\to G$. Since $\fp_I$ is locally closed in $G$, the inverse image $\fp_{I,k}^+$ under $\beta_k$ is locally closed in $G_k^+$, which means that $\fp_{I,k}^+$ contains a closed point of $G_k^+$. By Hilbert's Nullstellensatz, such a closed point corresponds to a $k$-rational point $a:k[G]^+\to k$, which is characterized by $a_{i,j}=a(T_{i,j})$ since $k[G]^+$ is a quotient of $k[\GL_n]^+$ and therefore generated by the $T_{i,j}$ as a $k$-algebra. This defines the sought matrix $(a_{i,j})\in G(k)$ of claim \eqref{part1}.

 Part \eqref{part2} is proven by the very same argument that we used already in the proof of Theorem \ref{thm: tits-weyl models of subgroups}. Namely, we can consider the double action of $T(\C)$ on $\fp_I(\C)$ from the left and from the right. Then the rank of $\fp_I$, which equals the complex dimension of $\fp_I(\C)$, is equal to the rank $r$ of $T(\C)$ if and only if $\fp_I(\C)$ is contained in the normalizer $N(\C)$ of $T(\C)$, i.e.\ $\fp_I(\C)=nT(\C)$ for some $n=(n_{i,j})\in N(\C)$, as claimed in \eqref{part2}.
\end{proof}

\subsubsection*{The rank space}

Let $n\in N(\Z)$ and $w=nT(\Z)$. For every root $r\in\Phi$, we can write the coroot $h_{w(r)}$ as an integral linear combination $h_{w(r)} = \sum_{s\in\Pi} \lambda_{s,r}^w h_s$ of fundamental coroots $h_s$. We write $\fp^w=\fp_{I(w)}$ where $I(w)\subset\Psi\times\Psi$ is defined as the set
\[
 I(w) \quad = \quad \Phi\times \Pi \ \ \cup \ \ \Pi\times\Phi \ \ \cup \ \ \{(r,s)\in\Phi\times\Phi|s\neq w(r)\} \ \ \cup \ \ \{(r,s)\in \Pi\times\Pi|\lambda_{s,r}^w=0\},
\]
which is the set of all $(i,j)$ such that $n_{i,j}=0$ if we regard $n$ as a matrix of $\GL_{n,\Z}^+$ (cf.\ the above formulas in for $N$ acting on $\fg$). Note that in general, $\fp^w$ is not a closed point since $\norm{\lambda_{s,r}^w}$ might be larger than $1$, which means that $\fp^w$ specializes to a point whose potential characteristics are those primes that divide $\norm{\lambda_{s,r}^w}$ (this situation occurs indeed for simple groups of types $B_n$, $C_n$, $F_4$ and $G_2$). 

\begin{prop}\label{prop: pseudo-hopf points of adjoint chevalley groups}
 With the notation as above, we have
 \[
  \cZ(G) \quad = \quad \{ \ \fp^w \ | \ w\in W \ \} \qquad\text{and}\qquad \Gamma \prerk{\fp^w}^\px\quad \simeq \quad \F_{1^\epsilon}[T_{r,w(r)}^{\pm1}|{r\in\Pi}]
 \]
 where $\epsilon=1$ if $w=e$ is the neutral element of $W$ and $\epsilon=2$ otherwise. In particular, $N=G^{\rk,+}_\Z$.
\end{prop}

\begin{proof}
 To start with, we will show that the points $p^w$ are pseudo-Hopf. The canonical torus $T$ equals the intersection of $\cG$ with the diagonal torus $\diag(\GL_{n,\Z}^+)$ inside $\GL_{n,\Z}^+$ since $\diag(\GL_{n,\Z}^+)$ normalizes $T$ and the normalizer $N$ of $T$ in $\cG$ is of the form $N=\bigcup_{n\in N(\Z)} nT$ where $n$ is not of diagonal form unless $n\in T(\Z)$. If $n\in\overline{\fp^w}(\Z)$, then $n^{-1}\overline{p^w}\subset\bigl(\cG\cap\diag(\GL_{n,\Z}^+)\bigr)=T$. Thus $\overline{p^w}^+_\Z=nT$, which shows that $\overline{p^w}^+_\Z$ is a flat scheme.

 By definition, we have
 \[
  \Gamma\prerk{p^w} \quad = \quad \Bigl(\bigl(\bpgenquot{\Fun[G]}{T_{r,s}\=0|(r,s)\in I(w)}\bigr)^\red\Bigr)^\hatexp
 \]
 which is a subblueprint of the coordinate ring $\Gamma nT$. In particular, the equations $T_{r,s}=\lambda_{r,s}^w$ hold in $\Gamma nT$. In $\Gamma\prerk{p^w}$, we have to read $T_{r,s}=\lambda_{r,s}^w$ as
 \[
  T_{r,s} \quad \= \quad \underbrace{1+\dotsb+1}_{\lambda_{r,s}^w\text{-times}} \qquad \text{or} \qquad T_{r,s} + \underbrace{1+\dotsb+1}_{(-\lambda_{r,s}^w)\text{-times}} \quad \= \quad 0,
 \]
 depending on whether $\lambda_{r,s}^w$ is positive or negative. 

 Since $\Gamma\prerk{p^w}$ is cancellative, we can perform calculations in the ring $\Gamma nT$ to obtain information about $\Gamma\prerk{p^w}$. Let $\det(T_{r,s})$ be the determinant of the $T_{r,s}$ with $r,s\in \Psi$. Then the defining equation of $\GL_{n\Z}^+$ is $d\cdot\det(T_{r,s})=1$ where $d$ is the variable for the inverse of the determinant. If we substitute $T_{r,s}$ by $0$ if $(r,s)\in I(w)$ and by $\lambda_{r,s}^+$ if $(r,s)\in\Pi\times\Pi$, then the determinant condition reads as
 \[
  d\cdot \prod_{r\in\Phi} T_{r,w(r)}\cdot\det(\lambda_{r,s}^w)_{r,s\in\Pi} \quad = \quad \pm 1.
 \]
 Since the submatrix $(\lambda_{r,s}^w)_{r,s\in\Pi}$ describes the action of $w$ on the root system $\Phi$ in the basis $\Pi$, the determinant of $(\lambda_{r,s}^w)_{r,s\in\Pi}$ equals the sign of $w$. Thus we end up with an equation $d\cdot \prod_{r\in\Phi} T_{r,w(r)}=\pm 1$ in $\Gamma nT$, which implies an additive relation in $\Gamma\prerk{p^w}_\inv$. This means that $T_{r,w(r)}$ is a unit in $\Gamma\prerk{p^w}$ for all $r\in \Phi$. Therefore, $\Gamma\prerk{p^w}_\inv$ is generated by its units.

 Since $\{p^w\}(\Q)=nT(\Q)$, the only point in $\overline{p^w}$ with potential characteristic $0$ is $p^w$ itself. Since $G$ and $\overline{p^w}$ are of finite type over $\Fun$, every other point in $\overline{p^w}$ has only finitely many potential characteristics. Since $\overline{p^w}$ is a finite space and $\overline{p^w}^+_\Z$ is a flat scheme, $p^w$ must be almost of indefinite characteristic. This completes the proof that $p^w$ is pseudo-Hopf.

 Clearly, 
 \[
  N \quad = \quad \bigcup_{n\in N(\Z)} nT  \quad =  \quad \coprod_{w\in W}\prerk{\fp^w}^+_\Z  \quad =  \quad \coprod_{w\in W}\bigl(\prerk{\fp^w}^\px\bigr)^+_\Z \quad = \quad G^{\rk,+}_\Z,
 \]
 thus Lemma \ref{lemma: on the points p_i of g} \eqref{part2} implies that pseudo-Hopf point $\fp_I$ is of rank $r=\rk T$ if and only if $\fp_I=\fp^w$ for some $w\in W$. This shows that $\cZ(G) =\{ \fp^w | w\in W \}$.

 We investigate the unit fields $\Gamma\prerk{p^w}^\px$. Since $\Gamma nT\simeq \Z[T_{r,w(r)}^{\pm1}|{r\in\Pi}]$, the unit fields of $\Gamma\prerk{p^w}$ are of the form $\F_{1^\epsilon}[T_{r,w(r)}^{\pm1}|{r\in\Pi}]$ where $\epsilon$ is either $1$ or $2$. Since the unit matrix of $\GL_n(\Z)$ is contained in $\cG(\Z)$, there is a morphism $\ast_\Fun\to G$ whose image is $\fp^e$. This shows that $\Gamma \prerk{\fp^w}^\px\simeq \Fun[T_{r,w(r)}^{\pm1}|{r\in\Pi}]$. If $w$ is not the neutral element of $W$, then there is a matrix $n\in N(\Z)$ such that $w=w_n$ operates non-trivially on the coroots $\Phi^\vee\subset\fh$. This means that at least one of the fundamental coroots $h_r$ is mapped to a negative coroot $h_{w(r)}$, i.e.\ $\lambda_{r,s}^w<0$ for all $s\in\Pi$. Therefore $\Gamma\prerk{\fp^w}$ is with $-1$, and the unit field of $\prerk{\fp^w}$ equals $\Funsq[T_{r,w(r)}^{\pm1}|{r\in\Pi}]$. This finishes the proof of the proposition.
\end{proof}

\subsubsection*{The Tits-Weyl model}

Finally, we are prepared to prove that adjoint Chevalley groups have Tits-Weyl models. More precisely, we formulate the following result.

\begin{thm}
 Let $G$ be the $\Fun$-model of $\cG$ as described above. 
 \begin{enumerate}
  \item The group law $\mu_\cG$ descends uniquely to a monoid law $\mu^+$ of $G^+$.
  \item The restriction of $\mu_\cG$ to $N$ descends uniquely to a group law $\mu^\rk$ of $G^\rk$.
  \item The pair $\mu=(\mu^\rk,\mu^+)$ is a Tits morphism, which turns $G$ into a Tits-Weyl model of $\cG$.
  \item The canonical torus is diagonal.
  \item The group $G^\sT(\Fun)$ is the trivial subgroup of $W=\cW(G)$. 
 \end{enumerate}
\end{thm}

\begin{proof}
 We prove \eqref{part1}. The existence and uniqueness of $\mu^+:G^+\Stimes G^+\to G^+$ with $\mu^+_\Z=\mu_\cG$ follows from an application of the cube lemma (Lemma \ref{lemma: cube lemma for blue schemes}) to the commutative diagram
 \[
  \xymatrix@C=4pc@R=1pc{                               & \GL_{n,\Z}^{+}\Stimes_\Z \GL_{n,\Z}^{+}  \ar[rr]^(.5){\mu_{\GL_n,\Z}^+}\ar@{->>}[dd]_(.7){}|\hole &       & \GL_{n,\Z}^{+}\ar@{->>}[dd]^{}  \\
                        \cG\Stimes_\Z \cG \arincl{[ru]^{}}\ar[rr]^(.7){\mu_\cG}\ar@{->>}[dd]_{} &                                  & \cG\arincl{[ru]_{}}\ar@{->>}[dd]^(.3){}     \\
                                                        &  \GL_n^+\Stimes_\N\GL_n^+\ar[rr]^(.4){\mu^+_{\GL_n}}|(.565)\hole                   &                         &  \GL_n^+\,.               \\
                        G^+\Stimes_\N G^+   \arincl[ru]^{}      &                                                     & G^+ \arincl[ru]_{}         } 
 \]
 The identity of $G^+$ is the unique morphism that completes the diagram
 \[
  \xymatrix@C=4pc@R=1pc{                               & \ast_\Z  \ar[rr]^(.5){\epsilon_{\GL_n,\Z}^+}\ar@{->>}[dd]_(.7){}|\hole &       & \GL_{n,\Z}^{+}\ar@{->>}[dd]^{}  \\
                        \ast_\Z \arincl{[ru]^{\id}}\ar[rr]^(.7){\epsilon_\cG}\ar@{->>}[dd]_{} &                                  & \cG\arincl{[ru]_{}}\ar@{->>}[dd]^(.3){}     \\
                                                        &  \ast_\N\ar[rr]^(.35){\epsilon^+_{\GL_n}}|(.47)\hole                   &                         &  \GL_n^+               \\
                        \ast_\N\arincl[ru]^{\id}      &                                                     & G^+ \arincl[ru]_{}         } 
 \]
 to a commuting cube, which exists by the cube lemma. This proves \eqref{part1}.

 We continue with \eqref{part2}. Since $N$ does not embed into the subgroup of monomial matrices of $\GL_{n,\Z}^+$, we have to use a different argument as in the proof of Theorem \ref{thm: tits-weyl models of subgroups}. The group law $\mu_N$ of $N$ can be restricted to morphisms between the connected components $\mu_{w_1,w_2}:n_1T\times n_2T \to (n_1n_2)T$ where $n_1$ and $n_2$ range through $N(\Z)$ and $w_1$ and $w_2$ are the corresponding elements of the Weyl group. Since $nT$ is isomorphic to the spectrum of $\Z[T_{r,w(r)}^{\pm1}]_{r\in\Pi}$ where $w=w_n$, the morphisms $\mu_{w_1,w_2}$ yield ring homomorphisms
 \[
  \begin{array}{cccc}
   \Gamma\mu_{w_1,w_2}:  &  \Z[T_{r,w_{12}(r)}^{\pm1}]_{r\in\Pi} & \longrightarrow &  \Z[(T'_{r,w_{1}(r)})^{\pm1},(T''_{r,w_{2}(r)})^{\pm1}]_{r,s\in\Pi},       \\
  \end{array}
 \]
 where $w_{12}=w_1w_2$. These ring homomorphisms restrict to blueprint morphisms
 \[
  \begin{array}{cccc}
   \Gamma\mu_{w_1,w_2}^\px:  &  \Funsq[T_{r,w_{12}(r)}^{\pm1}]_{r\in\Pi} & \longrightarrow &  \Funsq[(T'_{r,w_{1}(r)})^{\pm1},(T''_{r,w_{2}(r)})^{\pm1}]_{r,s\in\Pi},       \\
  \end{array}
 \]
 between the unit fields. Since the identity matrix is the neutral element of $N(\Z)$, the morphism $\Gamma\mu_{e,e}^\px$ must be compatible with the maps $T_{r,r}\mapsto 1$. This means that $\Gamma\mu_{e,e}^\px$ is the base extension of a morphism
 \[
  \begin{array}{cccc}
  \Gamma \mu_{e,e}^\rk: & \Gamma \prerk{p^e}^\px & \longrightarrow & \Gamma\bigl( \prerk{p^e}^\px \times \prerk{p^e}^\px\bigr)
  \end{array}
 \]
 from $\Fun$ to $\Funsq$ (note that by Proposition \ref{prop: pseudo-hopf points of adjoint chevalley groups}, $\prerk{p^e}^\px$ is without $-1$).

 If one of $n_1$ and $n_2$ differs from $e$, then it follows from Proposition \ref{prop: pseudo-hopf points of adjoint chevalley groups} that $\prerk{\fp^{w_1}}^\px\times_\Fun \prerk{\fp^{w_2}}^\px$ is with $-1$. This means that $\Gamma\bigl(\prerk{\fp^{w_1}}^\px\times_\Fun \prerk{\fp^{w_2}}^\px\bigr) \simeq \Funsq[(T'_{r,w_{1}(r)})^{\pm1},(T''_{r,w_{2}(r)})^{\pm1}]_{r,s\in\Pi}$. Therefore we can define the restriction of $\mu^\rk$ to $\prerk{\fp^{w_1}}^\px\times_\Fun \prerk{\fp^{w_2}}^\px$ by
 \[
  \Gamma \mu^\rk_{w_1,w_2}: \Gamma\prerk{\fp^{w_{12}}}^\px \quad \subset \quad \Funsq[T_{r,w_{12}(r)}^{\pm1}]_{r\in\Pi} \quad \stackrel{\Gamma\mu_{w_1,w_2}^\px}{\longrightarrow} \quad  \Gamma\bigl(\prerk{\fp^{w_1}}^\px\times_\Fun \prerk{\fp^{w_2}}^\px\bigr).
 \]
 This defines a morphism $\mu:G^\rk\times G^\rk\to G^\rk$ that base extends to the group law $\mu_N$ of $N$. The uniqueness of $\mu^\rk$ is clear. The associativity of $\mu^\rk$ follows easily from the associativity of $\mu_N$. 
 Since the Weyl kernel $\fe=\prerk{\fp^e}^\px$ is without $-1$, the identity $\ast_\Z\to N$ of $\mu_N$ descends to an identity $\ast_\Fun\to \fe\subset G^\rk$ of $\mu^\rk$. Similar arguments as above show that the inversion of $\mu_N$ restricts to an inversion $\iota^\rk$ of $\mu^\rk$. Thus $\mu^\rk$ is a group law for $G^\rk$.

 We proceed with \eqref{part3}. Since $\mu_N=\mu_\Z^{\rk,+}$ is the restriction of $\mu_\cG=\mu^+_\Z$, the pair $\mu=(\mu^\rk,\mu^+)$ is Tits. Since $\fe^+_\Z=T$, the canonical torus is a maximal torus of $\cG$ and the morphism $\Psi: G^{\rk,+}/\fe^+_\Z\to N/T$ is an isomorphism of group schemes. Thus $(G,\mu)$ is a Tits-Weyl model of $\cG$.

 Part \eqref{part4} is clear. Part \eqref{part5} follows from the description of the unit fields of $\prerk{\fp^w}^\px$ in Proposition \ref{prop: pseudo-hopf points of adjoint chevalley groups}.
\end{proof}


\section{Tits-Weyl models of subgroups} 
 \label{section: tits-weyl models of subgroups}

In this part of the paper, we establish Tits-Weyl models of subgroups of Chevalley groups, i.e.\ split reductive group schemes. Namely, we will investigate parabolic subgroups, their unipotent radicals and their Levi subgroups.

As a preliminary observation, consider a group scheme $\cG$ of finite type and a torus $T$ in $\cG$. Let $C$ be the centralizer of $T$ in $\cG$ and $N$ the normalizer of $T$ in $\cG$. If $\cH$ is a subgroup of $\cG$ that contains $T$, then the centralizer $\tilde C$ of $T$ in $\cH$ equals the intersection of $C$ and $\cH$, and the normalizer $\tilde N$ of $T$ in $\cH$ equals the intersection of $N$ and $\cH$. 

This means that if we are in the situation of Theorem \ref{thm: tits-weyl models of subgroups}, i.e.\ if $\cG$ is an affine smooth group scheme of finite type with Tits Weyl model $G$ such that $G^{\rk,+}_\Z$ is the normalizer of the canonical torus $T$, then a smooth subgroup $\cH$ of $\cG$ that contains $T$ satisfies automatically all other hypotheses of Theorem \ref{thm: tits-weyl models of subgroups}.

\subsection{Parabolic subgroups} 
 \label{subsection: parabolic subgroups}

Let $\cG$ be a split reductive group scheme. A \emph{parabolic subgroup} of $\cG$ is a smooth affine subgroup $\cP$ of $\cG$ such that for all algebraically closed fields $k$, the algebraic group $\cP_k$ is a parabolic subgroup of $\cG_k$. 

\begin{df}
 Let $G$ be the Tits-Weyl model of a split reductive group scheme $\cG$. A closed submonoid $P$ is a \emph{parabolic submonoid of $G$} if it is the Tits-Weyl model of $P^+_\Z$ where $P^+_\Z$ is a parabolic subgroup of $G$ and if $P^\rk$ contains the Tits-Weyl kernel of $G$.
\end{df}

\begin{thm}\label{thm: parabolic subgroups}
  Let $\cG$ be a reductive group scheme with Tits-Weyl model $G$ and canonical torus $T$. The parabolic submonoids $P$ of $G$ stay in bijection with the parabolic subgroups $\cP$ of $\cG$ that contain $T$.
\end{thm}

\begin{proof}
 Given a parabolic submonoid $P$, then the parabolic subgroup $\cP=P_\Z^+$ contains the canonical torus $T=\fe^+_\Z$ since $P^\rk$ contains the Weyl kernel $\fe$ of $G$. If $\cP$ is a parabolic subgroup of $\cG$ that contains $T$, then it is an immediate consequence of Theorem \ref{thm: tits-weyl models of subgroups} that $\cP$ has a Tits-Weyl model $P$. It is clear that this two associations are inverse to each other.
\end{proof}






\subsection{Unipotent radicals}
 \label{subsection: unipotent radicals}

Let $\cP$ be a parabolic subgroup of a reductive group scheme $\cG$. Then $\cP$ has a \emph{unipotent radical $\cU$}, i.e.\ the smooth closed normal subgroup such that for all algebraically closed fields $\cU_k$ is the unipotent radical of $\cP_k$ (cf.\ \cite[XXII, 5.11.3, 5.11.4]{SGA3III} or \cite[Cor.\ 5.2.5]{Conrad11} for the existence of $\cU$). The group schemes $\cU$ that occur as unipotent radicals of parabolic subgroups of reductive group schemes have the following properties.

As a scheme, $\cU$ is isomorphic to an affine space $\SA^n_\Z$. The only torus contained in $\cU$ is the image $T$ of the identity $\epsilon:\ast_\Z\to\cU$, which is a $0$-dimensional torus. Trivially, $T$ is a maximal torus of $\cU$. The centralizer $C(T)$ and the normalizer $N(T)$ of $T$ in $\cU$ both equal $\cU$. Therefore, the Weyl group $\cW=N(T)/C(T)$ of $\cU$ is the trivial group scheme $\ast_\Z$.

Let $U$ be the $\Fun$-model of the inclusion $\cU\to\cP$ and $P$ the Tits-Weyl model of $\cP$. As a consequence of the cube lemma, $\mu_P^+$ restricts to a monoid law $\mu_U^+$ of $U^+$. Since $T$ is the intersection of the maximal torus of $\cP$ with $\cU$, $U$ contains a point $e$ such that $\prerk e^+_\Z$ is $T$. Thus $e\in \cZ(U)$. If one can show that $U$ does not contain any other pseudo-Hopf point of rank $0$, then $U^\rk=\prerk e^\px\simeq\ast_\Fun$, and the group law $\mu_P^\rk$ of $P^\rk$ restricts to the trivial group law $\mu_U^\rk$ of $U^\rk$. In this case, $U$ together with $(\mu_U^\rk,\mu_U^+)$ is a Tits-Weyl model of $U$.

\begin{df}
 Let $G$ be the Tits-Weyl model of a reductive group scheme and $P$ a parabolic submonoid. A submonoid $U$ of $P$ is the \emph{unipotent radical} of $P$ if $U^+_\Z$ is the unipotent radical of $P^+_\Z$ and if $U$ is a Tits-Weyl model of $U^+_\Z$.
\end{df}

\begin{rem}
 The uniqueness of $U$ is clear: if $\cU$ is the unipotent radical of $\cP=P^+_\Z$, then $U$ must be the $\Fun$-model associated to $\cU\to\cP$. It is, however, not clear to me whether unipotent radicals always exist, i.e.\ if always $\cZ(U)=\{e\}$.
\end{rem}

We can prove the existence of unipotent radicals in the following special case. We call a parabolic subgroup of $\GL_{n,\Z}^+$ that contains the subgroup of upper triangular matrices a \emph{standard parabolic subgroup of $\GL_{n,\Z}^+$}.

\begin{prop}
 Let $\cP$ be a standard parabolic subgroup of $\GL_{n,\Z}^+$ and $\cU$ its unipotent radical. Let $U\hookrightarrow P$ be the associated $\Fun$-models. Then $U$ is the unipotent radical of $P$.
\end{prop}

\begin{proof}
 Everything is clear from the preceding discussion if we can show that $\cZ(U)$ contains only one point. The unipotent radical of a standard parabolic subgroup is of the form $\cU=\Spec[T_{i,j}]/I$ where $I$ is the ideal generated by the equations $T_{i,j}=0$ for $i>j$, $T_{i,j}=1$ for $i=j$ and $T_{i,j}=0$ for certain pairs $(i,j)$ with $i<j$. Let $I$ be the subset of $\un\times\un$ that contains all pairs $(i,j)$ that did not occur in the previous relations. Then $U=\Spec \Fun[T_{i,j}]_{(i,j)\in I}\simeq\A^{\#I}_\Fun$ as a blue scheme. It is clear that $\cZ(U)$ consists of only one point, which is the maximal ideal $(T_{i,j})_{(i,j)\in I}$ of $\Fun[T_{i,j}]_{(i,j)\in I}$ (cf.\ Example \ref{ex: rank space of a monoidal scheme}).
\end{proof}



\subsection{Levi subgroups}
 \label{subsection: levi subgroups}

Let $\cP$ be a parabolic subgroup of a reductive group scheme $\cG$. Let $\cU$ be the \emph{unipotent radical} of $\cP$. If $\cP$ contains a maximal torus $T$ of $\cG$, then $\cP$ has a \emph{Levi subgroup} $\cM$, i.e.\ a reductive subgroup that is isomorphic to the scheme theoretic quotient $\cP/\cU$ such that for every algebraically closed field $k$, $\cM_k$ is the Levi subgroup of $\cP_k$ (see \cite[Thm.\ 4.1.7, Prop.\ 5.2.3]{Conrad11} or \cite[Lemmas 2.1.5 and 2.1.8]{Conrad-Gabber-Prasad10} for the existence of $\cM$). In particular, $\cM$ contains the maximal torus $T$. 

\begin{df}
 Let $G$ be the Tits-Weyl model of a reductive group scheme $\cG$. Let $P$ be a parabolic subgroup of $G$. A submonoid $M$ of $P$ is called a \emph{Levi submonoid} if $M^+_\Z$ is the Levi subgroup of $P$ and if $M$ is a Tits-Weyl model of $M^+_\Z$.
\end{df}

\begin{thm}
 Let $G$ be the Tits-Weyl model of a reductive group scheme $\cG$. Let $P$ be a parabolic submonoid of $G$. Then $P$ contains a unique Levi submonoid $M$.
\end{thm}

\begin{proof}
 The uniqueness follows from the uniqueness of the Levi subgroup of $P^+_\Z$. The existence follows from the existence of the Levi subgroup of $P^+_\Z$ and Theorem \ref{thm: tits-weyl models of subgroups}.
\end{proof}



\appendix
\section{Examples of Tits-Weyl models}
 \label{app: examples}

\subsection{Non-standard torus}

There are different blue schemes together with a monoid law in $\TSch$ that are Tits-Weyl models of the torus $\G_{m,Z}^r$ of rank $r$. We give one example for $r=1$, i.e.\ a non-standard Tits-Weyl model of the multiplicative group scheme $\G_{m,\Z}$.

Namely, consider the blueprint $B=\bpgenquot{\Fun[S,T^{\pm1}]}{S\=1+1}$. Its universal ring is $B^+_\Z=\Z[T^{\pm1}]$, the coordinate ring of $\G_{m,\Z}$. Thus $G=\Spec B$ is an $\Fun$-model of $\G_{m,\Z}$. The blue scheme $G$ consists of two points: the closed point $x=(S)$, which is of characteristic $2$, and the generic point $\eta=(0)$, which has all potential characteristics except for $2$. The point $\eta$ is the only pseudo-Hopf point of $G$, i.e.\ $\cZ(G)=\{\eta\}$. The rank space of $G$ is $G^\rk\simeq\Spec\Fun[T^{\pm1}]$ and its universal semiring scheme is $G^+\simeq\Spec\N[T^{\pm1}]$.

The group law of $\G_{m,\Z}$ descends to a morphism $\mu: G\times G\to G$. Namely, it is given by the morphism 
\[
 \Gamma\mu: \quad B \quad  \longrightarrow \quad B\otimes_\Fun B \quad = \quad \bpgenquot{\Fun[S_1,S_2,T_1^{\pm1},T_2^{\pm1}]}{S_1\=1+1\=S_2}
\]
between the global sections of $G$ and $G\times G$ that is defined by $\Gamma\mu(S)=S_1$ and $\Gamma\mu(T)=T_1\otimes T_2$. Indeed $G$ becomes a semigroup object in $\rkBSch$  without an identity: there is no morphism $B\to\Fun$ since $\Fun$ contains no element $S'$ that satisfies $S'\=1+1$. 

However, the morphism $\mu$ maps $\cZ(G\times G)$ to $\cZ(G)$, i.e.\ $\mu$ is Tits. In the category $\TSch$, the pair $\bigl(G, \mu)$ is a group. Since the Weyl group of $\G_{m,\Z}$ is the trivial group and $G^\rk$ consists of one point, $G$ is a Tits-Weyl model of $G_{m,\Z}$.

While it is clear that $G$ is not isomorphic to $G_{m,\Fun}=\Spec\Fun[T^{\pm1}]$ in $\BSch$, the locally algebraic morphism $\varphi:G\to\G_{m,\Fun}$ that is defined by the obvious inclusion $\Fun[T^{\pm1}]\hookrightarrow \bpgenquot{\Fun[S,T^{\pm1}]}{S\=1+1}$ is Tits and an isomorphism of groups in $\TSch$. 

More generally, it can be shown that every cancellative Tits model $G$ of $\G_{m,Z}^r$ is the spectrum of a subblueprint of $\Z[T_1^{\pm1},\dotsc,T_r^{\pm1}]$ that contains $\Fun[T_1^{\pm1},\dotsc,T_r^{\pm1}]$, but not $-1$. Moreover, the inclusion $\Fun[T_1^{\pm1},\dotsc,T_r^{\pm1}]\hookrightarrow \Gamma G$ is Tits and defines an isomorphism $G\to\G_{m,\Fun}$ of groups in $\TSch$.

\subsection{Tits-Weyl models of type $A_1$}
\label{app: tits-weyl models of type a_1}

In this section, we calculate explicitly the different Tits-Weyl models that we described in the main text of the paper for groups of type $A_1$. Namely, we reconsider the standard model $\SL_2$ of the special linear group, the Tits-Weyl model of the adjoint group $\cG$ of type $A_1$ given by the conjugation action on $\Mat_{2\times 2}$ and the Tits-Weyl model of $\cG$ given by the adjoint representation.

 \subsubsection*{The standard model of $\SL_2$}

We reconsider the example $\SL_2=\Spec\Fun[\SL_2]$ with $\Fun[\SL_2]=\bpgenquot{\Fun[T_1,\dotsc,T_4]}{T_1T_4\=T_2T_3+1}$ and make the heuristics from the introduction precise. The prime ideals $\fp$ of $\Fun[SL_2]$ are generated by a subset of $\{T_1,\dotsc,T_4\}$ such that not both $T_1T_4$ and $T_2T_3$ are contained in $\fp$. We illustrate $\SL_2$ in Figure \ref{figure: sl_2} where the encircled points are the pseudo Hopf points of minimal rank.
\begin{figure}[h]
 \begin{center}
  \includegraphics{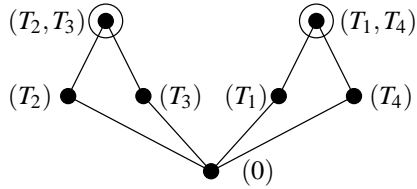}
  \caption{The standard model of $\SL_2$}
  \label{figure: sl_2}
 \end{center} 
\end{figure}

One sees clearly that the maximal ideal $\fp_{2,3}=(T_2,T_3)$ corresponds to the diagonal torus $T=\bigl\{\tinymat \ast 0 0 \ast \bigr\}$ of the matrix group $\SL_2(k)$ (where $k$ is a ring and $\ast$ stays for a non-zero entry) and $\fp_{1,4}=(T_1,T_4)$ corresponds to the subset $\bigl\{\tinymat 0 \ast \ast 0 \bigr\}$ of anti-diagonal matrices. The ideals $\fp_1=(T_1)$, $\fp_2=(T_2)$, $\fp_3=(T_3)$ and $\fp_4=(T_4)$ correspond to the respective subsets $\bigl\{\tinymat 0 \ast \ast \ast \bigr\}$, $\bigl\{\tinymat \ast 0 \ast \ast \bigr\}$, $\bigl\{\tinymat \ast \ast 0 \ast \bigr\}$ and $\bigl\{\tinymat \ast \ast \ast 0 \bigr\}$, while $(0)$ corresponds to the subset $\bigl\{\tinymat \ast \ast \ast \ast \bigr\}$.


 \subsubsection*{The adjoint group of type $A_1$ via conjugation}

We turn to the adjoint group $\cG$ of type $A_1$. Note that $\cG(k)=\PSL_2(k)$ if we consider an algebraically closed field $k$. One can represent $\PSL_2(k)$ by the conjugation action on $2\times 2$-matrices. Consider $\tinymat abcd\in\SL_2(k)$ and $\tinymat efgh\in\Mat_{2\times 2}(k)$. Then the product
\[
 \mat abcd \mat efgh \mat abcd ^{-1} \quad = \quad \mat{ad e-ac f+bd g-bc h}{-ab e+a^2 f-b^2 g+ab h}{cd e-c^2 f+d^2 g-cd h}{-bc e+ac f-bd g+ad h}
\]
shows that $\PSL_2(k)$ acts on the $4$-dimensional affine space, generated by $e$, $f$, $g$ and $h$, via the matrices
\[
   A(a,b,c,d) \qquad = \qquad  \begin{pmatrix}
                                    ad  & -ac  &  bd  & -bc  \\
                                   -ab  &  a^2 & -b^2 &  ab  \\
                                    cd  & -c^2 &  d^2 & -cd  \\
                                   -bc  &  ac  & -bd  &  ad  \\
                               \end{pmatrix} 
\]
with $ad-bc=1$. This is a faithful representation of $\PSL_2(k)$. The algebraic group $\cG_k$ over $k$ that is associated to the group $\PSL_2(k)=\{A(a,b,c,d)|ad-bc=1\}\subset\SA^4_k$ descends to an integral model $\cG\subset\GL_{4,\Z}^+$, which is an adjoint Chevalley group of type $A_1$. Let $G$ be the associated $\Fun$-model. Then the prime ideals of $G$ are generated by subsets of $\{T_{i,j}\}_{i,j=1,\dotsc,4}$ where $T_{i,j}$ is the matrix coefficient at $(i,j)$. Since $ad-bc=1$, one of $ad$ and $bc$ has to be non-zero for $A(a,b,c,d)\in \PSL_2(k)$. We consider the various possible combinations of $a$, $b$, $c$ and $d$ being zero or not (as above, $\ast$ denotes a non-zero entry):
\begin{align*}	
   a=0:\ &   \Biggl(\begin{smallmatrix}  0   &  0   & \ast & \ast \\
                                         0   &  0   & \ast &  0   \\
                                        \ast & \ast & \ast & \ast \\
                                        \ast &  0   & \ast &  0   \\    \end{smallmatrix}\Biggr)
 & b=0:\ &   \Biggl(\begin{smallmatrix} \ast & \ast &  0   &  0   \\
                                         0   & \ast &  0   &  0   \\
                                        \ast & \ast & \ast & \ast \\
                                         0   & \ast &  0   & \ast \\    \end{smallmatrix} \Biggr) 
 & c=0:\ &   \Biggl(\begin{smallmatrix} \ast &  0   & \ast &  0   \\
                                        \ast & \ast & \ast & \ast \\
                                         0   &  0   & \ast &  0   \\
                                         0   &  0   & \ast & \ast \\    \end{smallmatrix}\Biggr) \\
   d=0:\ &   \Biggl(\begin{smallmatrix}  0   & \ast &  0   & \ast \\
                                        \ast & \ast & \ast & \ast \\
                                         0   & \ast &  0   &  0   \\
                                        \ast & \ast &  0   &  0   \\    \end{smallmatrix}\Biggr)  
 & a=d=0:\ & \Biggl(\begin{smallmatrix}   0  &   0  &   0  & \ast \\
                                          0  &   0  & \ast &   0  \\
                                          0  & \ast &   0  &   0  \\
                                        \ast &   0  &   0  &   0  \\    \end{smallmatrix}\Biggr)
 & b=c=0:\ & \Biggl(\begin{smallmatrix} \ast &   0  &   0  &   0  \\
                                          0  & \ast &   0  &   0  \\
                                          0  &   0  & \ast &   0  \\
                                          0  &   0  &   0  & \ast \\    \end{smallmatrix}\Biggr) 
\end{align*}
The case $a,b,c,d\neq0$ corresponds to the matrices $A(a,b,c,d)$ with no vanishing coefficient. The zero entries of each case stay for the generators $T_{i,j}$ of the prime ideals of $G$. Without writing out the generating sets, we see in Figure \ref{figure: psl_2 by conjugation on mat_2x2} that the topological space of $G$ is the same as the topological space of $\SL_2$.
\begin{figure}[h]
 \begin{center}
  \includegraphics{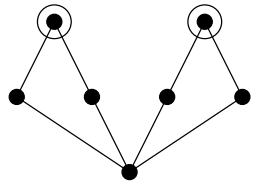}
  \caption{The Tits-Weyl model of type $A_1$ defined by the conjugation action}
  \label{figure: psl_2 by conjugation on mat_2x2}
 \end{center} 
\end{figure}

Since the maximal points (which are encircled in Figure \ref{figure: psl_2 by conjugation on mat_2x2}) correspond to the diagonal and anti-diagonal matrices, respectively, they are the pseudo-Hopf points of minimal rank. The pre-rank space $G^\prk$ is discrete and $G^\rk$ embeds into $G$.

 \subsubsection*{The adjoint group of type $A_1$ via the adjoint action}

Let $G$ be the Tits-Weyl model of the adjoint group $\cG$ of type $A_1$ that is defined by the adjoint action of $\cG$ on its Lie algebra. The roots system of type $A_1$ is $\Phi=\{\pm a\}$ and the set of primitive roots is $\Pi=\{a\}$. Thus a basis of the Lie algebra of $\cG$ is given by the ordered tuple $\Psi=(l_{-a},h_a,l_a)$ where we chose this ordering of $\Psi$ to obtain nice matrix representations below. Since $G$ is defined as a closed subscheme of $\GL_3$, every point $x$ of $G$ is of the form $\fp_I$ for some $I\subset\{1,2,3\}\times\{1,2,3\}$. Since $G$ is cancellative, the morphism $\beta_G:\cG\to G$ is surjective by Lemma \ref{lemma-beta-surjective-for-cancellative-scheme}. Thus a point $\fp_I\in\GL_3$ is contained in $G\subset\GL_3$ if and only if there is an algebraically closed field $k$ and a matrix $(a_{i,j})\in\cG(k)$ such that $a_{i,j}=0$ if and only if $(i,j)\in I$. This reduces the study of the topological space of $G$ to the study of matrices $(a_{i,j})\in\cG(k)$, for which we can use explicit formulas.

There is a surjective group homomorphism $\varphi:\SL_2(k)\to\cG(k)$ (see \cite[Section 6]{Carter89}). We describe the image of certain elements of $\SL_2(k)$ in $\cG\subset\GL_3(k)$ w.r.t.\ the basis $\Psi=(l_{-a},h_a,l_a)$:
\[
 \varphi(\tinymat 1 t 0 1 )                  = \biggl( \begin{smallmatrix} 1 & t & -t^2\ \\ 0 & 1 & -2t\ \\ 0 & 0 & 1 \end{smallmatrix}\biggr), \qquad
 \varphi(\tinymat \lambda {0\ \ \ } 0 {\lambda^{-1}} ) = \biggl( \begin{smallmatrix} \lambda^{-2} & 0 & 0\ \\ 0 \ \ \ & 1 & 0\ \\ 0\ \ \  & 0 & \lambda^2 \end{smallmatrix}\biggr) \quad\text{and}\quad
 \varphi(\tinymat {\ \ 0} {1} {-1}{0} )                 = \biggl( \begin{smallmatrix} \ \ 0 & \ \ 0 & -1\\ \ \ 0 & -1 & \ \ 0 \\ -1 & \ \ 0 & \ \ 0 \end{smallmatrix}\biggr).
\]
For the first equation, see Section 6.2 of \cite{Carter89}, for the second equation Proposition 6.4.1 and for the last equation Propositions 6.4.2 and 6.4.3 of \cite{Carter89}. 

The Bruhat decomposition of $\SL_2(k)$ is $\SL_2(k)=B(k) \amalg BwB(k)$ where $w=\tinymat {\ \ 0} {1} {-1}{0}$ and $B$ is the upper triangular Borel subgroup of $\SL_2$, i.e.\ $B(k)$ is the set of all matrices that can written as a product $\tinymat \lambda {0\ \ \ } 0 {\lambda^{-1}} \tinymat 1 t 0 1 $ with $\lambda\in k^\times$ and $t\in k$. In other words, every element $(a_{i,j})\in\SL_2(k)$ can be written as a product $\tinymat \lambda {0\ \ \ } 0 {\lambda^{-1}} \tinymat 1 t 0 1 $ or as a product $\tinymat 1 s 0 1 \tinymat {\ \ 0} {1} {-1}{0} \tinymat \lambda {0\ \ \ } 0 {\lambda^{-1}} \tinymat 1 t 0 1 $ with $\lambda\in k^\times$ and $s,t\in k$. Since $\varphi:\SL_2(k)\to\cG(k)$ is a surjective group homomorphism, we yield 
\[
 \cG(k) \quad = \quad \biggl\{\biggl( \begin{smallmatrix} \lambda^{-2} &\ \lambda^{-2}t &\ -\lambda^{-2}t^2\\ 0\ \ & 1\ & -2t\\ 0\ \ & 0\ &\ \ \lambda^2\end{smallmatrix} \biggr)\biggr\}_{\substack{\lambda\in k^\times\\t\in k}}
     \quad\amalg\quad \biggl\{\biggl( \begin{smallmatrix} \lambda^{-2}s^2  & \ -s+\lambda^{-2}ts^2 & \ -\lambda^{2}+2st-\lambda^{-2}s^2t^2 \\ 
                                                          2 \lambda^{-2} s & \ -1+2\lambda^{-2}st  & \ 2t - 2\lambda^{-2}st^2 \\
                                                          -\lambda^{-2}    & \ -\lambda^{-2}t      & \ \lambda^{-2} t^2                \end{smallmatrix} \biggr)\biggr\}_{\substack{\lambda\in k^\times\\s,t\in k}}.
\]
To find the points of $G$, we have to investigate for which $\lambda,s,t$ a matrix coefficient of the above matrices vanishes. Concerning the first matrix, we see that the following cases appear:
\begin{align*}
 t=0:                  \quad & \biggl(\begin{smallmatrix} \ast &0    &0    \\ 0 &\ast &0    \\ 0 &0 &\ast \end{smallmatrix} \biggr)&      &\quad \fp^e \\
 t\neq 0,\kar k\neq 2: \quad & \biggl(\begin{smallmatrix} \ast &\ast &\ast \\ 0 &\ast &\ast \\ 0 &0 &\ast \end{smallmatrix} \biggr)& x_1  &= \fp_{\{(2,1),(3,1),(3,2)\}} \\
 t\neq 0,\kar k=2:     \quad & \biggl(\begin{smallmatrix} \ast &\ast &\ast \\ 0 &\ast &0    \\ 0 &0 &\ast \end{smallmatrix} \biggr)& x_1' &= \fp_{\{(2,1),(2,3),(3,1),(3,2)\}} 
\end{align*}
where $\ast$ stays for a non-zero entry and the right hand side column lists the image points in $G\subset\GL_3$ together with the notation used in Figure \ref{figure: psl_2 by adjoint action}. Recall from Section \ref{subsection: special linear group} that $\fp^e=\fp_{I(e)}$ where $e\in S_3$ is the trivial permutation.

Concerning the second matrix, we have to consider more cases. If not both $s$ and $t$ are non-zero, we obtain immediately the following list:
\begin{align*}
 s=t=0:                   \quad & \biggl(\begin{smallmatrix}  0   &  0   & \ast \\  0   & \ast &  0   \\ \ast &  0   &  0   \end{smallmatrix} \biggr)&  &\quad \fp^\sigma \\
 s=0,t\neq 0,\kar k\neq 2:\quad & \biggl(\begin{smallmatrix}  0   &  0   & \ast \\  0   & \ast & \ast \\ \ast & \ast & \ast \end{smallmatrix} \biggr)& x_3  &= \fp_{\{(1,1),(1,2),(2,1)\}} \\
 s=0,t\neq0,\kar k=2:     \quad & \biggl(\begin{smallmatrix}  0   &  0   & \ast \\  0   & \ast &  0   \\ \ast & \ast & \ast \end{smallmatrix} \biggr)& x_3' &= \fp_{\{(1,1),(1,2),(2,1),(2,3)\}} \\
 s\neq0,t= 0,\kar k\neq 2:\quad & \biggl(\begin{smallmatrix} \ast & \ast & \ast \\ \ast & \ast &  0   \\ \ast &  0   &  0   \end{smallmatrix} \biggr)& x_4  &= \fp_{\{(2,3),(3,2),(3,3)\}} \\
 s\neq 0,t\neq0,\kar k=2: \quad & \biggl(\begin{smallmatrix} \ast & \ast & \ast \\  0   & \ast &  0   \\ \ast &  0   &  0   \end{smallmatrix} \biggr)& x_4' &= \fp_{\{(2,1),(2,3),(3,2),(3,3)\}} \\
\end{align*}
To investigate the cases of vanishing matrix coefficients with $s\neq0\neq t$, consider the following cases:
\begin{align*}
 -s+\lambda^{-2}s^2t \ &= \ 0  &&\iff &ts  \ &= \ \lambda^2 \\
 -\lambda^{2}+2st-\lambda^{-2}s^2t^2\ &= \ 0  &&\iff &ts  \ &= \ \lambda^2 \\
 -1+2\lambda^{-2}st  \ &= \ 0  &&\iff &2ts \ &= \ \lambda^2   &&(\text{in this case }\kar k\neq 2)\\
 2t - 2\lambda^{-2}st^2\ &= \ 0  &&\iff &ts  \ &= \ \lambda^2 &&(\text{if }\kar k\neq 2)\\
\end{align*}
This yields the following additional points of $G$ where $s\neq0\neq t$:
\begin{align*}
 st=\lambda^2,\kar k\neq 2:\quad & \biggl(\begin{smallmatrix} \ast &  0   &  0   \\ \ast & \ast &  0   \\ \ast & \ast & \ast \end{smallmatrix} \biggr) & x_2  &= \fp_{\{(1,2),(1,3),(2,3)\}} \\
 st=\lambda^2,\kar k= 2:\quad & \biggl(\begin{smallmatrix} \ast &  0   &  0   \\  0   & \ast &  0   \\ \ast & \ast & \ast \end{smallmatrix} \biggr)    & x_2' &= \fp_{\{(1,2),(1,3),(2,1),(2,3)\}} \\
 2st=\lambda^2,\kar k\neq 2:\quad & \biggl(\begin{smallmatrix} \ast & \ast & \ast \\ \ast &  0   & \ast \\ \ast & \ast & \ast \end{smallmatrix} \biggr)& x_5  &= \fp_{\{(2,2)\}} \\
 st\neq\lambda^2\neq 2st,\kar k\neq 2:\quad & \biggl(\begin{smallmatrix} \ast & \ast & \ast \\ \ast & \ast & \ast \\ \ast & \ast & \ast \end{smallmatrix} \biggr)&\eta &= \fp_{\emptyset} \\
 st\neq\lambda^2\neq 2st,\kar k= 2:\quad & \biggl(\begin{smallmatrix} \ast & \ast & \ast \\  0   & \ast &  0   \\ \ast & \ast & \ast \end{smallmatrix} \biggr)   &\eta'&= \fp_{\{(2,1),(2,3)\}} \\
\end{align*}

We summarize these calculations in Figure \ref{figure: psl_2 by adjoint action}. The circled points are the pseudo-Hopf points of minimal rank.
\begin{figure}[h]
 \begin{center}
  \includegraphics{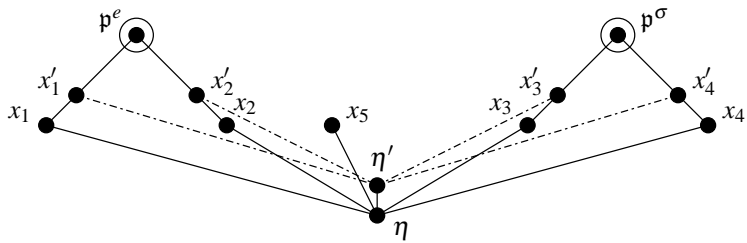}
  \caption{The Tits-Weyl model of type $A_1$ defined by the adjoint action}
  \label{figure: psl_2 by adjoint action}
 \end{center} 
\end{figure}

\begin{rem}
 It is clear that this Tits-Weyl model of $\cG$ differs in $\BSch$ from the Tits-Weyl model that is defined by the conjugation action on $2\times 2$-matrices (cf.\ Figure \ref{figure: psl_2 by conjugation on mat_2x2}). It is, however, not clear to me whether these two models of $\cG$ are isomorphic in $\TSch$ or not.
\end{rem}


\begin{small}
 
\end{small}

\end{document}